\documentclass[a4paper,11pt]{amsart}
\usepackage[margin=1in]{geometry}
\usepackage[utf8]{inputenc}
\usepackage{amsfonts,amssymb,amsmath,amstext,amscd,epsfig}
\usepackage{amsthm,cases,color,graphicx,graphics}
\usepackage[colorlinks=true,pdfborder={0 0 0}]{hyperref}
\usepackage{subcaption}
\usepackage{graphicx}
\hypersetup{urlcolor=blue,citecolor=red,linkcolor=blue}
\numberwithin{equation}{section}
\newtheorem{lemma}{Lemma}[section]
\newtheorem{theorem}[lemma]{Theorem}

\newtheorem{proposition}[lemma]{Proposition}
\newtheorem{remark}{Remark}

\newtheorem{hypothesis}{Hypothesis}
\newcommand{\thishypname}{}
\newtheorem*{generichypothesis}{\thishypname}

\usepackage{tikz}
\usetikzlibrary{patterns}

\newcommand{\eps}{\varepsilon}
\newcommand{\R}{\mathbb{R}}
\newcommand{\C}{\mathbb{C}}

\newcommand{\dd}{\, {\rm d}}

\newcommand{\cM}{\mathcal{M}}

\newcommand{\cP}{\mathcal{P}}

\newcommand{\cL}{\mathcal{L}} % opérateur de collision en v
\newcommand{\wv}{\lfloor v \rceil} % poids notation japonaise en v
\newcommand{\wdot}{\lfloor \cdot \rceil} % poids en dot
\newcommand{\cS}{\mathcal{S}} % l'opérateur du problème spectral
\newcommand{\cE}{\mathcal{E}} 
\newcommand{\cN}{\mathcal{N}}
\newcommand{\pa}{\partial}

\newcommand{\im}{\operatorname{Im}}
\newcommand{\re}{\operatorname{Re}}

\newcommand{\cC}{\mathcal{C}} %coefficients 

\newcommand{\ve}{\varepsilon}

\title{Quantitative fluid approximation in fractional regimes of transport equations with more invariants}

\author{\'Emeric Bouin}
\email{bouin@ceremade.dauphine.fr}
\address{CEREMADE - Universit\'e Paris-Dauphine, PSL Research University, UMR CNRS 7534, Place du Mar\'echal de Lattre de Tassigny, 75775 Paris Cedex 16, France.}

\author{Laura Kanzler}
\email{laura.kanzler@dauphine.psl.eu}
\address{CEREMADE - Universit\'e Paris-Dauphine, PSL Research University, UMR CNRS 7534, Place du Mar\'echal de Lattre de Tassigny, 75775 Paris Cedex 16, France.}

\author{Cl\'ement Mouhot}
\email{c.mouhot@dpmms.cam.ac.uk}
\address{University of Cambridge, Wilberforce Road, Cambridge CB3 0WA, United Kingdom}

\date{\today}

\subjclass[2010]{60J60,35Q84,82C40,35B27,60K50,60G52,76P05}

\keywords{transport process; kinetic theory; anomalous diffusion;
  scattering operator; Fokker-Planck operator; Lévy-Fokker-Planck
  operator; spectral theory}

\begin{document}
\begin{abstract}
We present an extension of results in a previous paper by the first and the last author [PMP, 2022] about macroscopic limits of linear kinetic equations in (potentially) \emph{fractional regimes}. More precisely, we develop a unified framework inspired by Ellis and Pinsky [J. Math. Pures Appl., 1975] for operators that preserve mass, momentum and energy, and have microscopic \emph{equilibrium with heavy tails} (typically polynomial). This paper also generalizes one of Hittmeir and Merino [KRM, 2016] in a related framework. The main difficulty, that leads to our main contribution, is the understanding of the spectrum of the generator in the Fourier space, which is significantly complicated by the lack of spectral gap and the fat tails of the equilibrium. Indeed, the scaling of the eigenelements in the suitable macroscopic rescaling is subtle to handle. In particular, our study uncovered an interesting difference in scaling in the fractional regime, where the transversal wave eigenvalues converge faster to zero than the Boussinesq and acoustic wave eigenvalues.
\end{abstract}

\maketitle

\tableofcontents

%%%%%%%%%%%%%%%%%%%%%%%%%%%%%%%%%%%%%%%%%%%%%%%%%%%%%%%%%%%%%%%%
\section{Introduction and main results}\label{s:intro}

In mathematical physics, the study of \emph{linear transport equations} serves as a powerful tool to describe and gain understanding of physical processes where particles are transported within a medium, while they undergo scattering, random absorption and emission. It's investigation was initiated by pioneering works in the 19th century by Maxwell and his mean-free path argument \cite{maxwell1860process} and by Maxwell and Boltzmann and their kinetic theory of gases \cite{boltzmann1872, maxwell1867iv}. More recently, important applications can be found in \emph{radiative transfer theory} \cite{pomraning1973}, \emph{nuclear reactor theory} \cite{MR0113336}, and \emph{semi-conductor theory} \cite{MR1063852}. The main mathematical object of study in \emph{transport theory} is the linear equation
	\begin{equation} \label{eq:kinetic}
		\partial_t f + v\cdot \nabla_x f = \cL f
	\end{equation}
	on the time dependent density of particles $f=~f(t,x,v) \ge 0$
	over $(x,v) \in \R^3 \times \R^3$, for $t \ge 0$. Although the methods and proofs presented in this article also apply to general dimensions $d>1$, we chose to set $d=3$ throughout the whole article to ease readability. The left-hand-side accounts for free motion and the right-hand-side accounts for the interaction with a background, for instance scatterers, with an operator $\cL$ that acts only on the kinetic variable $v$. The macroscopic moments are defined by 
	\begin{equation*}
		\rho_f = \int_{\R^3} f(v) \, \dd v, \qquad m_f = \int_{\R^3} v f(v) \dd v, \qquad \theta_f := \int_{\R^3}  \frac{|v|^2-3}{3}  f(v) \, \dd v.
	\end{equation*}
	In a regime, where relevant time and space scales of observation are much larger than the mean free time and mean free path of a particle, it is thus natural to search for a simplified regime than given by equation \eqref{eq:kinetic}. After an appropriate rescaling in space and time it is well known that a system of diffusive equations can be obtained in the limit, which was initiated by \emph{Hilbert's 6-th problem} \cite{H}, which he theorised abstractly with the introduction of Hilbertexpansions \cite{H2}, and followed by numerous contributions, see e.g. \cite{S-R, GS-R, BU, BLP, Ellis-Pinsky} and references therein.
	
	In this work, we aim to investigate the limit behaviour of a solution to equation \eqref{eq:kinetic} for vanishing mean free path and carefully chosen time scale. Since we expect the solution to become proportional to the equilibrium $\cM$ in the large space-time scaling, we first rewrite the equation in terms of $h:=\frac{f}{\cM}$:
	\begin{equation}\label{eq:hkinetic}
		\pa_t h + v \cdot \nabla_x h  = Lh, 
	\end{equation}
	with
	\begin{equation}\label{d:L}
		L h (v):= \cM^{-1}(v) \cL (\cM h)(v).
	\end{equation}
	Throughout this article we use the following notations regarding the functional spaces we are working in:
	\begin{itemize}
		\item We consider the complex Hilbert spaces $L_v^2\left(\cM\right):=L^2\left(\R^3;\cM \dd v\right)$ and $L_{x,v}^2\left(\cM\right):=L^2\left(\R^6;\cM \dd x \, \dd v\right)$.
		\item For a $h \in L^2(\cM)$, we denote $\| h \|_k:=\left\|(1+|\cdot|^2)^{\frac{k}{2}}h\right\|_{L^2(\cM)}$, where we precise the integration variables when needed. We omit the index $k$, when $k=0$.
		\item We use the notation $\wv := \sqrt{1+|v|^2}$ for describing the polynomial weights.
		\item For $h_1, h_2 \in L^2(\cM)$, we define the weighted inner product 
			$$
				\langle h_1,\, h_2 \rangle_{k}:= \int_{\R^3} h_1 \overline{h_2} \wv^{k} \cM \dd v\,.
			$$
			When $k=0$ the index will be omitted as well.
		\item Last, we want to remind the definition of the \emph{fractional Laplace} operator using the Fourier Transform $\mathcal{F}$
		 	$$
				\mathcal{F}\left(-\Delta^{\frac{\zeta}{2}}h\right)(\xi) = |\xi|^{\frac{\zeta}{2}}\mathcal{F}(h)(\xi)\,.
			$$	
	\end{itemize}
	We make the following assumptions for some constants $\alpha, \beta \in \R$ fulfilling the conditions
	\begin{align}\label{c:constants}
	\alpha>5\,, \,\, \alpha + \beta >4\,, \,\, \beta >-1\,, \,\, \text{and } \lambda >0\,.
	\end{align}
	
	\begin{hypothesis}[Equilibria]
		\label{hyp:equilibira}
		The microscopic equilibrium $\cM$ of \eqref{eq:kinetic} takes one of the following two forms.
		\begin{itemize}
			\item[(i)] Either it is given by
			\begin{equation}
				\label{eq:Mpoly}
				\cM(v) = c_{\alpha,\beta} \wv^{-(3+\alpha)} \quad \text{ with }
				c_{\alpha,\beta} := \left( \int_{\R^3} \wv^{-3-\alpha-\beta} \dd v
				\right)^{-1}\,.
			\end{equation}
		\end{itemize}
		\begin{itemize}
			\item[(ii)] Or it is a smooth positive radially symmetric
			function decaying faster than any polynomial. The latter case
			is denoted by `$\alpha = +\infty$' in the sequel.
		\end{itemize}
	\end{hypothesis}
	
	We assume the following normalisations
		\begin{equation}\label{eq:Mnorm}
			\int_{\R^3} \wv^{-\beta} \cM \, \dd v = 1, \quad \int_{\R^3} v_i^2  \wv^{-\beta}  \cM \, \dd v = 1, \quad \int_{\R^3} v_i^2 |v|^2  \wv^{-\beta}  \cM \, \dd v = 5, \, \quad i \in \{1,2,3\},
		\end{equation}
	which mean no loss of generality under assumption that $\alpha+\beta>4$. The results in this article are presented assuming that the equilibrium $\cM$ is given by the exact formula \eqref{eq:Mpoly}, which eases notation and calculations. However, we assume that the results remain true for more general equilibria with the signifiant polynomial decay at infinity.
		
	\begin{hypothesis}[Weighted coercivity]
		\label{hyp:coercivity}
		The operator $L$ (defined in \eqref{d:L}) is linear, only acting in the velocity variable $v$ and commuting with rotations in $v$. Moreover, it is closely and densely on $\operatorname{Dom}(L) \subset L_v^2(\cM)$ and satisfies $L\left(w(v)\right) = L^*\left(w(v)\right)=0$ for $w(v) \in \operatorname{span}\left\{1,v,|v|^2\right\}$.  The weighted operator $\tilde{L} (\cdot) := \wdot^{\frac{\beta}{2}} L (\wdot^{\frac{\beta}{2}} \cdot)$ is closed densely defined on $\operatorname{Dom}(\tilde{L}) \subset L_v^2(\cM)$, with the spectral gap estimate
		\begin{equation*}
			- \operatorname{Re} \langle \tilde{L} g,\, g \rangle \geq \lambda \|g\|^2, \quad \text{for all} \quad g \in \operatorname{Dom}(\tilde{L}), \, g \perp \wdot ^{-\frac{\beta}{2}}\operatorname{span}\left\{1,v,|v|^2\right\}.
		\end{equation*}
	\end{hypothesis}
	
	By $\cN(L)$ we denote the null-space of the operator $L$ in $L_{x,v}^2\left(\cM\right)$, which is spanned by $\{1,v,|v|^2\}$. It can be seen easily that due to the normalization conditions \eqref{eq:Mnorm} the projector
	\begin{align}
		\cP F : = \rho_{\wv^{-\beta}, F} + v \cdot m_{\wv^{-\beta}, F} + \frac{|v|^2-3}{2} \theta_{\wv^{-\beta}, F},
	\end{align}
	defines the $L^2\left(\wv^{-\beta} \cM\right)$-orthogonal projection onto $\cN(L)$, where we denote
	\begin{align*}
		 &\rho_{\wv^{-\beta}, F}:= \int_{\R^3} F \wv^{-\beta} \cM \, \dd v, \quad m_{\wv^{-\beta}, F} :=\int_{\R^3} v F \wv^{-\beta} \cM \, \dd v, \\
		 &\theta_{\wv^{-\beta}, F}:= \int_{\R^3} \frac{|v|^2-3}{3} F \wv^{-\beta}  \cM \, \dd v\,.
	\end{align*}
	
	\begin{hypothesis}[Amplitude Estimates]
		\label{hyp:amp}
		Moreover, the amplitude of collisions at large velocities can be controlled in the following way. We denote by $0\leq \chi^1 \leq 1$ a smooth function, such that $\chi^1 \equiv 1$ in $B_1(0)$,  $\chi^1 \equiv 0$ outside $B_2(0)$ and by $0\leq \chi^2 \leq 1$ a smooth function such that $\chi^2 \equiv 0$ in $B_1(0)$ and outside $B_4(0)$,  as well as $\chi^2 \equiv 1$ in $B_3(0)\setminus B_2(0)$. Then the following holds
		\begin{align}\label{amplitudes1}
			&\left\| L(\chi_R^1) \right\|_{L^2(\wv^{\beta} \cM)} \lesssim R^{-\frac{\alpha+\beta}{2}}\,, \\
			&\left\| L(v \chi_R^1) \right\|_{L^2(\wv^{\beta} \cM)} \lesssim R^{1-\frac{\alpha+\beta}{2}}\,, \\
			&\left\| L\left((|v|^2 -3)\chi_R^1\right) \right\|_{L^2(\wv^{\beta} \cM)} \lesssim R^{2-\frac{\alpha+\beta}{2}}\,,
		\end{align}
		where we want to point out, that  $v \in \R^3$ is a velocity vector and the corresponding estimate has to be interpreted component-wise. Moreover, for any function $r_k : \R^3 \to \R^{d}$, such that $|r_k(v)| \leq c_k |v|^k$, for $k \in \mathbb{N}$, $c_k >0$ one can estimate
		\begin{align}\label{amplitudes2}
			&\left\| L(r_k(v)\wv^{\beta}\chi^2_R) \right\|_{L^2(\wv^{\beta} \cM)} \lesssim \begin{cases} &\ln{\left(R\right)}^{\frac{1}{2}}\,, \quad  k = \frac{\alpha-\beta}{2} \\ &R^{k-\frac{\alpha-\beta}{2}}\,, \quad  \text{else}\,,\end{cases}\,,
		\end{align}
		 where we used the notation $\chi_R^{1,2}(\cdot):=\chi^{1,2}\left(\cdot/R\right)$, for a constant $R\geq1$.
	\end{hypothesis}
	
	\begin{hypothesis}[Decomposition]\label{hyp:decomp}
		We assume that $\tilde{L}(\cdot) := \wv^{\frac{\beta}{2}}L\left(\wv^{\frac{\beta}{2}} \cdot \right)$ can be written as 
		\begin{align}
			\tilde{L} = \mathcal{K} - 1\,,
		\end{align}
		where $\mathcal{K}$ is a compact operator in $L^2(\cM)$. 
	\end{hypothesis}
	\begin{remark}
		In general it can be assumed that $\tilde{L} = \mathcal{K} - \bar{\nu} $, where $\bar{\nu} \,: \R^3 \to \R_+$ is a positive, smooth and bounded function, which we simplified to the identity for the ease of notation.
	\end{remark}
	
	Performing now for an $\ve \ll 1$ the rescaling $x \mapsto \ve^{-1} x$, in space and $t \mapsto \gamma(\ve)^{-1} t$ in time as well as
	\begin{align*}
		f_\ve(t,x,v) = f\left(\frac{t}{\gamma(\ve)}, \frac{x}{\ve},v\right), \quad h_\ve(t,x,v) = h\left(\frac{t}{\gamma(\ve)}, \frac{x}{\ve},v\right),
	\end{align*}
	equation \eqref{eq:kinetic} reads as 
	\begin{equation}\label{eq:kineticeps}
			\gamma(\ve) \partial_t f_\ve + v\cdot \nabla_x f_\ve = \cL f_\ve,
	\end{equation}
	with initial conditions to be \emph{well-prepared}, such to fulfil
	\begin{equation}\label{eq:iceps}
		f_{\ve,I}(x,v) = \ve^{-3}f_I\left(\frac{x}{\ve},v\right) \in L_{x,v}^2(\cM^{-1}), \quad \nabla \cdot \int_{\R^3} v f_{\ve,I} \, \dd v = 0,
	\end{equation}
	and such that the fluid limit holds at $t=0$, i.e. $f_{\ve,I}(x,v) \sim v \cdot m_{f_{\ve,I}(x,v)} +\frac{|v|^2-3}{3}\theta_{f_{\ve,I}(x,v)}$. The corresponding rescaled equation in terms of $h$ \eqref{eq:hkinetic} has the following form after rescaling
	\begin{equation}\label{eq:hkineticeps}
		\gamma(\ve)\pa_t h_\ve + \ve v \cdot \nabla_x h_\ve = L h_\ve.
	\end{equation}
	Under these conditions we are able to state our main theorem:
	
	\begin{theorem}[Macroscopic limit]\label{t:main}
		Let Hypothesis \ref{hyp:equilibira}- \ref{hyp:coercivity} - \ref{hyp:amp} be fulfilled and $\alpha + \beta >4$, $\alpha>5$ and $\beta > -1$. And let the initial conditions be well-prepared and fulfil \eqref{eq:iceps}. 
		\begin{itemize}
		\item If $\alpha < 6 + \beta$, then a solution $f_\ve$ to \eqref{eq:kineticeps}-\eqref{eq:iceps} converges towards
		\begin{align*}
			f_\ve \longrightarrow \cM \left(v \cdot m + \frac{|v|^2 - 5}{2} \theta \right), \quad \text{as} \quad \ve \to 0\,, \quad \text{weak in} \quad L^{2}\left([0,T];H^{-\zeta}_xL^2_v(\wv^{-\beta}\cM)\right)\,,
		\end{align*}
		where the macroscopic moments are determined by
		\begin{align*}
			&\rho + \theta = 0, \quad \nabla \cdot m = 0\,, \\ 
			&\pa_t \theta = \kappa_1 \Delta^{\frac{\zeta}{2}} \theta\,, \\
			&\pa_t m = \nabla p\,, \\
			&m(x) = m_I(x), \quad \theta(0,x) = \theta_I(x),
		\end{align*}
		with
		\begin{align*}
			\zeta := \frac{-4+\alpha + \beta}{1+\beta},
		\end{align*} 
		pressure $p$ and diffusion constant $\kappa_1 >0$. 
		\item If $\alpha > 6 + \beta$, then a solution $f_\ve$ to \eqref{eq:kineticeps}-\eqref{eq:iceps} converges towards
		\begin{align*}
			f_\ve \longrightarrow \cM \left( v \cdot m + \frac{|v|^2 - 5}{2} \theta \right), \quad \text{as} \quad \ve \to 0\,, \quad \text{weak in} \quad L^{2}\left([0,T];H^{-2}_xL^2_v(\wv^{-\beta}\cM)\right)
		\end{align*}
		where the macroscopic quantities are determined by
		\begin{align*}
			&\rho + \theta = 0, \quad \nabla \cdot m = 0, \\
			&\pa_t m = \nabla \cdot (A \nabla m) + \nabla p, \\
			&\pa_t \theta = \kappa_2 \Delta \theta, \\
			&m(0,x) = m_I(x), \quad \theta(0,x) = \theta_I(x),
		\end{align*}
		with pressure $p$, diffusion coefficient $\kappa_2>0$ and diffusion matrix $A$.
		\end{itemize}
	\end{theorem}
	
	Hydrodynamic limits of kinetic equations have been matter of interest since Hilbert posed his famous 6-th problem \cite{H} in 1900. In several contributions it was shown that in the regime of a small mean-free path and appropriate rescaling a diffusive system of equations for the macroscopic quantities can be obtained. See therefore early works by Grad \cite{g} and Ellis et al. \cite{Ellis-Pinsky} for a derivation of the Navier-Stokes equation from the linearized Boltzmann equation assuming smooth initial data. While in the first mentioned work the method of Chapman–Enskog-Hilbert expansion [cite something] was used, in the latter the method was based on spectral analysis of the operator in Fourier-space. In further contributions by Golse et al. in \cite{GS-R}, see also Saint-Raymond's book \cite{S-R}, a Navier–Stokes–Fourier system was rigorously obtained from a hydrodynamical limit taking into account the non-linearity of the Boltzmann-equation. In more recent works Gervais et al. \cite{GL} a general concept proving convergence to the Navier–Stokes–Fourier system was developed and is applicable to several equations. Such diffusive approximations, as the ones in the works listed above, require the equilibrium distributions to have finite variance, which is given in the classical case of a Gaußian equilibrium. 
	
	In contrast to this, in several contexts as wave turbulence \cite{N, MMT, BPFS}, economy \cite{DT}, astrophysical plasmas \cite{ST} and bacterial movement \cite{ODA, AS}, equilibrium distributions appear, which have algebraic decay ("heavy tails") and might violate the finite variance condition, that is crucial in the established approaches. For such types of equilibrium distributions the diffusive time scale is too long. For the linear Boltzmann equation, which is just conserving mass, Mellet et al. \cite{M3} showed that under a shorter time-rescaling a fractional diffusion equation can be derived in the macroscopic limit. While their methods rely on the Laplace-Fourier transform, similar results have been achieved from a probabilistic point of view by Jara et al. \cite{MR2588245}. Further contributions can be found in Ben Abdallah et al. \cite{BMM}, where the authors were able to achieve a fractional diffusion limit by an Hilbert expansion approach. Spectral approaches were used recently by Dechicha et al. \cite{DM} for the kinetic Fokker–Planck operator as well as in the work by Bouin et al. \cite{BM}, where the latter contribution describes an unified spectral method for a huge class of linear kinetic equations having one macroscopically conserved quantity.
	
	The case of an equation with an heavy-tailed equilibrium and with several macroscopic invariants was so far just investigated by Hittmeir et. al \cite{MR3422647} in the very specific scenario of a linear BGK-type operator. The authors used the by Mellet in \cite{Mel} introduced moments method. In this work we propose to extend and generalize these results by performing a fine spectral analysis of the generator in the Fourier space, inspired by the early works \cite{Ellis-Pinsky} and the more recent contributions \cite{BM}. 
	
	This article is organized as follows. In Section \ref{s:strategy} the strategy and main steps of the proof are explained by stating the important intermediate results. This specially involves the construction of the "fluid modes" and their corresponding eigenvalues by solving a corresponding eigenvalue-problem related to the equation as well as investigating their behaviour in the limit. Section \ref{s:fluidmodes} is dedicated to the proof of Proposition \ref{p:fluidmodes}, giving existence of the aforementioned fluid modes. Section \ref{s:scaling} is then dedicated to the investigation of the limit behaviour of the eigenvalues, since their rates will be crucial for the final limit argument which is given in Section \ref{s:convergence}.

\section{Strategy of the proof Theorem \ref{t:main}}\label{s:strategy}

The first result presented in this work, on the basis of the four hypotheses stated above, is a quantitative construction of braches of "fluid eigenmodes" with their corresponding eigenvalues in the regime of small scaling parameter. These eigenvalues are branches from the known zero-eigenvalue of $\tilde{L}$ for the lightly perturbed operator $\tilde{L} + i \eta \wv^{-\beta} \left(\sigma \cdot v\right)$ for small $\eta$. It can be proven that there exist exactly five solutions to this eigenvalue problem (in general always two more than the dimension). Our proof is strongly inspired on the work by Ellis and Pinsky \cite{Ellis-Pinsky}, which results in an inversion argument on the finite dimensional subspace spanned by the kernel of $\tilde{L}$, while extending it to the setting where involved integrals might be infinite due to the heavy tail of the equilibrium. 

\begin{proposition}[Construction of the fluid modes]\label{p:fluidmodes}
	
	Given Hypotheses \ref{hyp:equilibira}-\ref{hyp:coercivity}-\ref{hyp:decomp} there are $\bar{\eta} >0$ and $0<\bar{r}<\lambda$ s.t. for any $\eta \in (0,\bar{\eta})$ and any $\sigma \in \mathbb{S}^{2}$ there are exactly 5 solutions \newline $\left\{\left(\mu(\eta)_l,\,\phi_{\eta,l}\right)\right\}_{l \in \pm,0,\dots,3} \in \C \times L_v^2\left(\wv^{-\beta} \cM\right)$ to
	\begin{equation*}
	L^* \phi_{\eta} + i \eta (v \cdot \sigma) \phi_\eta = -\mu(\eta)    \wv^{-\beta} \phi_{\eta}\,,
	\end{equation*}
	where for all $l \in \{\pm,0,1,2\}$ it holds $\mu(\eta)_l \in B(0,\bar{r})$  and 
	\begin{equation*}
	\begin{split}
		&1= \left\|\cP(\phi_{\eta,l}) \right\|_{-\beta}\,, \quad \text{and} \quad 0 = \left \langle \phi_{\eta,l},\, \phi_{\eta,k}  \right \rangle_{-\beta} \,, \quad l \neq k \,.
	\end{split}
	\end{equation*}
	Moreover, we have the following asymptotics 
	\begin{align*}
		\left\| \phi_{\eta,l}-\cP(\phi_{\eta,l})\right\|_{-\beta} \lesssim \sqrt{\re(\mu(\eta)_l)}\,.
	\end{align*}
	\end{proposition}
	In the proof of Proposition \ref{p:fluidmodes} (see Section \ref{s:fluidmodes}) it is further shown that three of the eigenvalues are real, which in hydrodynamic theory are usually referred to the \emph{Boussinesq eigenvalue} $\mu(\eta)_0$ and the two \emph{transversal wave eigenvalues} $\mu(\eta)_t$. In general dimensions one always observes the occurence of the first mentioned Boussinesq eigenvalue $\mu(\eta)_0$, while the number of the transversal wave eigenvalues $\mu(\eta)_t$ is usually given by the dimension  -1. Further, one always observes the two complex \emph{acoustic wave eigenvalues} $\mu(\eta)_\pm$, which are complex conjugates to each other. (See Figure \ref{fig:evsdiff} and Figure \ref{fig:evsfrac}.)
	We would also like to notice that the  fluid modes further depend on the frequency $\sigma$, but this Proposition \ref{p:fluidmodes} and in the rest of the paper the dependency in $\sigma$ is kept implicit rather than explicit in order to lighten notation. On the other hand, the eigenvalues are independent of $\sigma$ since we assume that the operator $L$ is invariant by rotations in $v$. To identify the macroscopic limit it is necessary to estimate the leading order of the eigenvalues $\{\mu(\eta)_l\}$, $l \in \{\pm, 0,1,2\}$, which asks for scaling integrability properties of the corresponding fluid modes, which leads to our last hypothesis.
	
	\begin{hypothesis}[Scaling of the fluid modes]
		\label{hyp:fluidm}
		For the fluid modes constructed Proposition \ref{p:fluidmodes} have the following integrability and scaling conditions:
		\begin{itemize}
			\item If $\alpha>6+\beta$, the fluid modes $\{\phi_{\eta,l}\}_l$, $l \in \{\pm,0\}$ fulfil the following integrability condition
			$$	
				\|\phi_{\eta,l}\|_\delta \lesssim 1\,, \quad \text{for all } \, \delta < \alpha\,.
			$$
			\item If $\alpha \in (4-\beta,6+\beta)$, the rescaled fluid mode $\Phi_{\eta,l}(\cdot):=\phi_{\eta,l}\left(\eta^{-\frac{1}{1+\beta}} \cdot \right)$, $l \in \{\pm,0\}$ is pointwise converging to a limit $\Phi_l$ in $L_{loc}^2\left(\R^d\setminus0\right)$ as $\eta \to 0$. Moreover, we assume the pointwise controls
			\begin{align*}
				\left|\re\left(\Phi_{\eta,l}(u)\right)\right|, \left|\im\left(\Phi_{\eta,l}(u)\right)\right| \lesssim \left(\eta^{\frac{2}{1+\beta}}+|u|^2\right)^{\frac{\delta}{2}}\,,
			\end{align*}
			with $\delta > \alpha -3$\,.
			\item  If $\alpha>4+\beta$, the fluid modes $\{\phi_{\eta,t}\}_l$, $t \in \{1,2\}$ fulfil the following integrability condition
			$$	
				\|\phi_{\eta,l}\|_\delta \lesssim 1\,, \quad \text{for all } \, \delta < \alpha\,.
			$$
			\item If $\alpha \in (4-\beta,4+\beta)$, the rescaled fluid mode $\Phi_{\eta,l}(\cdot):=\phi_{\eta,l}\left(\eta^{-\frac{1}{1+\beta}} \cdot \right)$, $t \in \{1,2\}$ is pointwise converging to a limit $\Phi_t$ in $L_{loc}^2\left(\R^d\setminus0\right)$ as $\eta \to 0$. Moreover, we assume the pointwise controls
			\begin{align*}
				\left|\re\left(\Phi_{\eta,t}(u)\right)\right|, \left|\im\left(\Phi_{\eta,t}(u)\right)\right| \lesssim \left(\eta^{\frac{2}{1+\beta}}+|u|^2\right)^{\frac{\delta}{2}}\,,
			\end{align*}
			with $\delta > \alpha -2$\,.
		\end{itemize}
	\end{hypothesis}
	
	With this Hypothesis, we are able to prove Proposition \ref{p:eigenvalues} (see Section \ref{s:scaling}), which gives the precise scaling of the eigenvalues, which is crucial and quantitive information when passing to  $\eta \to 0$ in our equation.
	
	\begin{proposition}[Scaling of the eigenvalues]\label{p:eigenvalues}
	
	We have the following asymptotic behaviour of the eigenvalues $\{\mu(\eta)_l\}$, $l \in \{0,\pm,1,2\}$ constructed in Proposition \ref{p:fluidmodes} as $\eta \to 0$:
	\begin{equation}\label{scalings}
	\begin{split}
		&\mu(\eta)_0 \in \R \quad \text{and} \quad \mu(\eta)_0 \sim \Theta(\eta) := \begin{cases} \eta^2 \quad &\text{if }\alpha>6+\beta\,,\\  \eta^{\frac{-4+\alpha+\beta}{1+\beta}} \quad &\text{if }\alpha <6+\beta\,, \end{cases} \\
		&\text{as well as} \\
		&\re(\mu(\eta)_{\pm}) \sim \Theta(\eta)  \quad \text{and} \quad \im(\mu(\eta)_{\pm}) \sim \eta \,, \\
		&\text{and} \\ 
		&\mu(\eta)_t \in \R \quad \text{and} \quad \mu(\eta)_t \sim \tilde{\Theta}(\eta) := \begin{cases} \eta^2 \quad &\text{if }\alpha>4+\beta\,,\\  \eta^{\frac{-2+\alpha+\beta}{1+\beta}} \quad &\text{if }\alpha <4+\beta\,, \end{cases}\, \quad \text{for all } t \in \{1,2\}\,.	
	\end{split}
	\end{equation}
	\end{proposition}
	
	Notice that due to our assumption \eqref{c:constants} both $\Theta(\eta)$ and $\tilde{\Theta}(\eta)$ are well defined also in the regimes $\alpha \in (4-\beta, 6+\beta)$ and $\alpha \in (4-\beta, 4+\beta)$, since both $1+\beta>0$ and $-4+\alpha +\beta>0$. We observe that under strong enough decay of the equilibrium, i.e. $\alpha >6+\beta$, the scaling of $\{\mu(\eta)_l\}$, $l \in \{\pm,0,1,2\}$ corresponds  to the known classical scaling, see e.g. \cite{Ellis-Pinsky, GL}. On the other hand, if $\alpha <  6+\beta$ the scaling of the Boussinesq eigenvalues $\mu(\eta)_0$ and the acoustic wave eigenvalues $\mu(\eta)_\pm$ becomes fractional, while the scaling of the transversal wave eigenvalues $\mu(\eta)_t$ stays classical as long as $\alpha >4 +\beta$, and becomes also fractional if $\alpha <4 +\beta$. In particular we want to point out that in the fractional regime $\alpha < 6+\beta$ the eigenvalues $\mu(\eta)_t$ always scale faster that $\mu(\eta)_{0,\pm}$ (see Figure \ref{fig:evsdiff} and Figure \ref{fig:evsfrac}), which is the underlying reason that the equation for the macroscopic momentum $m$ becomes constant (up to a pressure gradient) in the limit. (See Theorem \ref{t:main} and Section \ref{s:convergence}).

\begin{figure}[!h]
 \begin{center}
\begin{tikzpicture}[scale=0.9]

%% -- green zone --
%
\fill[pattern=north west lines, pattern color=black!50!green] (0,0) circle (2.2 cm);

%\fill[pattern=north west lines,pattern color=black!50!green, opacity=0.5] % aire hachurée
%(0,3.8) 
%-- (0,4.5) 
%-- plot [domain=4.5:0] ({-((4.5)^2-\x*\x)^(1/2)},\x)
%-- plot [domain=0:-4.5] ({-((4.5)^2-\x*\x)^(1/2)},\x)
%-- (0,-4.5)--(0,-3.8)
%-- plot [domain=-3.8:3.8] ({-((3.8)^2-\x*\x)^(1/2)},\x)
%-- cycle;

%\fill[pattern=north west lines,pattern color=black!50!green, opacity=0.5] % aire hachurée
%(0,3.8) 
%-- (0,4.5) 
%-- plot [domain=4.5:0] ({((4.5)^2-\x*\x)^(1/2)},\x)
%-- plot [domain=0:-4.5] ({((4.5)^2-\x*\x)^(1/2)},\x)
%-- (0,-4.5)--(0,-3.8)
%-- plot [domain=-3.8:3.8] ({((3.8)^2-\x*\x)^(1/2)},\x)
%-- cycle;

%
% -- blue zone --
%\fill[fill=blue!100, opacity=0.2] % aire grisée
\fill[pattern=north west lines, pattern color=blue, opacity=0.5] % aire hachurée
(-7,-5)--(-7,5)--(-5,5)--(-5,-5)
-- cycle;

% ----- AXES -----
%
% -- axe x --
%\draw[->,line width=1.2, color=black](0.5,1)--(10.5,1) node[below]{$s$};
\draw[->,line width=1, color=black](-7,0)--(5,0);
% -- axe y --
\draw[->,line width=1, color=black](0,-5)--(0,5);
\draw[line width=1, dashed, color=blue](-5,-5)--(-5,5);
%\draw[line width=1, dashed] ;
%
% ----- CURVES -----
%
\draw [line width=2, dashed, orange] (0,0) circle (4.5);
%\draw [line width=2, dashed, black!50!green] (0,0) circle (2.5);
%\draw [line width=2, dashed, black!50!green] (0,0) circle (3.8);
%\draw [line width=2, dashed, black!50!green] (0,0) circle (1.75);
\draw [line width=2, dashed, black!50!green] (0,0) circle (2.2);

%------ Important Nodes----

\node at (-5,0)[line width=6,blue]{$\bullet$};
\node at (-5.2,0)[line width=6, above right, blue]{\scriptsize $-\lambda$};

\node at (-2.2,0)[line width=6, red]{$\bullet$};
%\node at (-2.2,0)[line width=6, red]{$\bullet$};
\node at (-2.2,2.7)[line width=6, red]{$\bullet$};
\node at (-2.2,-2.7)[line width=6, red]{$\bullet$};

\node at (-2.8,0)[line width=3, below, red]{\scriptsize $-\mu(\eta)_0$};
\node at (-2.2,2.7)[line width=3, below, red]{\scriptsize $-\mu(\eta)_+$};
\node at (-2.2,-1.9)[line width=3, below, red]{\scriptsize $-\mu(\eta)_-$};
\node at (-2.8,0.7)[line width=3, below, red]{\scriptsize $-\mu(\eta)_t$};

\node at (4.5,0)[line width=6,orange]{$\bullet$};
%\node at (4.5,0)[line width=6, anchor= north west,orange]{$r_0$};

\end{tikzpicture}
   
  \caption{Eigenvalues in the diffusive regime $\alpha \in (4-\beta, 6+\beta)$: \\ 
  \centering$\mu(\eta)_0, \re(\mu(\eta)_\pm), \mu(\eta)_t \sim -\eta^2$ and $\im (\mu(\eta)_\pm) \sim \pm \eta$.}
  \label{fig:evsdiff}
   \end{center}
  \end{figure}
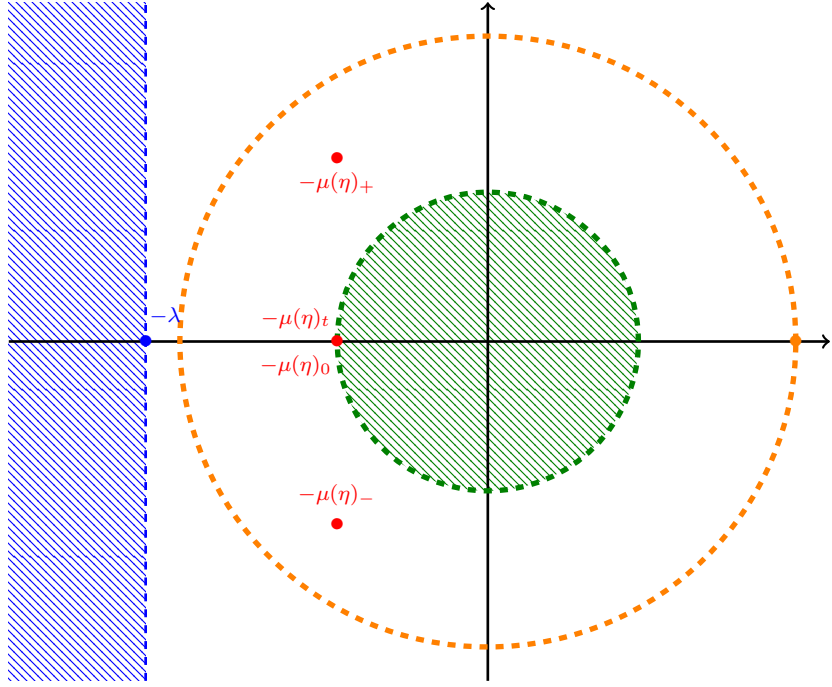
  
   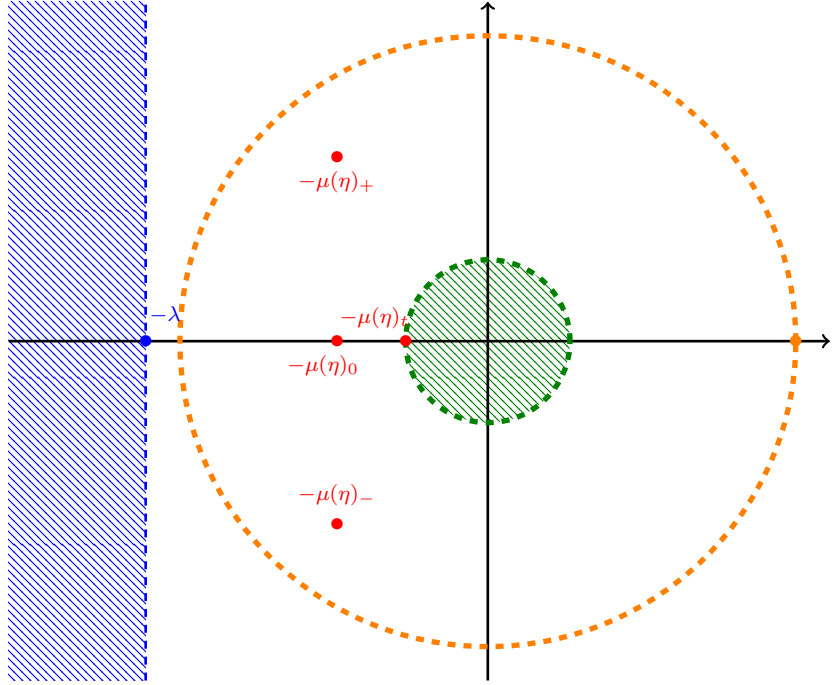
\begin{figure}[!h]
      \begin{center}
	\begin{tikzpicture}[scale=0.9]

%% -- green zone --
%
\fill[pattern=north west lines, pattern color=black!50!green] (0,0) circle (1.2 cm);

%\fill[pattern=north west lines,pattern color=black!50!green, opacity=0.5] % aire hachurée
%(0,3.8) 
%-- (0,4.5) 
%-- plot [domain=4.5:0] ({-((4.5)^2-\x*\x)^(1/2)},\x)
%-- plot [domain=0:-4.5] ({-((4.5)^2-\x*\x)^(1/2)},\x)
%-- (0,-4.5)--(0,-3.8)
%-- plot [domain=-3.8:3.8] ({-((3.8)^2-\x*\x)^(1/2)},\x)
%-- cycle;

%\fill[pattern=north west lines,pattern color=black!50!green, opacity=0.5] % aire hachurée
%(0,3.8) 
%-- (0,4.5) 
%-- plot [domain=4.5:0] ({((4.5)^2-\x*\x)^(1/2)},\x)
%-- plot [domain=0:-4.5] ({((4.5)^2-\x*\x)^(1/2)},\x)
%-- (0,-4.5)--(0,-3.8)
%-- plot [domain=-3.8:3.8] ({((3.8)^2-\x*\x)^(1/2)},\x)
%-- cycle;

%
% -- blue zone --
%\fill[fill=blue!100, opacity=0.2] % aire grisée
\fill[pattern=north west lines, pattern color=blue, opacity=0.5] % aire hachurée
(-7,-5)--(-7,5)--(-5,5)--(-5,-5)
-- cycle;

% ----- AXES -----
%
% -- axe x --
%\draw[->,line width=1.2, color=black](0.5,1)--(10.5,1) node[below]{$s$};
\draw[->,line width=1, color=black](-7,0)--(5,0);
% -- axe y --
\draw[->,line width=1, color=black](0,-5)--(0,5);
\draw[line width=1, dashed, color=blue](-5,-5)--(-5,5);
%\draw[line width=1, dashed] ;
%
% ----- CURVES -----
%
\draw [line width=2, dashed, orange] (0,0) circle (4.5);
%\draw [line width=2, dashed, black!50!green] (0,0) circle (2.5);
%\draw [line width=2, dashed, black!50!green] (0,0) circle (3.8);
%\draw [line width=2, dashed, black!50!green] (0,0) circle (1.75);
\draw [line width=2, dashed, black!50!green] (0,0) circle (1.2);

%------ Important Nodes----

\node at (-5,0)[line width=6,blue]{$\bullet$};
\node at (-5.2,0)[line width=6, above right, blue]{\scriptsize $-\lambda$};

\node at (-2.2,0)[line width=6, red]{$\bullet$};
\node at (-1.2,0)[line width=6, red]{$\bullet$};
\node at (-2.2,2.7)[line width=6, red]{$\bullet$};
\node at (-2.2,-2.7)[line width=6, red]{$\bullet$};

\node at (-2.4,0)[line width=3, below, red]{\scriptsize $-\mu(\eta)_0$};
\node at (-2.2,2.7)[line width=3, below, red]{\scriptsize $-\mu(\eta)_+$};
\node at (-2.2,-1.9)[line width=3, below, red]{\scriptsize $-\mu(\eta)_-$};
\node at (-1.65,0.7)[line width=3, below, red]{\scriptsize $-\mu(\eta)_t$};

\node at (4.5,0)[line width=6,orange]{$\bullet$};
%\node at (4.5,0)[line width=6, anchor= north west,orange]{$r_0$};

\end{tikzpicture}    
    \caption{Eigenvalues in the fractional regime $\alpha \in (4-\beta, 6+\beta)$: \\
    \centering $\mu(\eta)_0, \re(\mu(\eta)_\pm), \sim -\eta^{\frac{-4+\alpha+\beta}{1+\beta}}$, $\im (\mu(\eta)_\pm) \sim \pm \eta$, while $\mu(\eta)_t \sim \eta^2$, if $\alpha \in (4+\beta, 6+\beta)$ and $\mu(\eta)_t \sim \eta^{\frac{-2+\alpha+\beta}{1+\beta}}$ if $\alpha \in (4-\beta, 4+\beta)$.}
    \label{fig:evsfrac}
\end{center}
\end{figure}
	
%	\begin{proposition}[The diffusion coefficients]\label{p:diffcoef}
%		{\color{violet}...}
%	\end{proposition}
%
%	\begin{proposition}[Convergence]\label{p:diffcoef}
%		{\color{violet} ...  }
%	\end{proposition}

\section{Construction of the fluid modes (Proof of Proposition \ref{p:fluidmodes})}\label{s:fluidmodes}

We consider the following eigenvalue problem 
\begin{equation}\label{evp_weight}
	L_{\eta}^* \phi_\eta := L^* \phi_{\eta} + i \eta (v \cdot \sigma) \phi_\eta = -\mu(\eta)    \wv^{-\beta} \phi_{\eta}\,,
\end{equation}
and aim to show existence of eigenpairs $(\phi_\eta, \mu(\eta)) \in L^2(\wv^{-\beta} \cM) \times \C$. 

\subsection{Reformulating the eigenvalue problem into a finite dimensional system}\label{ss:reformulate}
By defining $\psi_{\eta}:=\wv^{-\frac{\beta}{2}} \phi_\eta \in L^2(\cM)$ and multiplying \eqref{evp_weight} by $\wv^{\frac{\beta}{2}}$ we can reformulate the equation above into the equivalent eigenvalue problem
\begin{equation}\label{evp}
	\tilde{L}_{\eta}^* \psi_\eta := \tilde{L}^* \psi_\eta + i \eta (v \cdot \sigma) \wv^{\beta} \psi_\eta = -\mu(\eta)    \psi_{\eta}\,.
\end{equation}
where we search for solutions $(\psi_\eta, \mu(\eta)) \in L^2( \cM) \times \mathbb{C}$. For this we make use of Hypothesis \ref{hyp:decomp}, which states that the operator $\tilde{L}$ can be written as
\begin{align}
	\tilde{L} = \mathcal{K} -1\,,
\end{align}
 where $\mathcal{K}$ is compact in $L^2(\cM)$. We further decompose 
 \begin{align}\label{decompK}
 	\mathcal{K} =   \wv^{-\frac{\beta}{2}} \left(\cP + \cS \right) \wv^{\frac{\beta}{2}} \,,
 \end{align}
 by extracting $  \wv^{-\frac{\beta}{2}}$ times the $L^2(\wv^{-\beta} \cM)$-orthogonal projection $\cP$ onto $\cN(\cL)$, which is spanned by $\left\{1, v, \frac{|v|^2-3}{3}\right\}$. This decomposition defines the second term involving the operator $\cS$. Using this decomposition, we are able to reformulate the eigenvalueproben \eqref{evp} as
 \begin{equation}\label{evp2}
 	\frac{ \wv^{-\frac{\beta}{2}} \cP \left( \wv^{\frac{\beta}{2}} \psi_\eta \right)}{ \left(1- \mu(\eta)\right)  - i \eta \wv^{\beta}(v \cdot \sigma)} = \left(\operatorname{Id} - \frac{ \wv^{-\frac{\beta}{2}} \cS^* \left( \wv^{\frac{\beta}{2}} \cdot \right)}{ \left(1- \mu(\eta)\right) - i \eta \wv^{\beta}(v \cdot \sigma)}\right) \psi_{\eta}\,.
 \end{equation}
Using that $\cP$ is the $L^2(\wv^{-\beta} \cM)$-orthogonal projection onto $\cN(\cL)$, we further notice 
 \begin{align}
 	\wv^{-\frac{\beta}{2}}\cP \left(\wv^{\frac{\beta}{2}}\psi_{\eta}\right) = \wv^{-\frac{\beta}{2}}\cP \phi_{\eta} = \wv^{-\frac{\beta}{2}} \cE(v) \cdot \begin{pmatrix} \rho_{\phi_\eta, \wv^{-\beta}} \\ m_{\phi_\eta, \wv^{-\beta}} \\ \theta_{\phi_\eta, \wv^{-\beta}}\end{pmatrix} \,,
 \end{align}
 where $\cE(v) := \left(1,\, v,\, \frac{|v|^2 -3}{2}\right)^t$ and we defined as usual 
 \begin{align}\label{evp_coeff}
 	\begin{cases}
		&\rho_{\phi_\eta, \wv^{-\beta}}: = \int_{\R^3} \phi_\eta \wv^{-\beta} \cM \, \dd v = \langle \phi_\eta,\, 1\rangle_{-\beta}\,, \\
		&m_{\phi_\eta, \wv^{-\beta}}: = \int_{\R^3} v \phi_\eta \wv^{-\beta} \cM \, \dd v = \langle \phi_\eta,\, v \rangle_{-\beta}\,, \\
		&\theta_{\phi_\eta, \wv^{-\beta}}: = \int_{\R^3} \frac{|v|^2-3}{3} \phi_\eta  \wv^{-\beta} \cM \, \dd v = \left\langle  \phi_\eta,\, \frac{|v|^2-3}{3}\right \rangle_{-\beta} \,.
	\end{cases}
 \end{align}
Next, we prove that $\operatorname{Id} (\cdot) + \left((  - \mu(\eta)) -  i \wv^{\beta}\eta (v\cdot \sigma)\right)^{-1} \wv^{-\frac{\beta}{2}} S^* \left(\wv^{\frac{\beta}{2}}\cdot \right)$ is invertible for $\eta \ll1$. Therefore, we take use of the fact that $\wv^{-\frac{\beta}{2}} \cP\left(\wv^{\frac{\beta}{2}} \cdot \right)$ is compact in $L^2(\cM)$ due to its finite rank and, since $\mathcal{K}$ is assumed to be compact, the same holds true for $\wv^{-\frac{\beta}{2}} \cS^* \left(\wv^{\frac{\beta}{2}} \cdot \right)$. This allows us to prove the following Lemma (see also Lemma 3.3 in \cite{Ellis-Pinsky}).
 
 \begin{lemma}\label{l:sinv}
 	There exists a neighbourhood $\mathcal{U}_1 \times \mathcal{U}_2 \subset \R \times \C$ around $(0,0)$ such that for all $(\eta, \mu(\eta)) \in \mathcal{U}_1 \times \mathcal{U}_2$ the operator 
	 \begin{align}\label{d:B}
 		B_{(\eta, \mu(\eta))}[|v|, v \cdot \sigma] (\cdot) : =\operatorname{Id} (\cdot) + \left( \left(1- \mu(\eta)\right) -  i \wv^{\beta}\eta (v\cdot \sigma)\right)^{-1}  \wv^{-\frac{\beta}{2}} S^* \left(\wv^{\frac{\beta}{2}}\cdot \right)
	 \end{align}
	 is invertible in $L^2(\cM)$. 
 \end{lemma}
 
 \begin{proof}
 	We start by observing that the multiplication operator $\left( \left(1- \mu(\eta)\right) + i \eta \wv^{\beta}(v\cdot \sigma)\right)^{-1}$ is bounded and, thus, 
	$$
		\left(  \left(1- \mu(\eta)\right) -  i \wv^{\beta}\eta (v\cdot \sigma)\right)^{-1}   \wv^{-\frac{\beta}{2}} S^* \left(\wv^{\frac{\beta}{2}}\cdot \right)
	$$
	is compact in $L^2(\cM)$, given that it is a product of a compact and a bounded operator, which makes the \emph{Fredholm alternative} valid. Hence, for proving that $B$ is invertible, it suffices to show that $1$ cannot be an eigenvalue of 
	$$
		g \mapsto \left( \left(1- \mu(\eta)\right) -  i \wv^{\beta}\eta (v\cdot \sigma)\right)^{-1}  \wv^{-\frac{\beta}{2}} S^* \left(\wv^{\frac{\beta}{2}} g \right)
	$$
	for $(\eta, \mu(\eta))$ small enough. Let us therefore assume that we have a sequence $\left\{(\eta_n, \mu(\eta_n)) \right\}_n$ such that $\lim_{n \to \infty}(\eta_n, \mu(\eta_n)) = (0,0)$ and for every $n \in \mathbb{N}$ let $g_n \in L^2(\cM)$ be the solution to
	\begin{align}
		g_n = \frac{  \wv^{-\frac{\beta}{2}} S^* \left(\wv^{\frac{\beta}{2}} g_n \right)}{ \left(1- \mu(\eta_n)\right) -  i \eta_n \wv^{\beta}(v\cdot \sigma)}\,.
	\end{align}
	Due to compactness of $\wv^{-\frac{\beta}{2}} S^* \left(\wv^{\frac{\beta}{2}} \cdot \right)$ we know that $\lim_{n \to \infty}\wv^{-\frac{\beta}{2}} S^* \left(\wv^{\frac{\beta}{2}} g_n \right) = \tilde{g}$ for a $\tilde{g} \in L^2(\cM)$ and a subsequence $\{g_n\}_n$ labeled the same out of convenience. We reformulate
	\begin{align}
		&g_n =  \frac{1}{( \left(1- \mu(\eta_n)\right) + i \eta_n \wv^{\beta}(v\cdot \sigma)} \left( \wv^{-\frac{\beta}{2}} S^* \left(\wv^{\frac{\beta}{2}} g_n \right) -\tilde{g} \right) +  \frac{ \tilde{g}}{ \left(1- \mu(\eta_n)\right)+ i \eta_n \wv^{\beta}(v\cdot \sigma)} \,.
	\end{align}
The multiplication operator $\frac{1}{ \left(1- \mu(\eta_n)\right)+ i \eta_n \wv^{\beta}(v\cdot \sigma)}$ is point-wise bounded in $L^2(\cM)$, hence it converges strongly to the identity operator. We obtain 
\begin{align}
	 \| g_n - \tilde{g} \|^2_{L^2(\cM)} \leq& \int_{\R^3} \left(\frac{1}{ \left(1- \mu(\eta_n)\right) + i \eta_n \wv^{\beta}(v\cdot \sigma)}\right)^2 \left(\wv^{-\frac{\beta}{2}} S^* \left(\wv^{\frac{\beta}{2}} g_n \right)- \tilde{g}\right)^2 \cM \, \dd v \\
	+& \int_{\R^3} \left(\operatorname{Id} - \frac{1}{ \left(1- \mu(\eta_n)\right) + i \eta_n \wv^{\beta}(v\cdot \sigma)}\right)^2 \tilde{g}^2 \cM \, \dd v \,.
\end{align}
Since for $n$ sufficiently large one can estimate
\begin{align}
	\left| \frac{1}{ \left(1- \mu(\eta_n)\right) + i \eta_n \wv^{\beta}(v\cdot \sigma)} \right| \leq 2\,, \quad \forall \, v \in \R^3\,,
\end{align}
we can conclude
\begin{align}
	\lim_{n \to \infty}  \| g_n - \tilde{g} \|^2_{L^2(\cM)}  = 0\,.
\end{align}
	Combining this with $\lim_{n \to \infty}\wv^{-\frac{\beta}{2}} S^* \left(\wv^{\frac{\beta}{2}} g_n \right) = \tilde{g}$ leads us to the following conclusion
	\begin{align}
		\tilde{g} = \lim_{n \to \infty} \wv^{-\frac{\beta}{2}} S^* \left(\wv^{\frac{\beta}{2}} g_n \right) = \wv^{-\frac{\beta}{2}} S^* \left(\wv^{\frac{\beta}{2}} \tilde{g} \right)\,.
	\end{align}
	If $\tilde{g} \neq 0$, this would imply that 1 is an eigenvalue of $\wv^{-\frac{\beta}{2}} S^* \left(\wv^{\frac{\beta}{2}} \cdot \right)$. In order to show that this is not possible, we assume that $f \in L^2(\cM)$ in an eigenfunction of $\wv^{-\frac{\beta}{2}} S^* \left(\wv^{\frac{\beta}{2}} \cdot \right)$ to the eigenvalue 1. From Hypothesis \ref{hyp:decomp} and \eqref{decompK} we obtain for the operator $\tilde{L}$
	\begin{align}
		\tilde{L} f =   \wv^{-\frac{\beta}{2}} \cP \left( \wv^{\frac{\beta}{2}} f \right)\,.
	\end{align}
	Taking the inner product of the equation above with elements from $\left\{ \wv^{-\frac{\beta}{2}}, v  \wv^{-\frac{\beta}{2}},  \frac{|v|^2-3}{3}\wv^{-\frac{\beta}{2}} \right\}$ yields that $ \wv^{-\frac{\beta}{2}} f$ is orthogonal to the range of $ \wv^{-\frac{\beta}{2}} \cP$, hence $ \wv^{-\frac{\beta}{2}} \cP\left( \wv^{-\frac{\beta}{2}} f \right) = 0$, which implies $\tilde{L} f =0$ and leaves $f \equiv 0$ as only possibility. Hence, 1 is not an eigenvalue of $\wv^{-\frac{\beta}{2}} S^* \left(\wv^{\frac{\beta}{2}} \cdot \right)$, which implies that $\tilde{g} \equiv 0$ and finally the desired invertibility of $B_{(\eta, \mu(\eta))}[|v|, v \cdot \sigma] (\cdot)$.
 \end{proof}
 
Lemma \ref{l:sinv} allows us to rewrite the eigenvalue problem \eqref{evp2} as
 \begin{equation}
 	B_{(\eta, \mu(\eta))}[|v|, v \cdot \sigma]^{-1} \left( \frac{ \wv^{-\frac{\beta}{2}} \cP \left( \wv^{\frac{\beta}{2}} \psi_\eta \right)}{ \left(1- \mu(\eta)\right)  - i \eta \wv^{\beta}(v \cdot \sigma)} \right) =\psi_\eta \,,
 \end{equation}
 which is equivalent to
  \begin{equation}\label{evp3}
 	G_{(\eta, \mu(\eta))}[|v|, v \cdot \sigma] \left(\cP \phi_\eta  \right) = \phi_\eta \,,
 \end{equation}
 by defining the linear operator
 \begin{align}
 	G_{(\eta, \mu(\eta))}[|v|, v \cdot \sigma] := \wv^{\frac{\beta}{2} } B_{(\eta, \mu(\eta))}[|v|, v \cdot \sigma]^{-1}\frac{ \wv^{-\frac{\beta}{2}}}{ \left(1- \mu(\eta)\right)  - i \eta \wv^{\beta}(v \cdot \sigma)}\,.
 \end{align}
By multiplying \eqref{evp3} by each entry of $\tilde{\cE}(v) :=\left(1,\, v\,, \frac{|v|^2-3}{3}\right)^t$ times $\wv^{-\beta}\cM$ before we integrate it over $v \in \R^3$ and by and using the identities \eqref{evp_coeff}, we are now able to transform the eigenvalue problem \eqref{evp3} into a closed system of algebraic equations for the unknown 5 -3imensional vector $\cC = \left( \cC_0, \, \overrightarrow{\cC}, \, \cC_4\right):= \left(\rho_{\phi_\eta, \wv^{-\beta}},\, m_{\phi_\eta, \wv^{-\beta}},\, \theta_{\phi_\eta, \wv^{-\beta}}\right) \in \C^{5}$. It is given by
\begin{small}
\begin{equation}
	\begin{split}
	  &\cC_0 \left\langle G_{(\eta, \mu(\eta))}[|v|, v \cdot \sigma] \left(1)\right) ,\, 1\right\rangle_{-\beta} +  \left\langle G_{(\eta, \mu(\eta))}[|v|, v \cdot \sigma] \left(\overrightarrow{\cC} \cdot v \right) ,\, 1\right\rangle_{-\beta} \\
	 &+ \cC_4  \left\langle G_{(\eta, \mu(\eta))}[|v|, v \cdot \sigma] \left(\frac{|v|^2-3}{2} \right) ,\, 1\right\rangle_{-\beta}= \cC_0 \,,\label{c0}
	 \end{split}
\end{equation}
\begin{equation}
	\begin{split}
	 &\cC_0 \left\langle G_{(\eta, \mu(\eta))}[|v|, v \cdot \sigma] \left(1\right) ,\, v_i \right\rangle_{-\beta} +  \left\langle G_{(\eta, \mu(\eta))}[|v|, v \cdot \sigma] \left(\overrightarrow{\cC} \cdot v \right) ,\, v_i \right\rangle_{-\beta} \\
	 &+ \cC_4  \left\langle G_{(\eta, \mu(\eta))}[|v|, v \cdot \sigma] \left(\frac{|v|^2-3}{2} \right) ,\, v_i \right\rangle_{-\beta} = \cC_j \,, \quad j \in \{1,2,3\}\,, \label{ci}
	\end{split}
\end{equation}
\begin{equation}
\begin{split}
	& \cC_0 \left\langle G_{(\eta, \mu(\eta))}[|v|, v \cdot \sigma] \left(1)\right) ,\, \frac{|v|^2-3}{3}\right\rangle_{-\beta} +  \left\langle G_{(\eta, \mu(\eta))}[|v|, v \cdot \sigma] \left(\overrightarrow{\cC} \cdot v \right) ,\, \frac{|v|^2-3}{3} \right\rangle_{-\beta} \\
	 &+ \cC_4  \left\langle G_{(\eta, \mu(\eta))}[|v|, v \cdot \sigma] \left(\frac{|v|^2-3}{2} \right) ,\, \frac{|v|^2-3}{3}\right\rangle_{-\beta} =  \cC_4 \,, \label{c4}
	\end{split}
\end{equation}
\end{small}
which can be written in matrix form
\begin{equation}\label{algsyst}
		A_{(\eta, \mu(\eta))} \cC = \cC,
\end{equation}
by defining
	\begin{align}\label{def:A}
		&\Big(A_{(\eta, \mu(\eta))} \Big)_{jk}:= \left\langle G_{(\eta, \mu(\eta))}[|v|, v \cdot \sigma] \left(\cE_k(v)\right) ,\, \tilde{\cE}_j(v) \right\rangle_{-\beta}\,.
	\end{align}
A non-trivial solution to \eqref{algsyst} exists if and only if 
$$
	\det\Big(A_{(\eta, \mu(\eta))} -I_{5}\Big)=0\,,
$$
with $I_{5}$ being the identity matrix in $\C^{5 \times 5}$. But before investigating this implicit relation, we reduce the algebraic system \eqref{algsyst} to a system of just three unknowns $\left(\cC_0,\sigma \cdot \overrightarrow{\cC}, \cC_4 \right)^t$, where $\overrightarrow{\cC}:= \big(\cC_j\big)_{j=1}^3$, and a scalar equation. 
\begin{lemma}\label{l:syst_rot}
	The algebraic system \eqref{algsyst} 
	$$	
		A_{(\eta, \mu(\eta))} \cC = \cC,
	$$
	of size $5 \times 5$ can be transformed into \\
	\begin{align}
		&0=\left(\mathcal{A}_{(\eta, \mu(\eta))}  \right)\left(\cC_0,\sigma \cdot \overrightarrow{\cC}, \cC_d \right)^t \,, \label{A3long}\\
		&0 = \left(\cC_j - \sigma_j \sigma \cdot \overrightarrow{\cC} \right)\left\langle \left(G_{(\eta, \mu(\eta))}[|v|, v_1]-\operatorname{Id}\right)(v_2) \,, v_2 \right\rangle_{-\beta}\,, \quad j \in \{1,2,3\}\,, \label{A3trans}
	\end{align}
	where $\mathcal{A}_{(\eta, \mu(\eta))} \in \C^{3\times3}$ is given by
	\begin{equation}
		\left(\mathcal{A}_{(\eta, \mu(\eta))}\right)_{j,k} := \left \langle G_{(\eta, \mu(\eta))}[|v|, v_1]\left(E(v)_k\right) \,, \tilde{E}(v)_j\right\rangle_{-\beta} - \delta_{j,k} \,, 
	\end{equation}
%	\begin{scriptsize}
%	\begin{equation*}
%	\begin{split}
%		 &\begin{pmatrix} &\left\langle G_{(\eta, \mu(\eta))}[|v|, v_1](1) \,, \right\rangle_{-\beta}  &\left\langle G_{(\eta, \mu(\eta))}[|v|, v_1](v_1) \,, \right\rangle_{-\beta} &\left\langle G_{(\eta, \mu(\eta))}[|v|, v_1]\left(\frac{|v|^2-3}{2} \right) \,, \right\rangle_{-\beta} \\  & \left\langle G_{(\eta, \mu(\eta))}[|v|, v_1](1) \,, v_1\right\rangle_{-\beta}  &\left\langle G_{(\eta, \mu(\eta))}[|v|, v_1](v_1)\,, v_1\right\rangle_{-\beta} & \left\langle G_{(\eta, \mu(\eta))}[|v|, v_1]\left(\frac{|v|^2-3}{2} \right)\,, v_1 \right\rangle_{-\beta} \\  & \left\langle G_{(\eta, \mu(\eta))}[|v|, v_1](1) \,, \frac{|v|^2-3}{3}\right\rangle_{-\beta} &\left\langle G_{(\eta, \mu(\eta))}[|v|, v_1](v_1) \,,  \frac{|v|^2-3}{3}\right\rangle_{-\beta} &\left\langle G_{(\eta, \mu(\eta))}[|v|, v_1]\left( \frac{|v|^2-3}{2} \right) \,, \frac{|v|^2-3}{3}\right\rangle_{-\beta} \end{pmatrix} - I_3\,,
%	 \end{split}
%	\end{equation*}
%	\end{scriptsize}
	with $\delta_{j,k}$ being the Kronecker -delta and
	\begin{align}
		G[\eta, \mu(\eta)](|v|,v_1) \cdot :=  \wv^{\frac{\beta}{2} } B_{(\eta, \mu(\eta))}[|v|, v_1]^{-1}\left(\frac{ \wv^{-\frac{\beta}{2}}\cdot}{ \left( 1- \mu(\eta) \right) - i \eta \wv^{\beta}v_1}\right)\,
	\end{align}
	as well as
	\begin{align}\label{def:E}
		E(v) := \left(1, v_1, \frac{|v|^2-3}{2} \right)^t\,, \quad \quad \tilde{E}(v) := \left(1, v_1, \frac{|v|^2-3}{3}\right)^t\,.
	\end{align}
\end{lemma}

\begin{proof}
The proof relies on performing the coordinate change $v \mapsto \mathcal{O} v$ in each integral, where the rotation matrix $\mathcal{O}$ is specially chosen such that
	\begin{align*}
		\mathcal{O}^t \sigma = (1,0,\cdots,0)^t.
	\end{align*}
We notice that the above choice implies that the first column of $\mathcal{O}$ is given by $\sigma$ as well as that $v \mapsto \mathcal{O}v$ is a norm-preserving coordinate change. Moreover, the operator $G_{(\eta, \mu(\eta))}[|v|, v \cdot \sigma]$ acting on components of $\tilde{\cE}(v)$ in the integrand of each matrix entry in \eqref{def:A} depends on $v$ only in terms of $|v|$ and $v \cdot \sigma$, which are transformed into $|v|$ and $v_1$ after this rotation. We will deal with each term in \eqref{c0}, \eqref{ci} and \eqref{c4} separately.
\begin{itemize}
	\item We start with $\Big(A_{(\eta, \mu(\eta))} \Big)_{jk}$ for $j,k \in \{0,4\}$, i.e. the first and the last terms of \eqref{c0} and \eqref{c4}. In that case the integrand just depends on $|v|$ and $v \cdot \sigma$ due to $\cE_0(v)=\tilde{\cE}_0(v) = 1$ and $\cE_4(v)= \frac{|v|^2-3}{2} = \cE_4(|v|)$ and $\tilde{\cE}_4(v) = \frac{|v|^2-3}{3} = \tilde{\cE}_4(|v|)$ and we obtain 
\begin{align}
	\Big(A_{(\eta, \mu(\eta))} \Big)_{jk} \cC_k = \cC_k \left\langle G_{(\eta, \mu(\eta))}[|v|, v_1] \left(\cE_k(v)\right) ,\, \wv^{-\beta}\tilde{\cE}_j(v)\right\rangle\,,
\end{align}
which gives us the alternative formulation for each the first and the last term in \eqref{c0} and \eqref{c4}. 
	\item For the corresponding middle terms in \eqref{c0} and \eqref{c4} ($j=0,4$ and $k \in \{1,2,3\}$) a bit more care is needed. Indeed, recalling that the first column of $\mathcal{O}$ is given by $\sigma$, we have
\begin{align}
	\Big((A_{(\eta, \mu(\eta))})_{jk} \Big)_{k=1}^3\overrightarrow{\cC} =& \left\langle G_{(\eta, \mu(\eta))}[|v|, v_1] \left(\overrightarrow{\cC}  \cdot \mathcal{O} v\right) ,\, \tilde{\cE}_j(v)\right\rangle_{-\beta} \notag \\
	=& \left(\overrightarrow{\cC}\cdot \sigma\right) \left\langle G_{(\eta, \mu(\eta))}[|v|, v_1] \left(v_1\right) ,\, \tilde{\cE}_j(v)\right\rangle_{-\beta}\, \label{alg_long}\\
	+& \sum_{k=2}^3  \left(\mathcal{O}^t \cdot \overrightarrow{\cC} \right)_k \left\langle G_{(\eta, \mu(\eta))}[|v|, v_1] \left(v_k\right) ,\, \tilde{\cE}_j(v)\right\rangle_{-\beta}\,. \label{equ_trans} \notag
\end{align}
We notice that $\left\langle G_{(\eta, \mu(\eta))}[|v|, v_1] \left(v_k\right) ,\, \tilde{\cE}_j(v)\right\rangle_{-\beta} =0$ for $ k\in \{2,3\}$ and $j =0,4$, since $v_k$ is an odd function and the linear operator $G_{(\eta, \mu(\eta))}[|v|, v_1]$ acting on it only depends on the first velocity-component $v_1$, while all the other factors in the integral are even. Thus, we can conclude 
\begin{align}
	\Big((A_{(\eta, \mu(\eta))})_{jk} \Big)_{k=1}^3\overrightarrow{\cC}=\left(\overrightarrow{\cC}\cdot \sigma\right) \left\langle G_{(\eta, \mu(\eta))}[|v|, v_1] \left(v_1\right) ,\, \tilde{\cE}_j(v)\right\rangle_{-\beta} \,,
\end{align}
which makes the transformation of \eqref{c0} to 
\begin{equation}\label{c02}
	\begin{split}
	  &\cC_0 \left\langle G_{(\eta, \mu(\eta))}[|v|, v_1] \left(1)\right) ,\, 1\right\rangle_{-\beta} +  \left(\sigma \cdot \overrightarrow{\cC} \right) \left\langle G_{(\eta, \mu(\eta))}[|v|, v_1] \left(v_1 \right) ,\, 1\right\rangle_{-\beta} \\
	 &+ \cC_4 \left\langle G_{(\eta, \mu(\eta))}[|v|, v_1] \left(\frac{|v|^2-3}{2} \right) ,\, 1\right\rangle_{-\beta}= \cC_0 
	 \end{split}
\end{equation}
and of \eqref{c4} to
\begin{equation} \label{c42}
\begin{split}
	& \cC_0 \left\langle G_{(\eta, \mu(\eta))}[|v|, v_1] \left(1)\right) ,\, \frac{|v|^2-3}{3}\right\rangle_{-\beta} +  \left(\sigma \cdot \overrightarrow{\cC} \right) \left\langle G_{(\eta, \mu(\eta))}[|v|, v_1] \left(v_1 \right) ,\, \frac{|v|^2-3}{3} \right\rangle_{-\beta} \\
	 &+ \cC_4  \left\langle G_{(\eta, \mu(\eta))}[|v|, v_2] \left(\frac{|v|^2-3}{2} \right) ,\, \frac{|v|^2-3}{3}\right\rangle_{-\beta} =  \cC_4 
	\end{split}
\end{equation}

complete.
\end{itemize}
We now turn to equations \eqref{ci}:
\begin{itemize}
	\item  For $k=0,4$, $j \in \{1,2,3\}$ the first and the last term are given by $ (A_{(\eta, \mu(\eta))})_{jk}\cC_k$. We transform
\begin{equation}
\begin{split}
	&(A_{(\eta, \mu(\eta))})_{jk} = \left\langle G_{(\eta, \mu(\eta))}[|v|, v_1] \left(\cE_k(v)\right) ,\, \left(\mathcal{O}v\right)_j \right\rangle_{-\beta} \\
	&=  \sigma_j \left\langle G_{(\eta, \mu(\eta))}[|v|, v_1] \left(\cE_k(v)\right) ,\, v_1 \right\rangle_{-\beta} \,, \notag
\end{split}
\end{equation}
where also here we used the first column of $\mathcal{O}$ is given by $\sigma$ and for the other terms noticed that $ \left\langle G_{(\eta, \mu(\eta))}[|v|, v_1] \left(\cE_k(v)\right) ,\, v_j \right\rangle_{-\beta} = 0$ for $k=0,4$ and $j\in \{1,2,3\}$. 
\item Last, we deal with the terms of the form 
\begin{align}
	&\left\langle G_{(\eta, \mu(\eta))}[|v|, v \cdot \sigma] \left(\overrightarrow{\cC} \cdot v \right) ,\, v_j \right\rangle_{-\beta} \notag \\
	=& \sum_{k=1}^3 \cC_k \left\langle G_{(\eta, \mu(\eta))}[|v|, v_1] \left( \left(\mathcal{O}v\right)_k \right) ,\, \left(\mathcal{O}v\right)_j \right\rangle_{-\beta} \\
	=& \sum_{k=1}^3 \cC_k \sum_{m,l =1}^3 \mathcal{O}_{k,m} \mathcal{O}_{j,l} \left\langle G_{(\eta, \mu(\eta))}[|v|, v_1] \left( v_m \right) ,\, v_l \right\rangle_{-\beta} \notag
\end{align}
Also here we notice that if $m\neq l$ the corresponding integral in the sum gives zero due to oddness of the integrand. We calculate further using orthogonality of $\mathcal{O}$ and, as already before, that that it's first column is given by $\sigma$:
\begin{small}
\begin{align}
	&\left\langle G_{(\eta, \mu(\eta))}[|v|, v_1] \left(\overrightarrow{\cC} \cdot v \right) ,\, v_j \right\rangle_{-\beta} \\
	=& \sum_{k=1}^3 \cC_k  \sigma_k \sigma_j \left\langle G_{(\eta, \mu(\eta))}[|v|, v_1] \left( v_1 \right) ,\, v_1 \right\rangle_{-\beta}  + \sum_{k=1}^3 \cC_k \sum_{l =2}^3 \mathcal{O}_{k,l} \mathcal{O}_{j,l} \left\langle G_{(\eta, \mu(\eta))}[|v|, v_1] \left( v_l \right) ,\, v_l \right\rangle_{-\beta}\notag \\
	 =&\sum_{k=1}^3 \cC_k  \sigma_k \sigma_j \left\langle G_{(\eta, \mu(\eta))}[|v|, v_1] \left( v_1 \right) ,\, v_1 \right\rangle_{-\beta} + \sum_{k=1}^3 \cC_k \sum_{l =2}^3 \mathcal{O}_{k,l} \mathcal{O}_{j,l} \left\langle G_{(\eta, \mu(\eta))}[|v|, v_1] \left( v_2 \right) ,\, v_2 \right\rangle_{-\beta}  \notag \\
	 =& \sigma_j \overrightarrow{\cC} \cdot  \sigma \left\langle G_{(\eta, \mu(\eta))}[|v|, v_1] \left( v_1 \right) ,\, v_1 \right\rangle_{-\beta}  +  \sum_{k=1}^3 \cC_k \left( \sum_{l =1}^3 \mathcal{O}_{k,l} \mathcal{O}^{-1}_{l,j} - \sigma_k\sigma_j\right)  \left\langle G_{(\eta, \mu(\eta))}[|v|, v_1] \left( v_2 \right) ,\, v_2 \right\rangle_{-\beta} \notag \\
	=& \sigma_j \overrightarrow{\cC} \cdot  \sigma \left\langle G_{(\eta, \mu(\eta))}[|v|, v_1] \left( v_1 \right) ,\, v_1 \right\rangle_{-\beta}  + \sum_{k=1}^3 \cC_k \left( \delta_{k,j} - \sigma_k\sigma_j\right) \left\langle G_{(\eta, \mu(\eta))}[|v|, v_1] \left( v_2 \right) ,\, v_2 \right\rangle_{-\beta} \notag \\
	=& \sigma_j \left(\cC \cdot \sigma\right) \left\langle G_{(\eta, \mu(\eta))}[|v|, v_1](v_1)\,, v_1\right\rangle_{-\beta} +   \left(\cC_j - \sigma_j \sigma \cdot \overrightarrow{\cC} \right)\left\langle G_{(\eta, \mu(\eta))}[|v|, v_1](v_2) \,, v_2 \right\rangle_{-\beta} \notag \,,
\end{align}
\end{small}
which concludes the transformation of \eqref{ci}:
\begin{equation}\label{ci2}
	\begin{split}
	 &\cC_0 \sigma_k \left\langle G_{(\eta, \mu(\eta))}[|v|, v_1] \left(1\right) ,\, v_1 \right\rangle_{-\beta} +  \sigma_k \left(\cC \cdot \sigma\right) \left\langle G_{(\eta, \mu(\eta))}[|v|, v_1](v_1)\,, v_1\right\rangle_{-\beta} \\
	 &+  \left(\cC_k - \sigma_k \sigma \cdot \overrightarrow{\cC} \right)\left\langle G_{(\eta, \mu(\eta))}[|v|, v_1](v_2) \,, v_2 \right\rangle_{-\beta} \\
	 &+ \cC_4 \sigma_k \left\langle G_{(\eta, \mu(\eta))}[|v|, v \cdot \sigma] \left(\frac{|v|^2-3}{2} \right) ,\, v_1\right\rangle_{-\beta} = \cC_k \,, \quad k \in \{1,2,3\}\,. 
	\end{split}
\end{equation}
\end{itemize}
Finally, by multiplying \eqref{ci} by $\sigma_k$ before summation over $k \in \{1,2,3\}$ gives
\begin{equation}\label{csigma}
\begin{split}
	&\cC_0 \left\langle G_{(\eta, \mu(\eta))}[|v|, v_1] \left(1\right) ,\, v_1 \right\rangle_{-\beta} +  \left(\cC \cdot \sigma\right) \left\langle G_{(\eta, \mu(\eta))}[|v|, v_1] \left( v_1 \right) ,\, v_1 \right\rangle_{-\beta} \\
	 &+ \cC_4  \left\langle G_{(\eta, \mu(\eta))}[|v|, v_1] \left(\frac{|v|^2-3}{2} \right) ,\, v_1 \right\rangle_{-\beta} = \left(\cC \cdot \sigma\right)  \,.
\end{split}
\end{equation}
Equations \eqref{c02}, \eqref{csigma} and \eqref{c42} now form the reduced algebraic system \eqref{A3long}. The scalar equation \eqref{A3trans} is obtained after multiplying \eqref{csigma} by $\sigma_k$ and substracting it from \eqref{ci2}, which concludes the proof.
\end{proof}
Hence, in order to solve the eigenvalue problem \eqref{evp}, we need to find solutions $\mu=\mu(\eta)$ such that 
\begin{align}
		&0=\operatorname{det} \left(\mathcal{A}_{(\eta, \mu(\eta))} \right)\,, \label{evslong}\\
		&\text{and} \notag \\
		&0 = \left\langle \left(G_{(\eta, \mu(\eta))}[|v|, v_1]-\operatorname{Id}\right)(v_2) \,, v_2 \right\rangle_{-\beta}\,. \label{evstrans}
\end{align}

\subsection{Existence of the eigenvalues}\label{ss:existence}

We aim to apply the implicite function theorem to \eqref{evslong} and \eqref{evstrans} to evolve branches of solutions $\eta \mapsto \mu(\eta)$ around $(0,0) \in \R \times \C$. Due to the infinitness of higher moments of $\cM$, we expect to encounder problems regarding differentiability at $\eta =0$ of the functions due to the lack of integrability in certain parameter-regimes. To overcome this difficulty, we solve an approximation to the problem \eqref{algsyst}, which we obtain by multiplying \eqref{evp3} by each entry of $\tilde{\cE}(v) :=\left(1,\, v\,, \frac{|v|^2-3}{3}\right)^t$ times $\chi_R\left(v\right)\wv^{-\beta}\cM$ before we integrate it over $v \in \R^3$, with $0 \leq \chi_R \leq 1$ being a smooth and radially symmetric test-function such that $\chi_R \equiv 1$ on $B(0,R)$ and $\chi_R \equiv 0$ out of $B(0,2R)$. For $R > 1$ fixed, we obtain the following approximate problem
	\begin{equation*}\label{algsystR}
		A_{(\eta, \mu(\eta))}^R \cC = \operatorname{Er}(R)\,,
	\end{equation*}
where the entries of the matrix $A_{(\eta, \mu(\eta))}^R \in \C^{5} \times \C^{5}$ are defined by
	\begin{align}\label{def:AR}
		&\Big(A^R_{(\eta, \mu(\eta))} \Big)_{jk}:= \left\langle G_{(\eta, \mu(\eta))}[|v|, v \cdot \sigma] \Big(\cE_k(v)\Big)- \cE_k(v),\, \tilde{\cE}_j(v) \chi_R \left(v\right) \right\rangle_{-\beta} \notag \\
		=& \left \langle \wv^{\frac{\beta}{2} } B_{(\eta, \mu(\eta))}[|v|,v \cdot \sigma]^{-1}\left(\frac{\wv^{-\frac{\beta}{2}}\cE(v)_k\left(   \mu(\eta) + i \eta \wv^{\beta} \left(v\cdot \sigma\right) \right)}{ (1- \mu(\eta))  - i \eta \wv^{\beta}\left(v\cdot \sigma\right)}\right),\, \tilde{\cE}(v)_j \chi_R \left(v\right) \right \rangle_{-\beta} \,.
	\end{align}
As before, the equality is given by the reformulation using $S^*(\cE(v)_k)=0$ so pull the elements of $\cN(L)$ into $B_{(\eta, \mu(\eta))}[|v|,v \cdot \sigma]^{-1}$. The vector of errors is given by
\begin{align}\label{def:er}
	 \operatorname{Er}(R):= \begin{pmatrix} \langle \cP^{\perp} \phi_\eta,\, \chi_R \left(v\right) \rangle_{-\beta} \\  \langle \cP^{\perp} \phi_\eta,\, v \chi_R \left(v\right) \rangle_{-\beta} \\  \left\langle \cP^{\perp} \phi_\eta,\, \frac{|v|^2-3}{2} \chi_R \left(v\right)\right\rangle_{-\beta} \end{pmatrix} \sim \begin{pmatrix} R^{-\beta - \alpha} \\ R^{1-\beta - \alpha}\\ R^{2-\beta -\alpha} \end{pmatrix} \longrightarrow 0 \,, \quad \text{as } R \to \infty\,, \eta \to 0\,,
\end{align}
which makes it negligible for $R \gg 1$ and, thus, we aim to find solutions $\left(\mu(\eta)^R, \cC^R \right) \in \C \times \C^{5}$ to 
\begin{equation}\label{algsystR}
		A_{(\eta, \mu(\eta)^R)}^R \cC^R = 0\,.
	\end{equation}
Analogous to Lemma \ref{l:syst_rot}, this system \eqref{def:AR} can be transformed in a $(3 \times 3)$-system together with a scalar equation, which are given by
	\begin{align}
		&0=\mathcal{A}^R_{(\eta, \mu(\eta)^R)} \left(\cC_0^R,\, \sigma \cdot \overrightarrow{\cC^R},\, \cC_4^R\right)^t\,,  \label{evslongR}\\
		&\text{and} \notag \\
		&0 = \left\langle \left(G_{(\eta, \mu(\eta))}[|v|, v_1]-\operatorname{Id}\right)(v_2) \,, v_2 \chi_R \left(v\right) \right\rangle_{-\beta} \,, \label{evstransR}
	\end{align}
and where the entries of the matrix $\mathcal{A}^R_{(\eta, \mu(\eta)^R)} \in \C^3 \times \C^3$ are defined as
\begin{align}\label{detRentries}
	{\mathcal{A}^R_{(\eta, \mu(\eta)^R)}}_{jk}  := \ \left \langle B_{(\eta, \mu(\eta)^R)}[|v|,v_1]^{-1}\left(\frac{\wv^{-\frac{\beta}{2}}E(v)_k\left(   \mu(\eta)^R + i \eta \wv^{\beta}v_1 \right)}{ (1- \mu(\eta)^R)  - i \eta \wv^{\beta}v_1}\right),\, \tilde{E}(v)_j  \chi_R \left(v\right) \wv^{-\frac{\beta}{2} } \right \rangle\,.
\end{align}
Existence of solutions $\mu(\eta)^R$ to \eqref{evslongR} can be constructed implicitly, which is now possible due to the cut-off in the integrals defining the entries of the matrix $\mathcal{A}^R_{(\eta, \mu(\eta)^R)}$.

\begin{proposition}\label{p:evslong_frac}
	For $\eta \in \R$ sufficiently close to zero equation \eqref{evslong}
	\begin{align*}
		\mu(\eta) \mapsto \operatorname{det} \left(\mathcal{A}_{(\eta, \mu(\eta))} \right) =0\,,
	\end{align*}
	has exactly three solutions $\mu_0(\eta), \mu_{+}(\eta), \mu_-(\eta) \in \C$ fulfilling the conditions
	\begin{align*}
		\mu_l(0) = 0\,, \quad \mu_l'(0) = \begin{cases} 0\,, \quad &l=0\,, \\  \pm i D \,, \quad &l = \pm\,. \end{cases}
	\end{align*}
\end{proposition}
\begin{proof}
	We start from the truncated system \eqref{evslongR} approximating \eqref{evslong} and set $\mu(\eta)^R =\eta\tilde{\mu}^R$, which allows us to extract a factor $\eta^3$ and, thus, we are left with investigating the equivalent problem
	\begin{align}\label{det2R}
		&0=\operatorname{det} \left(\mathcal{A}^R_{(\eta, \tilde{\mu}^R)} \right)\,,
	\end{align}
	where
	\begin{align*}
	{\mathcal{A}^R_{(\eta, \tilde{\mu}^R)}}_{jk}  := \ \left \langle B_{(\eta, \mu(\eta)^R)}[|v|,v_1]^{-1}\left(\frac{\wv^{-\frac{\beta}{2}}E(v)_k\left(   \tilde{\mu}^R + i \wv^{\beta}v_1 \right)}{ (1- \mu(\eta))  - i \eta \wv^{\beta}v_1}\right),\, \tilde{E}(v)_j  \chi_R \left(v\right) \wv^{-\frac{\beta}{2} } \right \rangle\,.
\end{align*}
Continuous differentiability in $(\eta, \tilde{\mu})$ up to $\eta=0$ is given due to the $\eta$-independent truncation of the integral. Next, we notice that the multiplication operator
	\begin{align*}
		v \mapsto \quad \frac{\left(   \tilde{\mu}^R + i  \wv^{\beta}v_1 \right)}{ (1- \mu(\eta)^R)  - i \eta \wv^{\beta}v_1}\,,
	\end{align*}
	occurs in each entry of the matrix $\mathcal{A}^R_{(\eta, \mu(\eta)^R)}$: in the argument of $B_{(\eta, \mu(\eta))}[|v|,v_1]^{-1}$ is is the factor the moment $\wv^{-\frac{\beta}{2}}E_k(v)$ is multiplied with. Separating it into its even part
	\begin{align*}
		v \mapsto \quad \frac{\tilde{\mu}^R (1-\mu(\eta)^R)-\eta \wv^{2\beta}v_1^2}{(1-\mu(\eta))^2+\eta^2\wv^{2\beta}v_1^2}\,,
	\end{align*}
	and its odd part
	\begin{align*}
		v \mapsto \quad \frac{i \wv^{\beta}v_1}{ (1-\mu(\eta))^2+\eta^2\wv^{2\beta}v_1^2}\,,
	\end{align*}
	allows us to split each matrixentry due to the linearity of $B_{(\eta, \mu(\eta))}[|v|,v_1]^{-1}$. With the separation of the multiplication operator into its even- and odd part the entries of the matrix in \eqref{det2R} are
	\begin{align}
		&{\mathcal{A}^R_{(\eta, \mu(\eta)^R)}}_{jk} \notag  \\
		&:=   \left \langle B_{\left(\eta, \mu(\eta)^R\right)}[|v|,v_1]^{-1}\left(\wv^{-\frac{\beta}{2}}E(v)_k\frac{\tilde{\mu}^R (1-\mu(\eta)^R)-\eta \wv^{2\beta}v_1^2}{ (1-\mu(\eta))^2+\eta^2\wv^{2\beta}v_1^2}\right),\,\tilde{E}(v)_j  \chi_R \left(v\right) \wv^{-\frac{\beta}{2} } \right \rangle\, \label{detReven}\\
		&+ i \left \langle B_{\left(\eta, \mu(\eta)^R\right)}[|v|,v_1]^{-1}\left(\wv^{-\frac{\beta}{2}} E(v)_k\frac{ \wv^{\beta}v_1}{ (1-\mu(\eta))^2+\eta^2\wv^{2\beta}v_1^2}\right),\, \tilde{E}(v)_j  \chi_R \left(v\right) \wv^{-\frac{\beta}{2} }  \right \rangle\,. \label{detRodd}
	\end{align}
	 We aim to apply the implicit function theorem at $\eta =0$, which corresponds to $(\eta, \mu(\eta)^R)= (\eta,\eta\tilde{\mu}^R)=(0,0) \in \R \times \C$. Therefore, we investigate the entries of $\mathcal{A}^R_{(\eta, \mu(\eta)^R)}$ and their behaviour as $\eta \to 0$ in more detail. Important to notice is that the multiplication operator $ \left( (1-\mu(\eta)) - \eta \wv^{-\beta}v_1\right)^{-1}$ as well as $ \left( (1-\mu(\eta))^2 +\eta^2 \wv^{-2\beta}v_1^2\right)^{-1}$ are pointwise bounded and converge strongly to the identity, thus 
	\begin{align}
		B_{(\eta, \mu(\eta))}[|v|,v_1](\cdot)=\operatorname{Id} (\cdot) + \frac{  \wv^{-\frac{\beta}{2}} S^* \left(\wv^{\frac{\beta}{2}}\cdot \right)}{ \left(1- \mu(\eta)\right) -  i \wv^{\beta}\eta (v\cdot \sigma)} \quad \longrightarrow \quad  \operatorname{Id}(\cdot) + \wv^{-\frac{\beta}{2}} S^*\left(\wv^{\frac{\beta}{2}} \cdot \right)\,,
	\end{align}
	strongly in $L^2(\cM)$ as $\eta \to 0$ (see also proof of Theroem \ref{l:sinv}). Moreover, by writing
	\begin{small}
	\begin{align*}
		&B_{(\eta, \mu(\eta))}[|v|,v_1](\cdot)=\operatorname{Id} (\cdot) + \frac{  \wv^{-\frac{\beta}{2}} S^* \left(\wv^{\frac{\beta}{2}}\cdot \right)}{ \left(1- \mu(\eta)\right) -  i \wv^{\beta}\eta (v\cdot \sigma)} \\
		=& \left(\operatorname{Id}(\cdot) + \wv^{-\frac{\beta}{2}} S^*\left(\wv^{\frac{\beta}{2}} \cdot \right)\right) - \left(1-\frac{1}{ \left(1- \mu(\eta)\right) -  i \wv^{\beta}\eta (v\cdot \sigma)}\right)\wv^{-\frac{\beta}{2}} S^* \left(\wv^{\frac{\beta}{2}}\cdot \right) \\
		=& \left(\operatorname{Id}(\cdot) + \wv^{-\frac{\beta}{2}} S^*\left(\wv^{\frac{\beta}{2}} \cdot \right)\right) \left[\operatorname{Id}(\cdot) - \left(\frac{ \mu(\eta) +  i \wv^{\beta}\eta (v\cdot \sigma)}{ \left(1- \mu(\eta)\right) -  i \wv^{\beta}\eta (v\cdot \sigma)}\right)  \left(\operatorname{Id}(\cdot) + \wv^{-\frac{\beta}{2}} S^*\left(\wv^{\frac{\beta}{2}} \cdot \right)\right)^{-1}\right. \\
		&\left. \wv^{-\frac{\beta}{2}} S^*\left(\wv^{\frac{\beta}{2}} \cdot \right)\right] \,,		
	\end{align*}
	\end{small}
	we can use the \emph{Neumann expansion} and can say the following about its inverse
	\begin{small}
	\begin{align}\label{c:Bconv}
		&B_{(\eta, \mu(\eta))}[|v|,v_1]^{-1}(\cdot) = \left[\operatorname{Id}(\cdot) - \left(\frac{ \mu(\eta) +  i \wv^{\beta}\eta (v\cdot \sigma)}{ \left(1- \mu(\eta)\right) -  i \wv^{\beta}\eta (v\cdot \sigma)}\right)  \left(\operatorname{Id}(\cdot) + \wv^{-\frac{\beta}{2}} S^*\left(\wv^{\frac{\beta}{2}} \cdot \right)\right)^{-1}\right. \notag\\
		&\left. \wv^{-\frac{\beta}{2}} S^*\left(\wv^{\frac{\beta}{2}} \cdot \right)\right]^{-1} \left(\operatorname{Id}(\cdot) + \wv^{-\frac{\beta}{2}} S^*\left(\wv^{\frac{\beta}{2}} \cdot \right)\right)^{-1}\notag\\
		=&\left[\sum_{k=0}^{\infty}  \left(\frac{ \mu(\eta) +  i \wv^{\beta}\eta (v\cdot \sigma)}{ \left(1- \mu(\eta)\right) -  i \wv^{\beta}\eta (v\cdot \sigma)}\right)^k \left(\left(\operatorname{Id}(\cdot) + \wv^{-\frac{\beta}{2}} S^*\left(\wv^{\frac{\beta}{2}} \cdot \right)\right)^{-1}\wv^{-\frac{\beta}{2}} S^*\left(\wv^{\frac{\beta}{2}}\right) \right)^{k}\right]\notag\\
		& \left(\operatorname{Id}(\cdot) + \wv^{-\frac{\beta}{2}} S^*\left(\wv^{\frac{\beta}{2}} \cdot \right)\right)^{-1} \notag\\
		=& \left(\operatorname{Id}(\cdot) + \wv^{-\frac{\beta}{2}} S^*\left(\wv^{\frac{\beta}{2}} \cdot \right)\right)^{-1} \notag\\
		&+\left[\sum_{k=1}^{\infty}  \left(\frac{ \mu(\eta) +  i \wv^{\beta}\eta (v\cdot \sigma)}{ \left(1- \mu(\eta)\right) -  i \wv^{\beta}\eta (v\cdot \sigma)}\right)^k \left(\left(\operatorname{Id}(\cdot) + \wv^{-\frac{\beta}{2}} S^*\left(\wv^{\frac{\beta}{2}} \cdot \right)\right)^{-1}\wv^{-\frac{\beta}{2}} S^*\left(\wv^{\frac{\beta}{2}}\right) \right)^{k}\right] \notag\\
		& \left(\operatorname{Id}(\cdot) + \wv^{-\frac{\beta}{2}} S^*\left(\wv^{\frac{\beta}{2}} \cdot \right)\right)^{-1} \,,
	\end{align}
	\end{small}
	where the terms in the sum converge pointwise to zero as $\eta \to 0$. Especially, we have for $\eta$ small enough 
	\begin{align*}
		\left|B_{(\eta, \mu(\eta))}[|v|,v_1]^{-1}(\cdot)\right| \leq (1+\delta)\left(\operatorname{Id}(\cdot) + \wv^{-\frac{\beta}{2}} S^*\left(\wv^{\frac{\beta}{2}} \cdot \right)\right)^{-1}\,, 
	\end{align*}
	with $\delta<1$.
	We start with the first term of the even part \eqref{detReven} 
	\begin{align*}
		\int_{\R^3} B_{\left(\eta, \mu(\eta)^R\right)}[|v|,v_1]^{-1}\left(\wv^{-\frac{\beta}{2}}E(v)_k\frac{\tilde{\mu}^R (1-\mu(\eta)^R)}{ (1-\mu(\eta))^2+\eta^2\wv^{2\beta}v_1^2}\right)\tilde{E}(v)_j  \chi_R \left(v\right) \wv^{-\frac{\beta}{2}} \cM \, \dd v\,,
	\end{align*}
	and use \eqref{c:Bconv} as well as the fact that 
	\begin{align*}
		\left|\frac{ (1-\mu(\eta)^R)}{ (1-\mu(\eta))^2+\eta^2\wv^{2\beta}v_1^2}\right|<2\,,
	\end{align*}
	for $\eta$ small enough to dominate the integrand in the following way
	\begin{small}
	\begin{align*}
		&B_{\left(\eta, \mu(\eta)^R\right)}[|v|,v_1]^{-1}\left(\wv^{-\frac{\beta}{2}}E(v)_k\frac{\tilde{\mu}^R (1-\mu(\eta)^R)}{ (1-\mu(\eta))^2+\eta^2\wv^{2\beta}v_1^2}\right)\tilde{E}(v)_j  \chi_R \left(v\right) \wv^{-\frac{\beta}{2}}\\
		& \leq 2\tilde{\mu}^R (1+\delta)  \left(\operatorname{Id}(\cdot) + \wv^{-\frac{\beta}{2}} S^*\left(\wv^{\frac{\beta}{2}} \cdot \right)\right)^{-1}\left(E_k(v)\wv^{-\frac{\beta}{2}}\right) \tilde{E}_j(v) \chi_R \left(v\right) \wv^{-\frac{\beta}{2}} \cM \,,
	\end{align*}
	\end{small}
	which is integrable. By dominated convergence we can deduce 
	\begin{align}\label{lim_even1}
		&\int_{\R^3} B_{\left(\eta, \mu(\eta)^R\right)}[|v|,v_1]^{-1}\left(\wv^{-\frac{\beta}{2}}E(v)_k\frac{\tilde{\mu}^R (1-\mu(\eta)^R)}{ (1-\mu(\eta))^2+\eta^2\wv^{2\beta}v_1^2}\right)\tilde{E}(v)_j  \chi_R \left(v\right) \wv^{-\frac{\beta}{2}} \cM \, \dd v \notag \\
		&\longrightarrow \tilde{\mu}^R \int_{\R^3}  \left(\operatorname{Id}(\cdot) + \wv^{-\frac{\beta}{2}} S^*\left(\wv^{\frac{\beta}{2}} \cdot \right)\right)^{-1}\left(E_k(v)\wv^{-\frac{\beta}{2}}\right) \tilde{E}_j(v) \chi_R \left(v\right) \wv^{-\frac{\beta}{2}} \cM \, \dd v \,,
	\end{align}
	where the integrand in the limit is the pointwise limit of the fromer. Moreover, we can split the limit integral such that
	\begin{align*}
		 &\tilde{\mu}^R \int_{\R^3}  \left(\operatorname{Id}(\cdot) + \wv^{-\frac{\beta}{2}} S^*\left(\wv^{\frac{\beta}{2}} \cdot \right)\right)^{-1}\left(E_k(v)\wv^{-\frac{\beta}{2}}\right) \tilde{E}_j(v) \chi_R \left(v\right) \wv^{-\frac{\beta}{2}} \cM \, \dd v \\
		=&\tilde{\mu}^R \int_{\R^3} E_k(v) \tilde{E}_j(v) \wv^{\beta} \cM \, \dd v \\
		&+ \tilde{\mu}^R \int_{\R^3}  \left(\operatorname{Id}(\cdot) + \wv^{-\frac{\beta}{2}} S^*\left(\wv^{\frac{\beta}{2}} \cdot \right)\right)^{-1}\left(E_k(v)\wv^{-\frac{\beta}{2}}\right) \tilde{E}_j(v) \left(\chi_R \left(v\right) -1\right)\wv^{-\frac{\beta}{2}} \cM \, \dd v \,,
		%=&\begin{cases}\tilde{\mu}^R \quad &\text{if } k=j\,, &0 \quad &\text{else} \end{cases}+ \tilde{\mu}^R \int_{\R^3}  \left(\operatorname{Id}(\cdot) + \wv^{-\frac{\beta}{2}} S^*\left(\wv^{\frac{\beta}{2}} \cdot \right)\right)^{-1}\left(E_k(v)\wv^{-\frac{\beta}{2}}\right) \tilde{E}_j(v) \left(\chi_R \left(v\right) -1\right)\wv^{-\frac{\beta}{2}} \cM \, \dd v
	\end{align*}
	where due to orthogonality we have
	\begin{align*}
		\tilde{\mu}^R \int_{\R^3} E_k(v) \tilde{E}_j(v) \wv^{\beta} \cM \, \dd v =\begin{cases}\tilde{\mu}^R \quad &\text{if } k=j,\\ 0 \quad &\text{else,} \end{cases}
	\end{align*}
	while the second term converges to 0 as $R \to \infty$ du to the intergability of the integrand.
	Similarly by dominated convergence we obtain for the odd part \eqref{detRodd}
	\begin{align}\label{lim_odd}
		&\int_{\R^3} B_{\left(\eta, \mu(\eta)^R\right)}[|v|,v_1]^{-1}\left(\wv^{-\frac{\beta}{2}} E(v)_k\frac{ \wv^{\beta}v_1}{ (1-\mu(\eta))^2+\eta^2\wv^{2\beta}v_1^2}\right) \tilde{E}(v)_j  \chi_R \left(v\right) \wv^{-\frac{\beta}{2} } \cM \, \dd v \notag  \\
		&\longrightarrow \int_{\R^3} \left(\operatorname{Id}(\cdot) + \wv^{-\frac{\beta}{2}} S^*\left(\wv^{\frac{\beta}{2}} \cdot \right)\right)^{-1}\left( E_k(v)\wv^{\frac{\beta}{2}}\right)  \tilde{E}_j(v) \chi_R(v) \wv^{-\frac{\beta}{2} }\cM \, \dd v\,.
	\end{align}
	Also here we split the limit integral into
	\begin{align*}
		&\int_{\R^3} \left(\operatorname{Id}(\cdot) + \wv^{-\frac{\beta}{2}} S^*\left(\wv^{\frac{\beta}{2}} \cdot \right)\right)^{-1}\left(E_k(v)\wv^{\frac{\beta}{2}}\right)  \tilde{E}_j(v) \chi_R(v)\wv^{-\frac{\beta}{2} } \cM \, \dd v \\
		=&\int_{\R^3} E_k(v) \tilde{E}_j(v) \cM \, \dd v \\
		&+\int_{\R^3} \left(\operatorname{Id}(\cdot) + \wv^{-\frac{\beta}{2}} S^*\left(\wv^{\frac{\beta}{2}} \cdot \right)\right)^{-1}\left(E_k(v)\wv^{\frac{\beta}{2}}\right)  \tilde{E}_j(v) \left(\chi_R(v)-1\right)\wv^{-\frac{\beta}{2} } \cM \, \dd v\,,
	\end{align*}
	where for the first term it holds
	\begin{align*}
		\int_{\R^3} E_k(v) \tilde{E}_j(v) \cM \, \dd v = \begin{cases} 0 \quad &\text{if }j+k \text{ even,} \\\neq 0 \quad &\text{else,} \end{cases}
	\end{align*}
	while also the second is negligable for big $R$. What remains is to have a closer look at the second term of the even part \eqref{detReven}
	\begin{align*}
		&- \eta \int_{\R^3} B_{\left(\eta, \mu(\eta)^R\right)}[|v|,v_1]^{-1}\left(\wv^{-\frac{\beta}{2}}E(v)_k\frac{\wv^{2\beta}v_1^2}{ (1-\mu(\eta))^2+\eta^2\wv^{2\beta}v_1^2}\right) \tilde{E}(v)_j  \chi_R \left(v\right) \wv^{-\frac{\beta}{2}} \cM \, \dd v \\
		=&- \eta \int_{\R^3}\wv^{-\frac{\beta}{2}}E(v)_k\frac{\wv^{2\beta}v_1^2}{ (1-\mu(\eta))^2+\eta^2\wv^{2\beta}v_1^2} B_{\left(\eta, \mu(\eta)^R\right)}[|v|,v_1]^{-1,*}\left( \tilde{E}(v)_j  \chi_R \left(v\right) \wv^{-\frac{\beta}{2}}\right) \cM \, \dd v\,,
	\end{align*}
	which converges to 0 as $\eta \to 0$. This can be seen by dominating the integrand by
	\begin{align*}
		\wv^{-\frac{\beta}{2}}E(v)_k\wv^{2\beta}v_1^2 \left(\operatorname{Id}(\cdot) + \wv^{-\frac{\beta}{2}} S\left(\wv^{\frac{\beta}{2}} \cdot \right)\right)^{-1}\left( \tilde{E}(v)_j  \chi_R \left(v\right) \wv^{-\frac{\beta}{2}}\right) \cM\,,
	\end{align*}
	which is integrable for every $R<\infty$. Thus, again by dominated convergence we have
	\begin{align}\label{lim_even2}
		&- \eta \int_{\R^3} B_{\left(\eta, \mu(\eta)^R\right)}[|v|,v_1]^{-1}\left(\wv^{-\frac{\beta}{2}}E(v)_k\frac{\wv^{2\beta}v_1^2}{ (1-\mu(\eta))^2+\eta^2\wv^{2\beta}v_1^2}\right) \tilde{E}(v)_j  \chi_R \left(v\right) \wv^{-\frac{\beta}{2}} \cM \, \dd v \notag \\ 
		&\longrightarrow \quad 0\,, \quad \text{as } \eta \to 0\,, \quad R<\infty \text{ fixed.}
	\end{align}
	The matrix at the limit $\eta=0$ is of the form	
	\begin{align}\label{A30R}
		\mathcal{A}_{(0, \tilde{\mu})}^R = 
		\begin{pmatrix}
			 \tilde{\mu}^R(1+I_{0,0}(R)) & i \left(\langle 1,\,v_1^2 \rangle+I_{0,1}(R)\right ) &I_{0,2}(R)  \\
			 i \left(\langle1 ,\,v_1^2 \rangle+I_{1,0}(R)\right)  & \tilde{\mu}^R (1+I_{1,1}(R))&  i \left(\langle  v_1^2,\,\frac{|v|^2-3}{2} \rangle+I_{1,2}(R)\right) \\
			 I_{2,0}(R) & i \left(\langle  v_1^2,\,\frac{|v|^2-3}{3} \rangle+I_{2,1}(R)\right) & \tilde{\mu}^R(1+I_{2,2}(R))
		\end{pmatrix}\,, 
	\end{align}
	with
	\begin{align*}
		I_{j,k}(R):=\int_{\R^3}  \left(\operatorname{Id}(\cdot) + \wv^{-\frac{\beta}{2}} S^*\left(\wv^{\frac{\beta}{2}} \cdot \right)\right)^{-1}\left(E_k(v)\wv^{-\frac{\beta}{2}}\right) \tilde{E}_j(v) \left(\chi_R \left(v\right) -1\right)\wv^{-\frac{\beta}{2}} \cM \, \dd v\,,
	\end{align*}
	which converges to zero as $R \to \infty$. Computing the determinant of the matrix at $\eta=0$ we obtain
	\begin{align}\label{det0R}
		\operatorname{det}\left(\mathcal{A}_{(0, \tilde{\mu})}^R \right) = p_3(\tilde{\mu}^R, R)\,,
	\end{align}
	which is a polynomial of order 3 in $\tilde{\mu}^R$. At leading order (for large $R$) it has the shape
	\begin{align}\label{det0}
		p_3(\tilde{\mu}, \infty) = \tilde{\mu}^3  + \tilde{\mu}D^2\,,
	\end{align}
	with $D \in \R$ defined as
	\begin{align}\label{d:D}
		D^2: = \left\langle  v_1^2,\,   \right\rangle^2 +  \left\langle  v_1^2,\, \frac{|v|^2-3}{3}  \right\rangle^2  \neq 0\,.
	\end{align}
	The three roots of \eqref{det0} can be calculated as
	\begin{align}\label{roots}
		\tilde{\mu}_0 = 0, \quad \tilde{\mu}_{\pm}= \pm i D\,,
	\end{align}
	which implies that also $\operatorname{det}\left(\mathcal{A}_{(0, \tilde{\mu})}^R \right) = p_3(\tilde{\mu}^R, R)$ has three different roots $\tilde{\mu}^R_0, \tilde{\mu}^R_{\pm}$, which are, for $R$ big, small perturbations of \eqref{roots}. In order to apply the \emph{implicit function theorem} at these points $\left(0, \tilde{\mu}^R_{l}\right) \in \R \times \C$, $l \in \{0,\pm\}$, we must show that $\pa_{\tilde{\mu}} \operatorname{det}\left(\mathcal{A}_{(\eta, \tilde{\mu})}^R \right) \neq 0$ at these points. From the formulations \eqref{detReven}, \eqref{detRodd} we notice that the only terms where the two variables $\eta$ and $\tilde{\mu}$ commonly appear have the shape $\eta \tilde{\mu}$, which gives zero, when derived with respect to $\tilde{\mu}$ and evaluated at $\eta =0$. Thus, it suffices to look at $\pa_{\tilde{\mu}} \operatorname{det}\left(\mathcal{A}_{(\eta, \tilde{\mu})}^R \right)$ and calculate
	$$
		\pa_{\tilde{\mu}}\operatorname{det}\left(\mathcal{A}_{(0, \tilde{\mu}_{0,\pm})}^R \right) = \left(3(\tilde{\mu}^R)^2  + D^2 + p_2(\tilde{\mu}^R,R)\right)_{|_{\tilde{\mu} = \tilde{\mu}_{0,\pm}}}= \begin{cases} D^2+I_0(R)\, \quad &\text{for } \tilde{\mu}^r=\tilde{\mu}_0^R\,, \\ - 2D^2+I_\pm(R) \, \quad &\text{for }  \tilde{\mu}^R=\tilde{\mu}_{\pm}^R \,,\end{cases}  
	$$
	where $I_l(R)$ note error-terms, which vanish as $R \to \infty$, which specially implies that 
	$$
		\pa_{\tilde{\mu}}\operatorname{det}\left(\mathcal{A}_{(0, \tilde{\mu}_{0,\pm})}^R \right)  \neq 0\,, \quad \text{for }R>\bar{R}\,,
	$$ 
	where $\bar{R}$ is big enough, but of order one, such that the determinant stays non-degenerate for all $R>\bar{R}$. Hence, we can conclude from the \emph{implicit function theorem} that there exists neighbourhoods $U_0 \in \R$ and $U_{\pm} \in \R$ of $\eta =0$, uniform for $r > \bar{R}$, as well as unique $C^1$ functions $\eta \mapsto \tilde{\mu}_0^R(\eta)$, $\eta \mapsto \tilde{\mu}_{\pm}^R(\eta)$ such that $\operatorname{det}\left(\mathcal{A}_{(0, \tilde{\mu}_{0,\pm})}^R \right)=0$ for all $\eta \in U^{0,\pm}$. Hence, equation \eqref{evslongR} has three solutions $\mu_{l}^R(\eta)= \eta \tilde{\mu}_{l}^R(\eta)$. Moreover, we can show that these three solutions are the only ones in a neighbourhood around $(0,0) \in \R \times \C$. For this, we observe first that from the implicit function theorem we have that the mapping $\mu \mapsto \operatorname{det}\left(\mathcal{A}_{(\eta, \mu)}^R \right)$ is holomorphic in a certain neighbourhood $U_2$ for each fixed $\eta \in U_1$. Next, we notice that $\mu \mapsto \operatorname{det}\left(\mathcal{A}_{(\eta, \mu)}^R \right)$ is of the form $\mu \mapsto \mu^3 H(\mu)$ at $\eta=0$, with $H$ holomorphic and $H(0)=1 + I(R)$, with $I(R)\ll 1$ for $R \gg1$, which can be seen directly from the explicit formulation of the entries \eqref{detRentries} and remembering the normalisation and orthogonality conditions, which give 1 at the diagonal and zero in all other entries of the matrix under the determinant. Form this we can deduce that if we choose $\mu$ on a path of the complex plane, which is circles around the origin one time, the map $\mu \mapsto \operatorname{det}\left(\mathcal{A}_{(0, \mu)}^R \right)$ encircles it 3 times. Last, we notice that for all $\mu$ in a neighbourhood around the origin, the determinant $\operatorname{det}\left(\mathcal{A}_{(\eta, \mu)}^R \right)$ converges to the determinant $\operatorname{det}\left(\mathcal{A}_{(0, \mu)}^R \right)$. Hence, for $\eta$ sufficiently small also $\mu \mapsto \operatorname{det}\left(\mathcal{A}_{(\eta, \mu)}^R \right)$ encircles the origin exactly three times if $\mu$ encircles the origin once. From which we can deduce that $\operatorname{det}\left(\mathcal{A}_{(\eta, \mu)}^R \right)=0$ also has exactly three solutions.
	These three solutions have the following conditions at $\eta=0$: 
	 \begin{align*}
	 	&\mu_{l}(0)^R =0\,, \quad \mu_{l}^{R'}(0)  = \begin{cases} I_0(R)  \,\,&\text{ if } \, l = 0 \\ \pm i D +I_\pm(R)  \,\,&\text{ if } \, l = \pm  \,.\end{cases} 
	 \end{align*}
	The error-terms $I_{0,\pm}(R)$ vanish uniformly as $R \to \infty$ implying that the limits \newline $\mu_l(\eta):=\lim_{R \to \infty} \mu_{l}(\eta)^R$ exist. Furhter, remembering \eqref{def:er} let us conclude that the limits $\mu_l(\eta)$, $l \in \{0,\pm\}$ solve the original problem \eqref{evslong} without truncation. Their behaviour for $\eta$ close to zero is characterised by the one of $\mu(\eta)_l^R$ at leading order for $R \gg1$: 
	\begin{align*}
		&\mu_{l}(0) = 0\,, \quad \mu_{l}'(0)  = \begin{cases} 0  \,\,&\text{ if } \, l = 0 \\ \pm i D   \,\,&\text{ if } \, l = \pm  \,,\end{cases} 
	\end{align*}
	which concludes the proof.
	\end{proof}

\begin{proposition}\label{p:evstrans_frac}
For $\eta$ sufficiently close to zero, the equation \eqref{evstrans}
$$
	\mu(\eta) \mapsto \left\langle \left(G_{(\eta, \mu(\eta))}[|v|, v_1]-\operatorname{Id}\right)(v_2) \,, v_2 \right\rangle_{-\beta} =0\,,
$$
has a unique solution $\mu_t(\eta) \in \C$ fulfilling the conditions
\begin{align*}
	\mu_t(0) = 0\,, \quad \mu_t'(0) =  0\,. 
\end{align*}
\end{proposition}
\begin{proof}
 	As before, we start from the truncated equation \eqref{evstransR} approximating \eqref{evstrans} and set $\mu(\eta)^R =\eta\tilde{\mu}^R$, which allows us to extract a factor $\eta$ and, thus, we are left with investigating the equivalent problem
	\begin{align}\label{eq:scalar_transR}
			0 = \left\langle B_{(\eta, \eta\tilde{\mu})}[|v|, v_1]^{-1}\left(\frac{\tilde{\mu}\wv^{-\frac{\beta}{2}}v_2 + i  \wv^{\frac{\beta}{2}}v_1v_2}{\left(1-\eta\tilde{\mu}\right) - i \eta \wv^{\beta}v_1} \right) \,, v_2 \chi_R(v)\wv^{-\frac{\beta}{2}} \right\rangle = : F^R\left(\eta, \tilde{\mu}\right)\,.
	\end{align}
	Continuous differentiability in $(\eta, \tilde{\mu})$ up to $\eta=0$ is given due to the $\eta$-independent truncation of the integral. At $\eta = 0$ is of the form
		\begin{align*}
			&F\left(0,\tilde{\mu}\right)= \left\langle B_{(0, 0)}[|v|, v_1]^{-1}\left(\tilde{\mu}\wv^{-\frac{\beta}{2}}v_2 + i   \wv^{\frac{\beta}{2}}v_1v_2\right),\, v_2 \wv^{-\frac{\beta}{2}} \right\rangle +I(R)\\
				&= \tilde{\mu} \left\langle B_{(0, 0)}[|v|, v_1]^{-1} \left(\wv^{-\frac{\beta}{2}}v_2 \right),\,\wv^{-\frac{\beta}{2}}v_2 \right \rangle + i  \left\langle B_{(0, 0)}[|v|, v_1]^{-1}\left(  \wv^{\frac{\beta}{2}}v_1v_2\right) ,\, v_2 \wv^{-\frac{\beta}{2}} \right\rangle + I(R) \\
				&= \tilde{\mu}^R + I(R) \,,
		\end{align*}
	where 
	\begin{align*}
		I(R):=\left\langle B_{(0, 0)}[|v|, v_1]^{-1}\left(\tilde{\mu}\wv^{-\frac{\beta}{2}}v_2 + i   \wv^{\frac{\beta}{2}}v_1v_2\right),\, v_2 \left(\chi_R(v)-1\right)\wv^{-\frac{\beta}{2}} \right\rangle\,,
	\end{align*}
	which due to integrability of the integrand, tends to zero as $R\to\infty$. That the second term in the equation above vanishes due to oddness of the integrand, while the first term is equal to $\tilde{\mu}$, since $\wv^{-\frac{\beta}{2}}v_2$ is invariant under $B_{(0, 0)}[|v|, v_1]$ and due to the normalisation conditions \eqref{eq:Mnorm}. Hence, $F^R\left(0,\tilde{\mu}^R\right)=0$ has $\tilde{\mu}^R = -I(R)$ as only solution. We now apply here the \emph{implicit function theorem} to $F^R\left(\eta, \tilde{\mu}^R\right)$ at $(0,-I(R)) \in \R \times \C$. It is clear that $F^R(0,-I(R))=0$. Also here, since in $F^R(\eta, \tilde{\mu}^R)$ the only terms where the two variables $\eta$ and $\tilde{\mu}^R$ have the shape $\eta \tilde{\mu}^R$, which gives zero, when derived with respect to $\tilde{\mu}^R$ and evaluated at $\eta =0$, it suffices to check that $\pa_{\tilde{\mu}^R}F^R(0, \tilde{\mu}^R)_{|_{\tilde{\mu}^R=0}} 1\neq 0$ uniformly in $R$. In virtue of the implicit function theorem there exist a neighbourhood $U^t \subset \R$ around 0 (uniformly in $R$), such that there exists a unique $C^1$ function $\eta \mapsto \tilde{\mu}_t^R(\eta)$ such that $F^R(\eta, \tilde{\mu}_t^R(\eta))=0$ for all $\eta \in U^t$. Hence, equation \eqref{evstrans} has a unique solution $\mu_t^R(\eta) = \eta \tilde{\mu}_t^R(\eta)$ fulfilling
		\begin{align*}
			\mu_t(0) = 0\,, \quad \mu_t'(0)=-I(R)\,.
		\end{align*}
The error-term $I_(R)$ vanishes uniformly in $\eta$ as $R \to \infty$ implying that the limits $\mu_t(\eta):=\lim_{R \to \infty} \mu_{l}(\eta)^R$ exist for all $\eta \in U^t$. Further, remembering \eqref{def:er} let us conclude that the limit $\mu_t(\eta)$ solves the original problem \eqref{evstrans} without truncation. Its behaviour for $\eta$ close to zero is characterised by the one of $\mu(\eta)_t^R$ at leading order for $R \gg1$: 
	\begin{align*}
		&\mu_{t}(0) = 0\,, \quad \mu_{t}'(0)  = -I(R)\,,
	\end{align*}
	which concludes the proof.
\end{proof}

\begin{remark}\label{r:symm}
	Exploiting symmetry properties of our problem, we can make even more precise statements about the three eigenvalues, solutions to \eqref{evslong}. Namely,
	$$
		\overline{\det{\left(\mathcal{A}_{(\eta, \mu)} - I_3 \right)}} = \det{\left(\mathcal{A}_{(-\eta, \overline{\mu})} - I_3 \right)} = \det{\left(\mathcal{A}_{(\eta, \overline{\mu})} - I_3 \right)}\,,
	$$
	from which it follows that if $\mu(\eta)$ is a solution, also $\overline{\mu}(\eta)$ and $\overline{\mu}(-\eta)$ have to be solutions to \eqref{evslong}. Thus
	\begin{align}\label{symmdet}
		\mu_0(\eta) \in \R\,, \quad \mu_{\pm}(\eta) = \overline{\mu_{\mp}(\eta)}\,.
	\end{align}
	Moreover, we observe for \eqref{evstrans}
	\begin{align*}
		0=& \overline{\left\langle \left(G_{(\eta, \mu(\eta))}[|v|, v_1]-\operatorname{Id}\right)(v_2) \,, v_2 \right\rangle_{-\beta}} = \left\langle \left(G_{(\eta, \overline{\mu(\eta)})}[|v|, -v_1]-\operatorname{Id}\right)(v_2) \,, v_2 \right\rangle_{-\beta} \\
		=& \left\langle \left(G_{(\eta, \mu(\eta))}[|v|, v_1]-\operatorname{Id}\right)(v_2) \,, v_2 \right\rangle_{-\beta}\,,
	\end{align*}
	and see that if $\mu(\eta)$ is a solution, then also $\overline{\mu(\eta)}$ has to be a solution. Thus, we have 
	\begin{align*}
		\mu(\eta)_t \in \R\,,
	\end{align*}
	since regarding the above Proposition \ref{p:evstrans_frac} equation \eqref{evstrans} has exactly one solution. 
	\end{remark}
	\begin{remark}\label{r:c}
		With Proposition \ref{p:evslong_frac} and Proposition \ref{p:evstrans_frac} we showed that as soon as $\alpha+\beta >4$, $\beta>-1$ and $\alpha>4$ there exist eigenvalues, denoted by $\mu_0(\eta)$, $\mu_+(\eta)$, $\mu_-(\eta)$, $\mu_t(\eta)$, for $\eta$ sufficiently close to zero. Knowing the coefficients $\cC^{l}(\eta) \in \C^{5}$, $ l \in \{0, \pm, t\}$ such that $\cP \phi_{\eta, l} = \cC^l(\eta) \cdot \cE (v)$ defines the corresponding eigenmode by \eqref{evp3}. The coefficients can be computed from \eqref{A3long} and \eqref{A3trans}. Indeed, for $\mu(\eta)_l$, with $l \in \{0,\pm\}$ we know due to Proposition \ref{p:evslong_frac} that a solution $\left(\cC_0^l(\eta), \tilde{\cC}^l(\eta)\cdot \sigma, \cC_4^l(\eta)\right)$ to \eqref{A3long} exists for $\eta$ sufficiently close to zero. Setting $\cC_i^l(\eta) = \sigma_i \tilde{\cC}^l(\eta)\cdot \sigma$ completes the vector in $\C^{5}$ and, thus, gives the corresponding eigenmode $\phi_{\eta}^l$ to the eigenvalue $\mu(\eta)_l$. For the eigenvalue $\mu(\eta)_t$ every $\cC^t(\eta)$ with $\cC_0^t (\eta)= \cC_4^t(\eta) \equiv0$ and $\overrightarrow{\cC}^t(\eta)$ such that $\overrightarrow{\cC}^t(\eta) \cdot \sigma = 0$ gives a solutionvector. Since $\sigma$ is a non-zero 3-imensional vector, this leaves us with 2 possibilities for $\cC^t(\eta)$ (hence for $\phi_{\eta, t}$), which are pairwise orthogonal to each other. Thus, $\mu(\eta)_t$ has algebraic multiplicity 2. In total we showed that there exist exactly 5 eigenpairs $\left(\mu(\eta)_l, \phi_{\eta, l}\right)$.
	\end{remark}
	
	\begin{remark}
		Let us remark at this point an imprecision in the work of Ellis and Pinsky \cite{Ellis-Pinsky}. In their analysis they assumed that the coefficients $\cC(\eta)$ in the projection of the fluid modes $\cP\left(\phi_{\eta,l}\right)= \cC(\eta) \cdot \cE(v)$ are real. Indeed, as our analysis in the upcoming Section \ref{ss:shape} shows, this is in fact not the case. even more, the interplay between their real- and imaginary parts will be of great importance.
	\end{remark}
		
	\subsection{Scalar estimate on the branches}\label{ss:scalarestimate}
	
	We remember equation \eqref{evp_weight}
	\begin{align*}
		L_{\eta}^* \phi_\eta := L^* \phi_{\eta} + i \eta (v \cdot \sigma) \phi_\eta = -\mu(\eta)    \wv^{-\beta} \phi_{\eta}\,,
	\end{align*}
	the eigenvalue problem, fulfilled by each of the $5$ eigenpairs $(\mu(\eta)_l, \phi_{\eta,l})$, which's existence is proved in the section above. We integrate \eqref{evp_weight} against $\bar{\phi_{\eta}}\cM$ and take the real part, which yields
	\begin{align}
		\operatorname{Re}\left(\mu(\eta)\right) \left\|   \phi_\eta \right\|_{-\beta}^2 = -\operatorname{Re} \langle L^* \phi_{\eta},\, \phi_\eta \rangle \geq \lambda \left\| \phi_\eta - \cP \phi_\eta \right\|_{-\beta}^2 \sim \lambda \left\|   \phi_\eta - \cP    \phi_\eta \right\|_{-\beta}^2\,,
	\end{align}
	where we used the weighted coercivity property (Hypothesis \ref{hyp:coercivity}). Further, due to $\left\|  \phi_\eta \right\|_{-\beta}^2 = \left\|   \phi_\eta  - \cP   \phi_\eta \right\|_{-\beta}^2 + 1$ we obtain the following convergence rate of the fluid modes $\phi_\eta$ towards their projections $\cP \phi_\eta$:
	\begin{align}\label{decay}
		\left\|   \phi_\eta -  \cP   \phi_\eta \right\|_{-\beta} \lesssim \operatorname{Re}\left(\mu(\eta)\right)^{\frac{1}{2}}\,.
	\end{align} 
	
	\subsection{Shape of the eigenmodes}\label{ss:shape}
	 We presuppose the following orthogonality and normality conditions of the fluid modes:
	\begin{equation}\label{c:orth}
	\begin{split}
		&0 = \left \langle \operatorname{Re}\left(\phi_{\eta,l}\right),\, \operatorname{Im}\left(\phi_{\eta,l}\right)  \right \rangle_{-\beta} \,, \quad l \in \{0,\pm,t\}\,, \\
		& 0 = \left \langle \phi_{\eta,l},\, \phi_{\eta,k}  \right \rangle_{-\beta} \,, \quad l \neq k \,,
	\end{split}
	\end{equation}
	and
	\begin{equation}\label{c:norm}
		1= \left\|\cP(\phi_{\eta,l}) \right\|_{-\beta} = \left\|\operatorname{Re}\left(\cP(\phi_{\eta,l})\right) \right\|_{-\beta} + \left\|\operatorname{Im}\left(\cP(\phi_{\eta,l})\right) \right\|_{-\beta}\,, \quad l \in \{0,\pm,t\}\,.
	\end{equation}
	Moreover, from the symmetry of the eigenvalue problem \eqref{evp_weight} we can further deduce 
	\begin{align*}
		 L^* \overline{\phi}_{\eta,l}^{\vee} + i \eta (v \cdot \sigma) \overline{\phi}_{\eta,l}^{\vee} = -\overline{\mu(\eta)_l}   \wv^{-\beta} \overline{\phi}_{\eta,l}^{\vee}\,,
	\end{align*}
	where we denoted $\phi_{\eta,l}^{\vee}(v) = \phi_{\eta,l}(-v)$. Hence, if $\mu(\eta)_l$ is an eigenvalue to the fluid mode $\phi_{\eta,l}$, the complex conjugate $\overline{\mu(\eta)_l}$ is also an eigenvalue, corresponding to the fluid mode $\overline{\phi}_{\eta}^{\vee}$. In combination with Remark \ref{r:symm}, this observation yields the following additional information to \eqref{symmdet}:
	\begin{equation}\label{c:symm}
	\begin{split}
		&\mu(\eta)_l \in \R, \quad \overline{\phi}^{\vee}_{\eta,l} = \phi_{\eta,l}\,, \quad l \in \{0,1,2\}\,, \\
		&\mu(\eta)_{\pm} = \overline{\mu(\eta)_{\mp}}, \quad \overline{\phi}^{\vee}_{\eta,\pm}= \phi_{\eta,\mp},
	\end{split}
	\end{equation}
	and further 
	\begin{align*}
		\operatorname{Re}\left( \phi_{\eta,l} \right) = \operatorname{Re}\left( \phi^{\vee}_{\eta,l} \right)\, , \quad \operatorname{Im}\left( \phi_{\eta,l} \right) = -\operatorname{Im}\left( \phi^{\vee}_{\eta,} \right)\,, \quad  l \in \{0,1,2\}\,,
	\end{align*}
	hence, evenness of $\operatorname{Re}\left( \phi_{\eta,l} \right)$ and oddness of $\operatorname{Im}\left( \phi_{\eta,l} \right)$ as well as for the wave-eigenmodes $\phi_{\eta,\pm}$ we can deduce
	\begin{align*}
		\operatorname{Re}\left( \phi_{\eta,+} \right) = \operatorname{Re}\left( \phi^{\vee}_{\eta,-} \right)\, , \quad \operatorname{Im}\left( \phi_{\eta,+} \right) = -\operatorname{Im}\left( \phi^{\vee}_{\eta,-} \right)\,.
	\end{align*}
	From this, one can deduce, respectively for each eigenvalue $\mu(\eta)_l$, the corresponding symmetry properties of the scalar coefficients $\cC^l(\eta) \in \mathbb{C}^{5}$, which are solutions to the algebraic system \eqref{algsyst} and which are fulfilling $\cP(\phi_{\eta, l}) = \cC^l(\eta) \cdot \cE(v)$. We start with the transversal waves $\phi_{\eta,t}$, $t \in \{1,2\}$, where the projection parts $\cP \phi_{\eta,t}=  \cC^t(\eta) \cdot \cE(v)$ do - per construction - not have even terms (see Remark \ref{r:c}). Thus we know
	\begin{align} \label{Ct}
		\cC^t(\eta) \in \C^3 \setminus \R^3\,, \text{for } \eta \geq 0\,.
	\end{align}
	Further, we have the following information about $\phi_{\eta,0}$ and the wave-eigenvalues $\phi_{\eta,\pm}$:
	\begin{align}
		&\cC^0(\eta)\cdot \cE = \overline{\cC^0}(\eta) \cdot \cE^\vee, \quad \text{hence} \quad C_0^0(\eta), C_4^0(\eta) \in \R,\,\,\, \overline{\overrightarrow{C^0}}(\eta) = - \overrightarrow{C^0}(\eta)\,, \label{C0} \\
		&\cC^{\pm}(\eta)\cdot \cE = \overline{\cC^{\mp}} (\eta)\cdot \cE^\vee, \quad \text{hence} \quad C_0^{\pm}(\eta)= \overline{C_0^{\mp}}(\eta),\, \,\,C_4^{\pm}(\eta)= \overline{C_4^{\mp}}(\eta),\,\,\, \overrightarrow{C^{\pm}}(\eta) = - \overline{\overrightarrow{C^{\mp}}}(\eta). \label{Cpm}
	\end{align}
	For the purely imaginary coefficients we use the notation
	$$
		\overrightarrow{C^l}(\eta) = i \tilde{C}^l(\eta) \in \C^3\setminus \R^3\,, \quad l \in \{0,1,2\}\,.
	$$
	In particular we have for $ \phi_{\eta,0}$
	\begin{equation}\label{c:C0}
		\begin{split}
			&\operatorname{Re}(\cP(\phi_{\eta,0})) = \operatorname{Re}\left(\cC^0(\eta)\right) \cdot \cE(v) = C_0^0(\eta) + \frac{|v|^2-3}{2} C_d^0(\eta) \\
			&\operatorname{Im}(\cP(\phi_{\eta,0})) = \operatorname{Im}\left(\cC^0(\eta)\right) \cdot \cE(v) = v \cdot \tilde{C}^0(\eta)\,.
		\end{split}
	\end{equation}	
		
	\subsubsection{Eigenmodes at $\eta=0$}\label{sss:coefat0}
		Having a closer look at the proof of Proposition \ref{p:evslong_frac}, we can gain information about  $\cC(0)^l\cdot \cE(v) = \cP\phi_{0,l}=\lim_{\eta \to 0} \phi_{\eta,l}$, $l \in \{0,\pm\}$, the limit of the fluid modes at $\eta \to 0$. The matrix in \eqref{A30R} in the limit $R \to \infty$ at $\eta=0$ is of the form
		\begin{align}\label{A30}
		\mathcal{A}_{(0, \tilde{\mu})} = 
		\begin{pmatrix}
			 \tilde{\mu} & i \langle 1 ,\,v_1^2 \rangle &0  \\
			 i \langle1 ,\,v_1^2 \rangle  & \tilde{\mu} &  i\langle  v_1^2,\,\frac{|v|^2-3}{2} \rangle \\
			0 & i \langle  v_1^2,\,\frac{|v|^2-3}{3} \rangle & \tilde{\mu}
		\end{pmatrix}\,,
	\end{align}
		and provides further information about the limits $\lim_{\eta \to 0}\cC^l(\eta)=:\cC^l(0)$ in the following way: 
		\begin{itemize}
			\item The coefficients $\cC^{\pm}(0)$ should fulfill
			\begin{align*}
				\mathcal{A}_{(0, \tilde{\mu})}\left(\cC^0_\pm(0),\, (\overrightarrow{\cC_\pm}(0) \cdot \sigma),\, \cC_\pm^4(0)  \left(\sigma\right)\right)^t = 0\,, \quad \text{for } \tilde{\mu}=\tilde{\mu}_{\pm}=\pm i D\,.
			\end{align*}
			First we notice that the matix has purely imaginary entries, which implies that the imaginary and real part of $\cC(0)_{\pm}$ fulfill the same equation. Due to the orthogonality condition \eqref{c:orth} we can choose $\cC(0)_{\pm} \in \R$. Combining this finding with \eqref{Cpm} we know
			\begin{align}\label{Cpm0}
				\cC_+^0(0)=\cC_-^0(0)\,, \quad \cC_+^4(0)=\cC_-^4(0)\,, \quad \overrightarrow{\cC_+}(0)=- \overrightarrow{\cC_-}(0)\,.
			\end{align}
			\item The limit $\cC^0(0)$ corresponds to the coefficients of $\cP(\phi_{\eta,0}) = \cC^0(\eta)\cdot\cE(v)$ as $\eta \to 0$. Hence, $\left(\cC_0^0(0),\, (\overrightarrow{\cC_0}(0) \cdot \sigma),\, \cC_4^0(0)  \left(\sigma\right)\right)$ should fulfill
			\begin{align*}
				\mathcal{A}_{(0, \tilde{\mu})}\left(\cC_0^0(0),\, (\overrightarrow{\cC_0}(0) \cdot \sigma),\, \cC_0^4(0)  \left(\sigma\right)\right)^t = 0\,, \quad \text{for } \tilde{\mu}=\tilde{\mu}_0=0\,,
			\end{align*}
			which implies 
			\begin{align}\label{C00}
				\overrightarrow{\cC_0}(0)=0\,, \quad \text{and} \quad \cC^0_0(0) \langle ,\,v_1^2 \rangle + \cC_0^4(0)\left\langle v_1^2,\,\frac{|v|^2-3}{2} \right\rangle =0\,.
			\end{align}
			Thus, at leading order the fluid mode $\phi_{\eta,0}$ fulfills
			$$
				\phi_{0,0}(v)=\cC^0_0(0) + \frac{|v|^2-3}{2} \cC_0^4(0)\,, \cC^0_0(0), \cC^4_0(0) \in \R\,.
			$$
			\item The limit at $\eta=0$ of the coefficients of the transversal waves is given by $\left(0,\overrightarrow{\cC}_t(0),0\right)$, where $\overrightarrow{\cC}_t(0)$, $t \in \{1,2\}$ corresponds to two vectors in $\C^3$, which are orthogonal to each other and to $\sigma$ and which are normalised.
			\end{itemize}
			In total, by combining the orthogonality and normality conditions \eqref{c:orth}, \eqref{c:norm} with \eqref{Cpm0} and \eqref{C00} we see that the leading order coefficients of $\phi_{\eta,0,\pm}$ have to fulfill the following system
			\begin{align*}%\label{c:Cpm0C00}
				&1= \left(\cC_0^0(0)\right)^2 + \frac{3}{2}\left(\cC_0^4(0)\right)^2\,, \\
				&1=\left(\cC_\pm^0(0)\right)^2 + \left(\overrightarrow{\cC_\pm}(0)\right)^2 + \frac{3}{2}\left(\cC_\pm^4(0)\right)^2\,, \\
				&0= \left(\cC_\pm^0(0)\right)^2 - \left(\overrightarrow{\cC_\pm}(0)\right)^2 + \frac{3}{2}\left(\cC_\pm^4(0)\right)^2\,, \\
				&0= \cC_0^0(0)\cC_\pm^0(0) + \frac{3}{2} \cC_0^4(0) \cC_\pm^4(0)\,.
			\end{align*}
			Algebraic manipulation of the system above reveals that the fluid modes $\phi_{\eta,l}$, $l \in \{0,\pm\}$ at leading order are of the form
			\begin{align}
				&\phi_{0,0} =  -\sqrt{\frac{2}{5}} + \frac{|v|^2 -3}{\sqrt{10}}\,, \label{phi00}\\
				&\phi_{0,+} = \sqrt{\frac{3}{10} }+ \frac{1}{\sqrt{2}}\left(v \cdot \sigma\right) + \frac{|v|^2 -3}{\sqrt{30}} \,,  \label{phi0p}\\
				&\phi_{0,-} =  \sqrt{\frac{3}{10} } - \frac{1}{\sqrt{2}}\left(v \cdot \sigma\right) + \frac{|v|^2 -3}{\sqrt{30}} \,, \label{phi0m} 
			\end{align}
			while the transversal waves have the form
			\begin{align}
				&\phi_{0,1} =  \overrightarrow{\cC}_1(0) \cdot v \,, \label{phi01}\\
				&\phi_{0,2} =  \overrightarrow{\cC}_2(0) \cdot v \,. \label{phi02}
			\end{align}

\section{Scaling of the eigenvalues (Proof of Proposition \ref{p:eigenvalues})}\label{s:scaling}
	
	As can already be seen in Proposition \ref{p:evslong_frac}, the real and the imaginary parts of the eigenvalues $\{\mu(\eta)\}_l$, $l \in \{\pm, 0, 1,2\}$ are expected to have significantly different convergence behaviour as $\eta \to 0$. In order to follow each of their parts closely, we split the eigenvalue problem \eqref{evp_weight} into its real an imaginary part, where we omit the index $l$, since the following considerations hold for a generic eigenmode and eigenvalue pair $\left(\phi_\eta, \mu(\eta)\right)$:
	\begin{align}
		&L^* \operatorname{Re}(\phi_{\eta}) - \eta (v \cdot \sigma) \operatorname{Im}(\phi_{\eta}) = -\operatorname{Re}(\mu(\eta))   \wv^{-\beta} \operatorname{Re}(\phi_{\eta}) + \operatorname{Im}(\mu(\eta))  \wv^{-\beta}\operatorname{Im}(\phi_{\eta}) \label{Re}\,,\\
		&L^* \operatorname{Im}(\phi_{\eta}) + \eta (v \cdot \sigma) \operatorname{Re}(\phi_{\eta}) = -\operatorname{Re}(\mu(\eta))   \wv^{-\beta} \operatorname{Im}(\phi_{\eta,l}) - \operatorname{Im}(\mu(\eta))  \wv^{-\beta}\operatorname{Re}(\phi_{\eta}) \label{Im}\,.
	\end{align}
	In order to obtain information about $\operatorname{Im}\left(\mu(\eta)\right)$ we multiply \eqref{Re} by $\eta^{-1}\operatorname{Im}\left(\cP(\phi_\eta)\right)\cM$, \eqref{Im} by $\eta^{-1}\operatorname{Re}\left(\cP(\phi_\eta)\right)\cM$ before integrating over the velocity space and adding the two resulting equations. This yields
	\begin{align}\label{Im_limit}
		\frac{\operatorname{Im}(\mu(\eta))}{\eta} = -\int_{\R^3} (v \cdot \sigma) \left[\operatorname{Im}\left(\phi_\eta\right)\operatorname{Im}\left(\cP(\phi_\eta)\right) + \operatorname{Re}\left(\phi_\eta\right)\operatorname{Re}\left(\cP(\phi_\eta)\right) \right]\cM \, \dd v\,.
	\end{align}
	Since due to \eqref{decay} we know that $\phi_\eta \to \cP(\phi_0)=\phi_0$ as $\eta \to 0$, we expect the integrand on the right-hand-side to converge to $(v \cdot \sigma)\left|\phi_0\right|^2 = (v \cdot \sigma)\left|\cC(0) \cdot \cE(v)\right|^2$, where
	\begin{align*}
		\cC(0) = \lim_{\eta \to 0}\cC(\eta)\,, \quad \phi_0 := \lim_{\eta \to 0} \phi_\eta\,,
	\end{align*}
	which also gives the limit, which will be proved in Section \ref{ss:Im}.
	An equation for $\operatorname{Re}\left(\mu(\eta)\right)$ can be obtained by multiplication of \eqref{Re} by $\eta^{-1}\operatorname{Re}\left(\cP(\phi_\eta)\right)\cM$ and of \eqref{Im} by $\eta^{-1}\operatorname{Im}\left(\cP(\phi_\eta)\right)\cM$, which yields after integration over $v \in \R^3$:
	\begin{align}\label{Re_limit}
		\frac{\operatorname{Re}(\mu(\eta))}{\eta} = - \int_{\R^3} (v \cdot \sigma) \left[\operatorname{Im}\left(\phi_\eta\right)\operatorname{Re}\left(\cP(\phi_\eta)\right) - \operatorname{Re}\left(\phi_\eta\right)\operatorname{Im}\left(\cP(\phi_\eta)\right) \right]\cM \, \dd v\,.
	\end{align}
	Here, we expect the right-hand-side to converge to $\operatorname{Im}\left(\phi_0\right)\operatorname{Re}\left(\phi_0\right) - \operatorname{Re}\left(\phi_0\right)\operatorname{Im}\left(\phi_0\right)=0$. Hence, a closer investigation of the behaviour in the regime $\eta \ll 1$ of the integrals defining the $\operatorname{Re}\left(\mu(\eta)\right)$ needs to be done, which can be found in Sections \ref{sss:Rediff} and \ref{sss:Refrac}.
	
	In the following, the highest order of the moments non-trivially represented in $\cP(\phi_{0}) = \cC(0) \cdot \cE(v)$ will play a crucial role. Hence, we define
	\begin{align*}
		k:= 
		\begin{cases} 
		&2\,, \quad \text{if } \cC_4(0)\neq 0\,, \\ 
		&1 \,, \quad \text{if } \cC_4(0)=0\,\, \& \,\,  \exists\, j \in \{1,2,3\} \,, \text{ s.t. } \cC_j(0) \neq 0\,, \\
		&0 \,, \quad \text{if } \cC_4(0)=0\,\, \& \,\, \overrightarrow{\cC}(0) \equiv 0\,.
		 \end{cases}
	\end{align*}	
	\begin{remark}\label{r:power}
		From Section \ref{ss:shape}, Remark \ref{r:c} and the conditions \eqref{c:orth} we can deduce that the highest moment has to be represented in $\cP \phi_{0,0}$ and  $\cP \phi_{0,\pm}$, which gives $k=2$ for this fluid mode. On the other hand, also from Remark \ref{r:c} we know that $v$ is the highest moment in $\cP \phi_{0,t}$, which results in $k=1$ for the wave- and transversal modes.
	\end{remark}
	
	\subsection{The real part of the eigenvalues}\label{ss:Re}
			
	\subsubsection{$\operatorname{Re}(\mu(\eta))$ in the regime $2+2k+\beta < \alpha$}\label{sss:Rediff}
	We define 
	$$
		H_{\eta}  := \frac{\operatorname{Re}(\cP^{\perp}\phi_{\eta})}{\eta}, \quad F_{\eta}  := \frac{\operatorname{Im}(\cP^{\perp} \phi_{\eta})}{\eta}, 
	$$
	which, according to \eqref{Re}, \eqref{Im} fulfil
	\begin{align}\label{e:Heta}
		L^* H_{\eta}  - (v \cdot \sigma) \operatorname{Im}( \phi_{\eta}) =  &-\operatorname{Re}(\mu(\eta)) \wv^{-\beta} H_{\eta} + \operatorname{Im}(\mu(\eta)) \wv^{-\beta} F_{\eta} \notag\\
		&+ \frac{\operatorname{Re}(\mu(\eta))}{\eta}  \wv^{-\beta}\operatorname{Re}\left(\cP \left(\phi_\eta\right)\right) -  \frac{\operatorname{Im}(\mu(\eta))}{\eta}  \wv^{-\beta}\operatorname{Im}\left(\cP \left(\phi_\eta\right)\right)\,, 
	\end{align}
	and
	\begin{align}\label{e:Feta}
		L^* F_{\eta}  + (v \cdot \sigma) \operatorname{Re}( \phi_{\eta}) =  &-\operatorname{Re}(\mu(\eta))   \wv^{-\beta} F_{\eta} - \operatorname{Im}(\mu(\eta))  \wv^{-\beta}H_{\eta} \notag \\
		&+ \frac{\operatorname{Re}(\mu(\eta))}{\eta}  \wv^{-\beta}\operatorname{Im}\left(\cP \left(\phi_\eta\right)\right) -  \frac{\operatorname{Im}(\mu(\eta))}{\eta}  \wv^{-\beta}\operatorname{Re}\left(\cP \left(\phi_\eta\right)\right)\,.
	\end{align}
	Due to Hypothesis \ref{hyp:coercivity}, $\alpha + \beta >4$ combined with the condition $2+2k+\beta < \alpha$ we can find solutions $H , F  \in L^2(\wv^{-\beta}\cM)$ such that 
	\begin{align*}
		&L^* H  = (v \cdot \sigma) \operatorname{Im}\left(\phi_0\right)- \operatorname{Im}(\bar{\mu})  \wv^{-\beta}\operatorname{Im}\left(\phi_0\right)\,, \\
		&L^* F  = - \left(v\cdot \sigma\right) \operatorname{Re}\left(\phi_0\right) - \operatorname{Im}(\bar{\mu})  \wv^{-\beta}\operatorname{Re}\left(\phi_0\right)\,,
	\end{align*}
	where $\operatorname{Im}(\bar{\mu})=:\lim_{\eta \to 0}\frac{\im{\mu(\eta)}}{\eta}$, later computed in Section \ref{ss:Im}, \eqref{eq:Immubar}. Subtracting each of these from the respective previous equations yields 
	\begin{align*}
		&L^* \left(H_{\eta} - H \right) - (v \cdot \sigma) \left(\operatorname{Im}( \phi_{\eta}) -\operatorname{Im}\left(\phi_0\right) \right) \\
		=& -\operatorname{Re}(\mu(\eta)) \wv^{-\beta} H_{\eta} + \operatorname{Im}(\mu(\eta)) \wv^{-\beta} F_{\eta} + \frac{\operatorname{Re}(\mu(\eta))}{\eta}  \wv^{-\beta}\operatorname{Re}\left(\cP \left(\phi_\eta\right)\right) \\
		&-  \frac{\operatorname{Im}(\mu(\eta))}{\eta} \wv^{-\beta}\operatorname{Im}\left(\cP \left(\phi_\eta\right)\right) + \operatorname{Im}(\bar{\mu}) \wv^{-\beta} \operatorname{Im}\left(\phi_0\right)\,, 
	\end{align*}
	and
	\begin{align*}
		&L^* \left(F_{\eta}  - F  \right) + (v \cdot \sigma) \left(\operatorname{Re}( \phi_{\eta}) - \operatorname{Re}\left(\phi_0\right)\right) \\
		=& -\operatorname{Re}(\mu(\eta))   \wv^{-\beta} F_{\eta} - \operatorname{Im}(\mu(\eta))  \wv^{-\beta}H_{\eta} + \frac{\operatorname{Re}(\mu(\eta))}{\eta}  \wv^{-\beta}\operatorname{Im}\left(\cP \left(\phi_\eta\right)\right) \\
		&- \frac{\operatorname{Im}(\mu(\eta))}{\eta}  \wv^{-\beta}\operatorname{Re}\left(\cP \left(\phi_\eta\right)\right) +\operatorname{Im}(\bar{\mu}) \wv^{-\beta} \operatorname{Re}\left(\phi_0\right)\,.
	\end{align*}
	After further multiplication by $\left(H_{\eta} -H \right)\cM$, resp. $\left(F_{\eta} -F \right)\cM$ and integration over the velocity space one obtains 
	\begin{align}\label{Heta}
		&-\int_{\R^3}L^*(H_{\eta} -H )(H_{\eta} -H ) \cM \, \dd v - \operatorname{Re}\left(\mu(\eta)\right) \left\|H_{\eta} -H \right\|_{-\beta}^2 \notag\\
		=& - \int_{\R^3}(v \cdot \sigma) \left(\operatorname{Im}( \phi_{\eta})-  \operatorname{Im}\left(\phi_0\right) \right)(H_{\eta} -H ) \cM \, \dd v + \operatorname{Re}\left(\mu(\eta)\right) \int_{\R^3}H (H_{\eta} -H ) \wv^{-\beta} \cM \, \dd v \notag \\
		&- \operatorname{Im}\left(\mu(\eta)\right) \int_{\R^3} F_\eta (H_{\eta} -H)   \wv^{-\beta} \cM \, \dd v \,, 
	\end{align}
	and
	\begin{align}\label{Feta}
		-&\int_{\R^3}L^*(F_{\eta} -F )(F_{\eta} -F ) \cM \, \dd v - \operatorname{Re}\left(\mu(\eta)\right) \left\|F_{\eta} -F \right\|_{-\beta}^2 \notag \\
		=&  \int_{\R^3}(v \cdot \sigma) \left(\operatorname{Re}( \phi_{\eta}) -\operatorname{Re}\left(\phi_0\right) \right)(F_{\eta} -F ) \cM \, \dd v  +  \operatorname{Im}(\mu(\eta))\int_{\R^3}H_{\eta} (F_{\eta} -F )  \wv^{-\beta} \cM \, \dd v \notag \\
		+&  \operatorname{Re}(\mu(\eta))\int_{\R^3}F (F_{\eta} -F )  \wv^{-\beta} \cM \, \dd v\,, 
	\end{align}
	where the contributions of each each the last two terms in \eqref{e:Heta} and \eqref{e:Feta} vanish, due to orthogonality of $(H_\eta-H)$ and $(F_\eta-F)$ to $\cP(\cdot)$. We first estimate the terms of equation \eqref{Heta}. Using the weighted coercivity property from Hypothesis \ref{hyp:coercivity} we can estimate the left-hand-side of \eqref{Heta} as
	\begin{align}\label{Hetalhs}
		&-\int_{\R^3}L^*(H_{\eta} -H )(H_{\eta} -H ) \cM \, \dd v - \operatorname{Re}\left(\mu(\eta)\right) \left\|H_{\eta} -H \right\|_{-\beta}^2 \\
		&\geq (\lambda - \operatorname{Re}\left(\mu(\eta)\right)) \left\|H_{\eta}-H\right\|_{-\beta}^2 + \,.
	\end{align}
	Next, we proceed with the right-hand-side of \eqref{Heta}, where the  the first term can be estimated as:
	\begin{align}\label{Hetarhs1}
		&-\int_{\R^3}(v \cdot \sigma) \left(\operatorname{Im}(\phi_{\eta})- \operatorname{Im}\left(\phi_0\right) \right)(H_{\eta}-H) \cM \, \dd v \\
		\leq& \left| \int_{\R^3}(v\cdot\sigma)^2\left(\operatorname{Im}( \phi_{\eta})-  \operatorname{Im}\left(\phi_0\right)\right)^2 \wv^{\beta} \cM \, \dd v \right|^{\frac{1}{2}}\left\|H_{\eta}^0-H^0\right\|_{-\beta}\,. \notag
	\end{align}
	We further calculate
	\begin{align}
		&\left|\int_{\R^3}(v\cdot\sigma)^2\left(\operatorname{Im}( \phi_{\eta})- \operatorname{Im}\left(\phi_0\right) \right)^2 \wv^{\beta} \cM \, \dd v \right|^{\frac{1}{2}} \\
		\sim & \left \| \operatorname{Im}( \phi_{\eta})-  \operatorname{Im}\left(\phi_0\right) \right\|_{\beta +2} \notag \\
		\leq& \left \| \operatorname{Im}( \phi_{\eta})-  \operatorname{Im}\left(\phi_0\right)  \right\|_{-\beta}^{1-\frac{1}{p}} \left \| \operatorname{Im}( \phi_{\eta})-  \operatorname{Im}\left(\phi_0\right)  \right\|_{\delta_1}^{\frac{1}{p}} \notag \\
		%\lesssim & \left \|\operatorname{Im}( \phi_{\eta})-  \operatorname{Im}\left(\phi_0\right)  \right\|_{-\beta}^{\frac{p-1}{p}} \\
	 	\lesssim & \operatorname{Re}\left(\mu(\eta)\right)^{\frac{p-1}{2p}} \left \|\operatorname{Im}\left(\cP\left(\phi_0\right) - \cP(\phi_\eta)\right)  \right\|_{-\beta}^{\frac{p-1}{p}} =:a(\eta) \longrightarrow 0 \quad \text{for } \eta \to 0\,, \notag
	\end{align}
	where we used Hölder's inequality with $p = \frac{\delta_1 + \beta}{2(\beta + 1)} $ and we chose $\delta_1   \in (2k+\beta,\,\alpha)$, which ensures $\frac{1}{p} \in (0,1)$ and further that due to Hypothesis \ref{hyp:fluidm} we have $\left \| \operatorname{Im}( \phi_{\eta})-  \operatorname{Im}\left(\phi_0\right)  \right\|_{\delta_1} \in \mathcal{O}(1)$ for $\eta \to 0$. The second term on the right-hand-side of \eqref{Heta} can be estimated as
	\begin{align}\label{Hetarhs2}
		&\left|\operatorname{Im}\left(\mu(\eta)\right) \int_{\R^3} (F_{\eta} -F) (H_{\eta} -H)   \wv^{-\beta} \cM \, \dd v \right|\leq \left| \operatorname{Im}\left(\mu(\eta)\right)\right| \left\|F_{\eta}-F\right\|_{-\beta}\left\|H_{\eta}-H\right\|_{-\beta} 
	\end{align}
	and the third term as 
	\begin{align}\label{Hetarhs4}
		&\left| \operatorname{Im}\left(\mu(\eta)\right) \int_{\R^3} F (H_{\eta} -H)   \wv^{-\beta} \cM \, \dd v  \right| \leq  \operatorname{Im}\left(\mu(\eta)\right) \left \|H_{\eta}-H\right\|_{-\beta} \left\|F\right\|_{-\beta}\\
		&\lesssim \left|\operatorname{Im}(\mu(\eta))\right| \left\|H_{\eta}-H_{\eta} \right \|_{-\beta}\,. \notag
	\end{align}
	Finally, for the error-term produced by completing the $L^2$ of $H_\eta-H$ we have
	\begin{align}\label{Hetarhs3}
		&\left|\operatorname{Re}(\mu(\eta)) \int_{\R^3}H(H_{\eta}-H)\wv^{-\beta} \cM \, \dd v \right| \leq \left|\operatorname{Re}(\mu(\eta))\right| \left \|H_{\eta}-H\right\|_{-\beta} \left\|H\right\|_{-\beta} \\
		&\lesssim \left|\operatorname{Re}(\mu(\eta))\right| \left\|H_{\eta}-H_{\eta} \right \|_{-\beta}\,. \notag
	\end{align}
	In total we obtain by putting together \eqref{Hetalhs}, \eqref{Hetarhs1}, \eqref{Hetarhs2}, \eqref{Hetarhs3} and \eqref{Hetarhs4} 
	\begin{align*}
		(\lambda- \operatorname{Re}(\mu(\eta)))  \left\|H_{\eta} -H\right\|_{-\beta} \lesssim a(\eta) + \operatorname{Re}(\mu(\eta)) +\operatorname{Im}(\mu(\eta)) + \operatorname{Im}(\mu(\eta)) \left\|F_{\eta}-F\right\|_{-\beta}\,.
	\end{align*}
	Similar estimations for \eqref{Feta} yield
	\begin{align*}
		(\lambda + \operatorname{Re}(\mu(\eta)))  \left\|F_{\eta} -F\right\|_{-\beta} \lesssim a(\eta) + \operatorname{Re}(\mu(\eta)) + \operatorname{Im}(\mu(\eta)) + \operatorname{Im}(\mu(\eta)) \left\|H_{\eta}-H\right\|_{-\beta}\,.
	\end{align*}
	Since $\beta >-1$ we have $\frac{p-1}{p} = \frac{\delta_1-\beta-2}{2(\delta_1+\beta)} <1$ and by combining these two estimates we obtain
	\begin{align}\label{H-F -3ecay}
		 \left\|F_{\eta} -F\right\|_{-\beta}, \left\|H_{\eta} -H\right\|_{-\beta}  \lesssim \max_{\eta<1}\left\{ a(\eta),\operatorname{Re}(\mu(\eta)), \eta \right\}=:b(\eta) \xrightarrow{\eta \to 0} 0\,.
	\end{align}
	For the final estimate we divide \eqref{Re_limit} by $\eta$, which yields 
	\begin{align*}
		&\frac{\operatorname{Re}(\mu(\eta))}{\eta^2} = \int_{\R^3} (v \cdot \sigma) \left(F_\eta\operatorname{Re}\left(\cP\left(\phi_\eta\right)\right) - H_\eta\operatorname{Im}\left(\cP\left(\phi_\eta\right)\right) \right) \cM\, \dd v\,,
	\end{align*}
	and which allows us to finally calculate for the real part of $\mu(\eta)$:
	\begin{align}
		&\left| \frac{\operatorname{Re}(\mu(\eta))}{\eta^2} - \int_{\R^3} (v \cdot \sigma) \left( F\operatorname{Re}\left(\phi_0\right) - H \operatorname{Im}\left(\phi_0\right) \right) \, \cM \dd v \right|  \\
		\leq&\left| \int_{\R^3} (v \cdot \sigma) \Big((F_\eta-F)\operatorname{Re}\left(\cP\left(\phi_\eta\right)\right) - (H_\eta-H)\operatorname{Im}\left(\cP\left(\phi_\eta\right)\right) \Big) \, \cM \dd v \right| \notag \\
		&+ \left| \int_{\R^3} (v \cdot \sigma) \Big(F\operatorname{Re}\left(\cP(\phi_\eta)-\phi_0\right) + H \operatorname{Im}\left(\cP(\phi_\eta)-\phi_0\right) \Big) \, \cM \dd v \right| \notag \\
		\leq & \left| \int_{\R^3} (v \cdot \sigma)^2 \left(\operatorname{Re}\left(\cP(\phi_\eta)\right) \right)^2 \wv^{\beta} \cM \dd v \right|^{\frac{1}{2}}  \left\|F_{\eta} -F\right\|_{-\beta} \notag \\
		& + \left| \int_{\R^3} (v \cdot \sigma)^2 \left(\operatorname{Im}\left(\cP(\phi_\eta)\right) \right)^2 \wv^{\beta} \cM \dd v \right|^{\frac{1}{2}}  \left\|H_{\eta} -H\right\|_{-\beta} \notag \\
		&+ \left| \int_{\R^3} (v\cdot \sigma)^2 \left(\cP(\phi_\eta)-\phi_0\right)^2 \wv^{-\beta} \right|^{\frac{1}{2}} \left(\left\|F\right\|_{-\beta} +  \left\|H\right\|_{-\beta} \right) \notag \\
		\lesssim & \left\|F_{\eta} -F\right\|_{-\beta} + \left\|H_{\eta} -H\right\|_{-\beta} + \left|\cC(\eta) - \cC(0)\right| \lesssim \operatorname{Re}(\mu(\eta))^{\frac{\delta_1-\beta-2}{2(\delta_1+\beta)}} \left|\cC(\eta) - \cC(0)\right| \longrightarrow 0 \quad \text{for } \eta \to 0\,, \notag
	\end{align}
	which provides both the limit of $\operatorname{Re}\left(\mu(\eta)\right)$ as well as its convergence rate as $\eta \to 0$.
	
	\subsubsection{$\operatorname{Re}(\mu(\eta))$ in the regime $2+2k+\beta > \alpha$}\label{sss:Refrac}
	
	Let $0 \leq \chi_R \leq 1$ be a smooth test-function such that $\chi_R \equiv 1$ on $B(0,R)$ and $\chi_R \equiv 0$ out of $B(0,2R)$. Moreover, we choose $\Theta(\eta) \sim \mu(\eta)$, who's value is to be determined and remember $\cP\left(\phi_\eta\right)=(C(\eta))\cdot \cE(v)$ for a $C(\eta)\in \C^{5}$. Integration of \eqref{Re} against $\Theta(\eta)^{-1} \operatorname{Re}\left(\cP(\phi_\eta)\right) \chi_R(\cdot \eta^{\frac{1}{1+\beta}})$ yields
	\begin{align}\label{ReRechi}
		&-\frac{\operatorname{Re}(\mu(\eta))}{\Theta(\eta)}\left\|\cP\left(\re\left(\phi_\eta\right)\right)\right\|^2_{-\beta} - \frac{1}{\Theta(\eta)}\left \langle L^*\left( \operatorname{Re}(\phi_{\eta}) - \operatorname{Re}\left(\cP(\phi_\eta)\right) \right), \operatorname{Re}\left(\cP(\phi_\eta)\right)  \chi_R(\cdot \eta^{\frac{1}{1+\beta}})\right\rangle \notag \\
		=&-\frac{\eta}{\Theta(\eta)} \int_{\R^3} (v \cdot \sigma) \operatorname{Im}(\phi_{\eta})  \operatorname{Re}\left(\cP(\phi_\eta)\right)  \chi_R(\cdot \eta^{\frac{1}{1+\beta}}) \cM \dd v  \notag \\
		&+ \frac{\operatorname{Re}(\mu(\eta))}{\Theta(\eta)} \left \langle \operatorname{Re}(\phi_{\eta})  \wv^{-\beta}, \operatorname{Re}\left(\cP(\phi_\eta)\right) \left( \chi_R(\cdot \eta^{\frac{1}{1+\beta}}) -1 \right)\right\rangle \\
		&- \frac{\operatorname{Im}(\mu(\eta))}{\Theta(\eta)} \left \langle \operatorname{Im}(\phi_{\eta})  \wv^{-\beta}, \operatorname{Re}\left(\cP(\phi_\eta)\right)\chi_R(\cdot \eta^{\frac{1}{1+\beta}}) \right\rangle \notag
	\end{align} 
	and integration of \eqref{Im} against $\Theta(\eta)^{-1} \operatorname{Im}\left(\cP(\phi_\eta)\right) \chi_R(\cdot \eta^{\frac{1}{1+\beta}})$ yields
	\begin{align}\label{ReImchi}
		&-\frac{\operatorname{Re}(\mu(\eta))}{\Theta(\eta)}\left\|\cP\left(\im\left(\phi_\eta\right)\right)\right\|^2_{-\beta} - \frac{1}{\Theta(\eta)} \left \langle L^*\left( \operatorname{Im}(\phi_{\eta}) - \operatorname{Im}\left(\cP(\phi_\eta)\right) \right), \operatorname{Im}\left(\cP(\phi_\eta)\right)  \chi_R(\cdot \eta^{\frac{1}{1+\beta}})\right\rangle \notag \\
		= & \frac{\eta}{\Theta(\eta)} \int_{\R^3} (v \cdot \sigma) \operatorname{Re}(\phi_{\eta})  \operatorname{Im}\left(\cP(\phi_\eta)\right)  \chi_R(\cdot \eta^{\frac{1}{1+\beta}}) \cM \dd v  \notag \\
		&+ \frac{\operatorname{Re}(\mu(\eta))}{\Theta(\eta)} \left \langle \operatorname{Im}(\phi_{\eta})  \wv^{-\beta}, \operatorname{Im}\left(\cP(\phi_\eta)\right) \left( \chi_R(\cdot \eta^{\frac{1}{1+\beta}}) -1\right)\right\rangle \\
		&+ \frac{\operatorname{Im}(\mu(\eta))}{\Theta(\eta)} \left \langle \operatorname{Re}(\phi_{\eta})  \wv^{-\beta}, \operatorname{Im}\left(\cP(\phi_\eta)\right)\chi_R(\cdot \eta^{\frac{1}{1+\beta}}) \right\rangle\,. \notag
	\end{align} 
	Substracting \eqref{ReRechi} from  \eqref{ReImchi} will provide information about the limiting behaviour of $\frac{\operatorname{Re}(\mu(\eta))}{\Theta(\eta)}$ as $\eta \to 0$, which is the reason we estimate all the remaining terms. We start with the terms in the first lines by using the amplitude estimates of Hypothesis \ref{hyp:amp} as well as \eqref{decay} and obtain
	\begin{align}
		 &\frac{1}{\Theta(\eta)}\left \langle \left( \operatorname{Re}(\phi_{\eta}) - \operatorname{Re}\left(\cP(\phi_\eta)\right) \right), L\left(\operatorname{Re}\left(\cP(\phi_\eta)\right)  \chi_R(\cdot \eta^{\frac{1}{1+\beta}})\right)\right\rangle  \notag \\
		 &+\frac{1}{\Theta(\eta)} \left \langle \left( \operatorname{Im}(\phi_{\eta}) - \operatorname{Im}\left(\cP(\phi_\eta)\right) \right), L\left(\operatorname{Im}\left(\cP(\phi_\eta)\right)  \chi_R(\cdot \eta^{\frac{1}{1+\beta}})\right)\right\rangle  \notag  \\
		 \leq& \frac{1}{\Theta(\eta)} \left\|\operatorname{Re}(\phi_{\eta}) - \operatorname{Re}\left(\cP(\phi_\eta)\right)\right\|_{-\beta} \left\|L\left(\operatorname{Re}\left(\cP(\phi_\eta)\right)  \chi_R(\cdot \eta^{\frac{1}{1+\beta}})\right)\right\|_{\beta}\\
		 &+ \frac{1}{\Theta(\eta)} \left\|\operatorname{Im}(\phi_{\eta}) - \operatorname{Im}\left(\cP(\phi_\eta)\right)\right\|_{-\beta} \left\|L\left(\operatorname{Im}\left(\cP(\phi_\eta)\right)  \chi_R(\cdot \eta^{\frac{1}{1+\beta}})\right)\right\|_{\beta}  \notag \\
		 \lesssim& \frac{\operatorname{Re}(\mu(\eta))^{\frac{1}{2}}}{\Theta(\eta)} R^{k-\frac{\alpha+\beta}{2}} \eta^{-\frac{k}{1+\beta}+\frac{\alpha+\beta}{2(1+\beta)}} \lesssim R^{k-\frac{\alpha+\beta}{2}} \,, \notag 
	\end{align}
	where the last approximation can be achieved by the choice 
	\begin{align}
		\Theta(\eta) = \eta^{\frac{-2k+\alpha+\beta}{1+\beta}}\,.
	\end{align}
	Next we tackle the second terms on the right-hand-side of the equations above and observe
	\begin{align}
		&\frac{\operatorname{Re}(\mu(\eta))}{\Theta(\eta)} \left \langle \operatorname{Re}(\phi_{\eta})  \wv^{-\beta}, \operatorname{Re}\left(\cP(\phi_\eta)\right) \left( \chi_R(\cdot \eta^{\frac{1}{1+\beta}}) -1 \right)\right\rangle \notag \\
		&+ \frac{\operatorname{Re}(\mu(\eta))}{\Theta(\eta)} \left \langle \operatorname{Im}(\phi_{\eta})  \wv^{-\beta}, \operatorname{Im}\left(\cP(\phi_\eta)\right) \left( \chi_R(\cdot \eta^{\frac{1}{1+\beta}}) -1\right)\right\rangle \notag \\
		\lesssim& \left| \left \langle \operatorname{Re}(\phi_{\eta})  \wv^{-\beta}, \operatorname{Re}\left(\cP(\phi_\eta)\right) \left( \chi_R(\cdot \eta^{\frac{1}{1+\beta}}) -1 \right)\right\rangle \right| \\
		&+ \left|\left \langle \operatorname{Im}(\phi_{\eta})  \wv^{-\beta}, \operatorname{Im}\left(\cP(\phi_\eta)\right) \left( \chi_R(\cdot \eta^{\frac{1}{1+\beta}}) -1\right)\right\rangle \right| \notag \\
		\lesssim& R^{2k-\alpha-\beta} \eta^{\frac{-2k+\alpha+\beta}{1+\beta}} \,. \notag
	\end{align}
	Similarly, we calculate for the last terms on the right-hand-side,
	\begin{align}
		&- \frac{\operatorname{Im}(\mu(\eta))}{\Theta(\eta)} \left \langle \operatorname{Im}(\phi_{\eta})  \wv^{-\beta}, \operatorname{Re}\left(\cP(\phi_\eta)\right)\chi_R(\cdot \eta^{\frac{1}{1+\beta}}) \right\rangle  \notag \\
		&- \frac{\operatorname{Im}(\mu(\eta))}{\Theta(\eta)} \left \langle \operatorname{Re}(\phi_{\eta})  \wv^{-\beta}, \operatorname{Im}\left(\cP(\phi_\eta)\right)\chi_R(\cdot \eta^{\frac{1}{1+\beta}}) \right\rangle \notag \\
		\leq& \frac{\left|\operatorname{Im}(\mu(\eta))\right|}{\Theta(\eta)} \left| \left \langle \operatorname{Re}(\phi_{\eta})  \wv^{-\beta}, \operatorname{Im}\left(\cP(\phi_\eta)\right) \left(\chi_R(\cdot \eta^{\frac{1}{1+\beta}})-1\right)\right\rangle \right| \\
		&+\frac{\left|\operatorname{Im}(\mu(\eta))\right|}{\Theta(\eta)}  \left|\left \langle \operatorname{Im}(\phi_{\eta})  \wv^{-\beta}, \operatorname{Re}\left(\cP(\phi_\eta)\right) \left(\chi_R(\cdot \eta^{\frac{1}{1+\beta}})-1\right) \right\rangle \right| \notag \\
		\lesssim& \frac{\left|\operatorname{Im}(\mu(\eta))\right|}{\Theta(\eta)} R^{2k-\alpha-\beta} \eta^{\frac{-2k+\alpha+\beta}{1+\beta}} \lesssim \left|\operatorname{Im}(\mu(\eta))\right| R^{2k-\alpha-\beta} \,. \notag 
	\end{align}
	Combining these estimates, we have shown
	\begin{align}\label{Relim}
		&\left|\frac{\operatorname{Re}(\mu(\eta))}{\Theta(\eta)} + \frac{\eta}{\Theta(\eta)} \int_{\R^3} (v \cdot \sigma) \left[ \operatorname{Im}(\phi_{\eta})  \operatorname{Re}\left(\cP(\phi_\eta)\right) + \operatorname{Re}(\phi_{\eta})  \operatorname{Im}\left(\cP(\phi_\eta)\right)\right] \chi_R(\cdot \eta^{\frac{1}{1+\beta}}) \cM \dd v \right| \\
		& \lesssim R^{k-\frac{\alpha+\beta}{2}} +R^{2k-\alpha-\beta} \eta^{\frac{-2k+\alpha+\beta}{1+\beta}}  + \left|\operatorname{Im}(\mu(\eta))\right| R^{2k-\alpha-\beta} \longrightarrow   R^{k-\frac{\alpha+\beta}{2}}\,, \quad \text{as } \eta \to 0\,. \notag
	\end{align}
	Finally, we have a closer look at the behaviour of the integral in \eqref{Relim} as $\eta \to 0$. We perform the coordinate change $u = v \eta^{\frac{1}{1+\beta}}$, which yields
	\begin{align}
		&\frac{\eta}{\Theta(\eta)} \int_{\R^3} (v \cdot \sigma) \operatorname{Im}(\phi_{\eta})  \operatorname{Re}\left(\cP(\phi_\eta)\right)  \chi_R(\cdot \eta^{\frac{1}{1+\beta}}) \cM \dd v \notag\\
		&+\frac{\eta}{\Theta(\eta)} \int_{\R^3} (v \cdot \sigma) \operatorname{Re}(\phi_{\eta})  \operatorname{Im}\left(\cP(\phi_\eta)\right)  \chi_R(\cdot \eta^{\frac{1}{1+\beta}}) \cM \dd v  \notag \\
		=&\frac{\eta^{\frac{-2k+\alpha+\beta}{1+\beta}}}{\Theta(\eta)} \int_{\R^3} (u \cdot \sigma) \operatorname{Im} (\Phi_{\eta}(u))\operatorname{Re} \left(C(\eta)\right) \cdot \cE_{\eta}\left( u \right)\chi_R(u) \cM_{\eta}  \, \dd u \notag \\
		&+ \frac{\eta^{\frac{-2k+\alpha+\beta}{1+\beta}}}{\Theta(\eta)} \int_{\R^3} (u \cdot \sigma)\operatorname{Re} (\Phi_{\eta}(u)) \operatorname{Im} \left(C(\eta)\right) \cdot \cE_{\eta}\left( u \right) \chi_R(u)\cM_{\eta} \, \dd u\,, \notag
	\end{align}
	where we defined
	\begin{align}\label{d:Phi}
		\Phi_{\eta}(u)  : =& \eta^{\frac{k}{1+\beta}} \phi_{\eta}(v) \,,
	\end{align}
	together with
	\begin{align}\label{d:moments_eta}
		\cE_{\eta, k}(u):= \left(\eta^{\frac{k}{1+\beta}}, \eta^{\frac{k-1}{1+\beta}} u, \frac{\eta^{\frac{k-2}{1+\beta}}|u|^2-\eta^{\frac{k}{1+\beta}}d}{2}\right)^t, \quad \cM_{\eta}(u) := c_0\left(\eta^{\frac{2}{1+\beta}} + |u|^2\right)^{-\frac{\alpha+d}{2}}\,.
	\end{align}
	Hence, with the significant scaling choice $\Theta(\eta) = \eta^{\frac{-2k+\alpha+\beta}{(1+\beta)}}$ we obtain
	\begin{align}
		&-\frac{\eta}{\Theta(\eta)} \int_{\R^3} (v \cdot \sigma) \operatorname{Im}(\phi_{\eta})  \operatorname{Re}\left(\cP(\phi_\eta)\right)  \chi_R(\cdot \eta^{\frac{1}{1+\beta}}) \cM \dd v \notag\\
		&+\frac{\eta}{\Theta(\eta)} \int_{\R^3} (v \cdot \sigma) \operatorname{Re}(\phi_{\eta})  \operatorname{Im}\left(\cP(\phi_\eta)\right)  \chi_R(\cdot \eta^{\frac{1}{1+\beta}}) \cM \dd v  \notag \\
		\lesssim& -\int_{\R^3} (u \cdot \sigma) \operatorname{Im} (\Phi_{\eta}(u))\operatorname{Re} \left(C(\eta)\right) \cdot \cE_{\eta,k}\left( u \right)\chi_R(u) \cM_{\eta}  \, \dd u \notag \\
		&+\int_{\R^3} (u \cdot \sigma)\operatorname{Re} (\Phi_{\eta}(u)) \operatorname{Im} \left(C(\eta)\right) \cdot \cE_{\eta,k}\left( u \right) \chi_R(u)\cM_{\eta} \, \dd u\,. \notag
		\end{align}
	In combination with the other estimates, we have shown
	\begin{align}
		&\left|\frac{\operatorname{Re}(\mu(\eta))}{\Theta(\eta)} + \int_{\R^3} (u \cdot \sigma) \operatorname{Im} (\Phi_{\eta}(u))\operatorname{Re} \left(C(\eta)\right) \cdot \cE_{\eta,k}\left( u \right)\chi_R(u) \cM_{\eta}  \, \dd u \right. \\
		&\left.- \int_{\R^3} (u \cdot \sigma)\operatorname{Re} (\Phi_{\eta}(u)) \operatorname{Im} \left(C(\eta)\right) \cdot \cE_{\eta,k}\left( u \right) \chi_R(u)\cM_{\eta} \, \dd u \right| \lesssim R^{2k-\alpha+\beta}\,. \notag
	\end{align}
	One observes easily that under the boundedness conditions in Hypothesis \ref{hyp:fluidm} for a given $R>0$ the integrand is integrable uniformly in $\eta$, from which convergence of the above integral for $\eta \to 0$ follows. The double limit $R \to \infty, \eta \to 0$ exists and we have
	\begin{align}
		 \lim_{R \to \infty, \eta \to 0} &\int_{\R^3} (u \cdot \sigma) \operatorname{Im} (\Phi_{\eta}(u))\operatorname{Re} \left(C(\eta)\right) \cdot \cE_{\eta,k}\left( u \right)\chi_R(u) \cM_{\eta}  \, \dd u \notag \\
		 &= \operatorname{Re}(\cC_k(0)) \int_{\R^3} (u \cdot \sigma) \operatorname{Im} (\Phi(u)) |u|^{k-\alpha -3} \, \dd u\,,
	\end{align} 
	as well as
	\begin{align}
		 \lim_{R \to \infty, \eta \to 0} &\int_{\R^3} (u \cdot \sigma)\operatorname{Re} (\Phi_{\eta}(u)) \operatorname{Im} \left(C(\eta)\right) \cdot \cE_{\eta,k}\left( u \right) \chi_R(u)\cM_{\eta} \, \dd u \notag \\
		 &= \operatorname{Im}(\cC_k(0)) \int_{\R^3} (u \cdot \sigma) \operatorname{Re} (\Phi(u)) |u|^{k-\alpha -3} \, \dd u\,,
	\end{align}
	where
	\begin{align*}
		 \Phi(u):=\lim_{\eta \to 0} \Phi_{\eta}(u) \,, \quad \cC_k (0) := \lim_{\eta \to 0} \cC_k(\eta) \,.
	\end{align*}
	Hence, in combination with the estimates of the other terms we have shown that $\frac{\operatorname{Re}(\mu(\eta))}{\Theta(\eta)}$ converges for $\eta \to 0$ and that
	\begin{equation}
		\begin{split}
	 	\lim_{\eta \to 0} \frac{\operatorname{Re}(\mu(\eta))}{\Theta(\eta)} =& \operatorname{Re}(\cC_k(0)) \int_{\R^3} (u \cdot \sigma) \operatorname{Im} (\Phi(u)) |u|^{k-\alpha -3} \, \dd u \\
		&- \operatorname{Im}(\cC_k(0)) \int_{\R^3} (u \cdot \sigma) \operatorname{Re} (\Phi(u)) |u|^{k-\alpha -3} \, \dd u\,.
		\end{split}
	\end{equation}
	
	In particular, taking into account that $k=2$ for $\phi_{\eta, 0,\pm}$ and $k=1$ for $\phi_{\eta,t}$, $t \in \{1,2\}$ (see Remark \ref{r:power})), we proved the first part of Proposition \ref{p:eigenvalues}, see also Figures \ref{fig:evsdiff} and \ref{fig:evsfrac}:
	
	\begin{lemma}
		The eigenvalue $\mu(\eta)_0$ corresponding to $\phi_{\eta,0}$ and the real part of $\mu(\eta)_{\pm}$ corresponding to $\phi_{\eta,\pm}$ scale as
		\begin{align*}
			\mu(\eta)_0, \re\left(\mu(\eta)_{\pm}\right) \sim \Theta(\eta) := \begin{cases} \eta^2 \quad &\text{if } \alpha > 6+\beta\,,\\   \eta^{\frac{-4+\alpha+\beta}{1+\beta}} \quad &\text{if } \alpha < 6+\beta\,,\end{cases}
		\end{align*}
		while the eigenvalues $\mu(\eta)_t$ corresponding to $\phi_{\eta,t}$ scale as
		\begin{align*}
			\mu(\eta)_t \sim \tilde{\Theta}(\eta) := \begin{cases} \eta^2 \quad &\text{if } \alpha > 4+\beta\,,\\   \eta^{\frac{-2+\alpha+\beta}{1+\beta}} \quad &\text{if } \alpha < 4+\beta\,.\end{cases}
		\end{align*}
	\end{lemma}
	
	\vspace{0.3cm}
	
\subsection{The imaginary parts of the eigenvalues}\label{ss:Im}

We know from Proposition \ref{p:evslong_frac} and Proposition \ref{p:evstrans_frac} that for the imaginary part of the eigenvalues scale as $\im\left(\mu(\eta)_\pm\right) \sim \eta$ as $\eta \to 0$, while $\im(\mu(\eta)_l) =0$, $l \in \{0,1,2\}$. In the following we aim to determine the limit $\bar{\mu}_\pm := \lim_{\eta \to 0} \frac{\im(\mu(\eta)_\pm)}{\eta}$ in the different parameter-regimes.

\subsubsection{$\operatorname{Im}(\mu(\eta))$ in the regime $6+\beta < \alpha$}\label{sss:Imdiff}
	
	Knowing the relation \eqref{Im_limit} we can calculate
	\begin{align*}
		 &\left|\frac{\operatorname{Im}(\mu(\eta))}{\eta} + \int_{\R^3} (v \cdot \sigma) \left[ \operatorname{Re}\left(\phi_0\right)^2 + \operatorname{Im}\left(\phi_0\right)^2\right]\cM \, \dd v\right| \\
		=&\left| \int_{\R^3} \left(v\cdot \sigma \right) \left[\operatorname{Re}\left(\phi_0\right)^2 - \re\left(\phi_\eta\right) \re\left(\cP(\phi_\eta)\right)\right] + \left[\operatorname{Im}\left(\phi_0\right)^2 -\im\left(\phi_\eta\right) \im\left(\cP(\phi_\eta)\right)\right] \cM \, \dd v \right| \\
		=&\left|\int_{\R^3} (v \cdot \sigma) \left[\operatorname{Re}\left(\cP(\phi_\eta)\right)\left(\operatorname{Re}\left(\cP(\phi_\eta)\right) -\operatorname{Re}(\phi_\eta) \right) + \left(\operatorname{Re}\left(\phi_0\right)^2 - \operatorname{Re}\left(\cP(\phi_\eta)\right)^2 \right) \right]\cM \, \dd v \right| \\
		&+\left|\int_{\R^3} (v \cdot \sigma) \left[\operatorname{Im}\left(\cP(\phi_\eta)\right)\left(\operatorname{Im}\left(\cP(\phi_\eta)\right) -\operatorname{Im}(\phi_\eta) \right) + \left(\operatorname{Im}\left(\phi_0\right)^2 - \operatorname{Im}\left(\cP(\phi_\eta)\right)^2 \right) \right]\cM \, \dd v \right| \\
		\leq & \left|\int_{\R^3} (v \cdot \sigma)^2 \operatorname{Re}\left(\cP(\phi_\eta)\right)^2 \wv^{\beta} \cM \, \dd v \right|^{\frac{1}{2}} \left\| \operatorname{Re}\left(\phi_\eta\right)-\operatorname{Re}\left(\cP(\phi_\eta)\right) \right\|_{-\beta} \\
		&+ \int_{\R^3} (v \cdot \sigma) \left|\operatorname{Re}\left(\phi_0\right)^2 - \operatorname{Re}\left(\cP(\phi_\eta)\right)^2 \right| \cM \, \dd v \\
		&+ \left|\int_{\R^3} (v \cdot \sigma)^2 \operatorname{Im}\left(\cP(\phi_\eta)\right)^2 \wv^{\beta} \cM \, \dd v \right|^{\frac{1}{2}} \left\| \operatorname{Im}\left(\phi_\eta\right)-\operatorname{Im}\left(\cP(\phi_\eta)\right) \right\|_{-\beta} \\
		&+ \int_{\R^3} (v \cdot \sigma) \left|\operatorname{Im}\left(\phi_0\right)^2 - \operatorname{Im}\left(\cP(\phi_\eta)\right)^2 \right| \cM \, \dd v \\
		\lesssim& \left\| \phi_\eta - \cP(\phi_\eta) \right\|_{-\beta} + \int_{\R^3} (v \cdot \sigma) \left|\phi_0^2 - \cP(\phi_\eta)^2 \right| \cM \, \dd v \\ 
		\lesssim& \left(\operatorname{Re}(\mu(\eta))^{\frac{1}{2}}+ \int_{\R^3} (v \cdot \sigma) \left|\phi_0^2 - \cP(\phi_\eta)^2 \right| \cM \, \dd v\right) \xrightarrow{\eta \to 0} 0\,,
	\end{align*}
	where after applying the Cauchy-Schwartz inequality we used that $\alpha>6+\beta$, which makes the the corresponding integrals finite. Moreover, the convergence to zero was deduced from \eqref{decay}. This provides both the convergence rate of $\operatorname{Im}(\mu(\eta))$ as $\eta \to 0$ and the limit
	\begin{equation}\label{eq:Immubar}
	\begin{split}
		\operatorname{Im}(\bar{\mu}):=& \lim_{\eta \to 0} \frac{\operatorname{Im}(\mu(\eta))}{\eta} = - \int_{\R^3} (v \cdot \sigma) \left[ \operatorname{Re}\left(\phi_0\right)^2 + \operatorname{Im}\left(\phi_0\right)^2\right]\cM \, \dd v\,.
	\end{split}
	\end{equation}
	
\subsubsection{$\operatorname{Im}(\mu(\eta))$ in the regime $6+\beta > \alpha$}\label{sss:Imfrac} 
	
Let $0 \leq \chi_R \leq 1$ be a smooth test-function such that $\chi_R \equiv 1$ on $B(0,R)$ and $\chi_R \equiv 0$ out of $B(0,2R)$. Integration of \eqref{Re} against $\eta^{-1} \operatorname{Im}\left(\cP(\phi_\eta)\right) \chi_R(\cdot \eta^{\frac{1}{1+\beta}})$ yields
	\begin{align}\label{ImRechi}
		&-\frac{\operatorname{Im}(\mu(\eta))}{\eta} \|\im\left(\cP(\phi_\eta)\right)\|_{-\beta}^2+ \frac{1}{\eta}\left \langle L^*\left( \operatorname{Re}(\phi_{\eta}) - \operatorname{Re}\left(\cP(\phi_\eta)\right) \right),\, \operatorname{Im}\left(\cP(\phi_\eta)\right)  \chi_R(\cdot \eta^{\frac{1}{1+\beta}})\right\rangle \notag \\
		=& \int_{\R^3} (v \cdot \sigma) \operatorname{Im}(\phi_{\eta})  \operatorname{Im}\left(\cP(\phi_\eta)\right)  \chi_R(\cdot \eta^{\frac{1}{1+\beta}}) \cM \dd v  \notag \\
		&-\frac{\operatorname{Re}(\mu(\eta))}{\eta}\left \langle \operatorname{Re}(\phi_{\eta})  \wv^{-\beta},\, \operatorname{Im}\left(\cP(\phi_\eta)\right) \chi_R(\cdot \eta^{\frac{1}{1+\beta}}) \right\rangle \\
		&+  \frac{\operatorname{Im}(\mu(\eta))}{\eta} \left \langle \operatorname{Im}(\phi_{\eta})  \wv^{-\beta},\, \operatorname{Im}\left(\cP(\phi_\eta)\right) \left( \chi_R(\cdot \eta^{\frac{1}{1+\beta}}) -1 \right) \right\rangle \notag
	\end{align}
	and integration of \eqref{Im} against $\eta^{-1} \operatorname{Re}\left(\cP(\phi_\eta)\right) \chi_R(\cdot \eta^{\frac{1}{1+\beta}})$ yields
	\begin{align}\label{ImImchi}
		&-\frac{\operatorname{Im}(\mu(\eta))}{\eta} \|\re\left(\cP(\phi_\eta)\right)\|_{-\beta}^2 - \frac{1}{\eta}\left \langle L^*\left( \operatorname{Im}(\phi_{\eta}) - \operatorname{Im}\left(\cP(\phi_\eta)\right) \right), \operatorname{Re}\left(\cP(\phi_\eta)\right)  \chi_R(\cdot \eta^{\frac{1}{1+\beta}})\right\rangle \notag \\
		= & \int_{\R^3} (v \cdot \sigma) \operatorname{Re}(\phi_{\eta})  \operatorname{Re}\left(\cP(\phi_\eta)\right)  \chi_R(\cdot \eta^{\frac{1}{1+\beta}}) \cM \dd v  \notag \\
		&+ \frac{\operatorname{Re}(\mu(\eta))}{\eta}\left \langle \operatorname{Im}(\phi_{\eta})  \wv^{-\beta}, \operatorname{Re}\left(\cP(\phi_\eta)\right) \chi_R(\cdot \eta^{\frac{1}{1+\beta}}) \right\rangle \\
		&+\frac{\operatorname{Im}(\mu(\eta))}{\eta} \left \langle \operatorname{Re}(\phi_{\eta})  \wv^{-\beta}, \operatorname{Re}\left(\cP(\phi_\eta)\right)\left(\chi_R(\cdot \eta^{\frac{1}{1+\beta}})-1\right) \right\rangle\,. \notag
	\end{align}
	Also here, adding \eqref{ImImchi} and \eqref{ImRechi} yields information about $\frac{\operatorname{Im}(\mu(\eta))}{\eta}$ as $\eta \to 0$. We start investigating this asymptotic behaviour with estimating the terms in the first rows of the above equations
	\begin{align}
		 &\frac{1}{\eta}\left \langle L^*\left( \operatorname{Re}(\phi_{\eta}) - \operatorname{Re}\left(\cP(\phi_\eta)\right) \right), \operatorname{Im}\left(\cP(\phi_\eta)\right)  \chi_R(\cdot \eta^{\frac{1}{1+\beta}})\right\rangle \notag\\
		 &- \frac{1}{\eta}\left \langle L^*\left( \operatorname{Im}(\phi_{\eta}) - \operatorname{Im}\left(\cP(\phi_\eta)\right) \right), \operatorname{Re}\left(\cP(\phi_\eta)\right)  \chi_R(\cdot \eta^{\frac{1}{1+\beta}})\right\rangle \notag \\
		  \leq& \frac{1}{\eta}\left\|\operatorname{Re}(\phi_{\eta}) - \operatorname{Re}\left(\cP(\phi_\eta)\right)\right\|_{-\beta} \left\|L\left(\operatorname{Im}\left(\cP(\phi_\eta)\right)  \chi_R(\cdot \eta^{\frac{1}{1+\beta}})\right)\right\|_{\beta}\\
		 &+ \frac{1}{\eta}\left\|\operatorname{Im}(\phi_{\eta}) - \operatorname{Im}\left(\cP(\phi_\eta)\right)\right\|_{-\beta} \left\|L\left(\operatorname{Re}\left(\cP(\phi_\eta)\right)  \chi_R(\cdot \eta^{\frac{1}{1+\beta}})\right)\right\|_{\beta}  \notag \\
		 \lesssim& \frac{\operatorname{Re}(\mu(\eta))^{\frac{1}{2}}}{\eta}R^{k-\frac{\alpha+\beta}{2}} \eta^{-\frac{k}{1+\beta}+\frac{\alpha+\beta}{2(1+\beta)}} \sim \eta^{\frac{-4+\alpha+\beta}{2(1+\beta)}} \eta^{\frac{-2(k+1)-\beta+\alpha}{2(1+\beta)}} R^{k-\frac{\alpha+\beta}{2}} \\
		 =& \eta^{\frac{-3-k+\alpha}{1+\beta}} R^{k-\frac{\alpha+\beta}{2}} \longrightarrow 0\,, \quad \text{as } \eta \to 0\,, \notag 
	\end{align}
	where we used that $\operatorname{Re}(\mu(\eta))^{\frac{1}{2}} \sim  \eta^{\frac{-4+\alpha+\beta}{2(1+\beta)}}$ and that $-3-k+\alpha \geq -5+\alpha >0$ and thus $\eta^{\frac{-3-k+\alpha}{1+\beta}}$ converges to zero as $\eta \to 0$, since we assume $\alpha>5$. Next, we estimate the sum of the second terms of the right-hand-side and obtain, again already knowing $\operatorname{Re}(\mu(\eta)) \sim \eta^{\frac{-2k+\alpha+\beta}{1+\beta}}$,
	\begin{align}
		 &-\frac{\operatorname{Re}(\mu(\eta))}{\eta}\left \langle \operatorname{Re}(\phi_{\eta})  \wv^{-\beta}, \operatorname{Im}\left(\cP(\phi_\eta)\right) \chi_R(\cdot \eta^{\frac{1}{1+\beta}}) \right\rangle \notag\\
		&+ \frac{\operatorname{Re}(\mu(\eta))}{\eta}\left \langle \operatorname{Im}(\phi_{\eta})  \wv^{-\beta}, \operatorname{Re}\left(\cP(\phi_\eta)\right) \chi_R(\cdot \eta^{\frac{1}{1+\beta}}) \right\rangle \\
		\lesssim& R^{2k-\alpha-\beta} \eta^{\frac{-2k+\alpha+\beta}{1+\beta}} \longrightarrow 0\,, \quad \text{as } \eta \to 0 \,. \notag
	\end{align}
	For the last terms on the right-hand-side we obtain in a similar manner
	\begin{align}
		&- \frac{\operatorname{Im}(\mu(\eta))}{\eta} \left \langle \operatorname{Im}(\phi_{\eta})  \wv^{-\beta},\, \operatorname{Im}\left(\cP(\phi_\eta)\right) \left( \chi_R(\cdot \eta^{\frac{1}{1+\beta}}) -1 \right) \right\rangle \notag \\
		&+ \frac{\operatorname{Im}(\mu(\eta))}{\eta}\left \langle \operatorname{Re}(\phi_{\eta})  \wv^{-\beta}, \, \operatorname{Re}\left(\cP(\phi_\eta)\right)\left(\chi_R(\cdot \eta^{\frac{1}{1+\beta}})-1\right) \right\rangle \\
		&\lesssim R^{2k-\alpha-\beta} \eta^{\frac{-2k+\alpha+\beta}{1+\beta}}\longrightarrow 0\,, \quad \text{as } \eta \to 0\,. \notag
	\end{align}
	This in combination with the estimates of the other terms yields
	\begin{align}\label{Imlim}
		&\left|\frac{\im(\mu(\eta))}{\eta} + \int_{\R^3} (v \cdot \sigma) \left[ \operatorname{Re}(\phi_{\eta})  \operatorname{Re}\left(\cP(\phi_\eta)\right)  + \operatorname{Im}(\phi_{\eta})  \operatorname{Im}\left(\cP(\phi_\eta)\right) \right] \chi_R(v \eta^{\frac{1}{1+\beta}}) \cM \dd v \right| \\
		&\lesssim  R^{2k-\alpha-\beta} \eta^{\frac{-2k+\alpha+\beta}{1+\beta}}  +  \eta^{\frac{-3-k+\alpha}{1+\beta}} R^{k-\frac{\alpha+\beta}{2}} \longrightarrow 0\,, \quad \text{as } \eta \to 0\,. \notag
	\end{align}
	Last, we investigate the behaviour of the integrals in \eqref{Imlim} as $\eta \to 0$ and calculate
	\begin{align}
		&\int_{\R^3} (v \cdot \sigma) \left[ \operatorname{Re}(\phi_{\eta})  \operatorname{Re}\left(\cP(\phi_\eta)\right)  + \operatorname{Im}(\phi_{\eta})  \operatorname{Im}\left(\cP(\phi_\eta)\right) \right] \chi_R(v \eta^{\frac{1}{1+\beta}}) \cM \dd v  \notag \\
		=&\int_{\R^3} (v \cdot \sigma) \left[ \operatorname{Re}\left(\cP(\phi_\eta)\right)^2 +\operatorname{Im}\left(\cP(\phi_\eta)\right)^2\right] \chi_R(v \eta^{\frac{1}{1+\beta}}) \cM \dd v  \\
		&+ \int_{\R^3} (v \cdot \sigma) \left[\operatorname{Re}\left(\cP^{\perp}(\phi_\eta)\right)  \operatorname{Re}\left(\cP(\phi_\eta)\right)  +  \operatorname{Im}\left(\cP^{\perp}(\phi_\eta)\right)  \operatorname{Im}\left(\cP(\phi_\eta)\right)\right] \chi_R(v \eta^{\frac{1}{1+\beta}}) \cM \dd v\,,  \notag
	\end{align}
	where the terms
	\begin{align*}
		\int_{\R^3} (v \cdot \sigma) \left[ \operatorname{Re}\left(\cP(\phi_\eta)\right)^2 +\operatorname{Im}\left(\cP(\phi_\eta)\right)^2\right] \chi_R(v \eta^{\frac{1}{1+\beta}}) \cM \dd v
	\end{align*}
	are well defined for all $\eta \geq 0$ and converge uniformly to
	\begin{align}
		&\int_{\R^3} (v \cdot \sigma) \left[\operatorname{Re}\left(\phi_0\right)^2 + \operatorname{Im}\left(\phi_0\right)^2\right] \cM \dd v  \,,
	\end{align}
	as $\eta \to 0$. For the other two terms we perform the coordinate change $u = \eta^{\frac{1}{1+\beta}}$ and obtain
	\begin{align*}
		 &\int_{\R^3} (v \cdot \sigma) \left[ \operatorname{Re}\left(\cP^{\perp}(\phi_\eta)\right)  \operatorname{Re}\left(\cP(\phi_\eta)\right) + \operatorname{Im}\left(\cP^{\perp}(\phi_\eta)\right)  \operatorname{Im}\left(\cP(\phi_\eta)\right)\right] \chi_R(v \eta^{\frac{1}{1+\beta}}) \cM \dd v \\
		=& \eta^{\frac{-2k-1+\alpha}{1+\beta}} \int_{\R^3} (u \cdot \sigma) \operatorname{Re}\left(\cP^{\perp}(\Phi_\eta)\right) \operatorname{Re}\left(\cC(\eta)\right)\cdot\cE_{\eta,k}(u) \chi_R(u) \cM_{\eta} \dd u \\
		&+ \eta^{\frac{-2k-1+\alpha}{1+\beta}} \int_{\R^3} (u \cdot \sigma) \operatorname{Im}\left(\cP^{\perp}(\Phi_\eta)\right) \operatorname{Im}\left(\cC(\eta)\right)\cdot\cE_{\eta,k}(u) \chi_R(u) \cM_{\eta} \dd u \\
		&\longrightarrow 0\,, \quad \text{as } \eta \to 0\,,
	\end{align*}
	where $\Phi_{\eta}$ and $\cE_{\eta,k}(u)$ are defined as in \eqref{d:Phi} and \eqref{d:moments_eta} and the convergence to zero can be deduced from the observation that under the boundedness conditions in Hypothesis \ref{hyp:fluidm}, for a given $R>0$ the integrand is integrable uniformly in $\eta$ for any $R>1$ fixed, from which convergence to zero of the above integrals for $\eta \to 0$ follows. Thus, the double limit $R \to \infty, \eta \to 0$ exists and we have analogous to \eqref{eq:Immubar}
	 \begin{equation*}
		\begin{split}
	 	\lim_{\eta \to 0} \frac{\operatorname{Im}(\mu(\eta))}{\eta} =&-  \int_{\R^3} (v \cdot \sigma) \left[\operatorname{Re}\left(\phi_0\right)^2 + \operatorname{Im}\left(\phi_0\right)^2\right] \cM \dd v \,.
		\end{split}
	\end{equation*}

This concludes the second part of Proposition \ref{p:eigenvalues}, see also Figures \ref{fig:evsdiff} and \ref{fig:evsfrac}:
	
	\begin{lemma}
		Independent of the parameter regime, the imaginary parts of $\mu(\eta)_{\pm}$ corresponding to $\phi_{\eta,\pm}$ scale as
		\begin{align*}
			\im(\mu(\eta)_\pm) \sim \eta\,,
		\end{align*}
		and we have the following formular for the limit 
		\begin{align*}
			\lim_{\eta \to 0} \frac{\im(\mu(\eta)_\pm)}{\eta} = -\int_{\R^3} (v \cdot \sigma) \left[\operatorname{Re}\left(\phi_{0,\pm}\right)^2 + \operatorname{Im}\left(\phi_{0,\pm}\right)^2\right] \cM \dd v\,.
		\end{align*}
	\end{lemma}

\section{Convergence}\label{s:convergence}

We remind the equation \eqref{eq:kineticeps}
\begin{equation}
	\gamma(\ve) \pa_t h_\ve + \ve v \cdot \nabla_x h_\ve  = L h_\ve\,,
\end{equation}
where we want obtain information about the macroscopic moments $\rho_{h_\ve}$, $m_{h_\ve}$ and $\theta_{h_\ve}$ in the limit $\ve \to 0$. 

\subsection{The energy estimate}\label{ss:energyest}

Multiplication of the above equation \eqref{eq:kineticeps} by $h_\ve\cM$, integrate in $(x,v,s) \in \R^3\times\R^2\times[0,t)$ and taking the real part yields
\begin{align}
	\gamma(\ve) \int (\pa_s h_\ve) h_\ve \cM \, \dd x \, \dd v \dd s = \int  \operatorname{Re} \left \langle L h_\ve,\, h_\ve \right \rangle \, \dd x \dd v \dd s \,,
\end{align}
which by Hypothesis \ref{hyp:coercivity} is equivalent to 
\begin{align}
	\frac{\gamma(\ve)}{2} \|h_\ve(t)\|^2_{L^2_{x,v}(\cM)} = \frac{\gamma(\ve)}{2} \|h_\ve(0)\|^2_{L^2_{x,v}(\cM)} - \lambda \int_0^t \|h_\ve - \cP h_\ve\|_{-\beta}^2 \, \dd s\,,
\end{align}
with the ususal definition of the projection $\cP h_\ve = \rho_{h_\ve,\wv^{-\beta}} + v \cdot m_{h_\ve,\wv^{-\beta}} + \frac{|v|^2-3}{2} \theta_{h_\ve,\wv^{-\beta}}$. Thus, we can deduce
\begin{equation}\label{est:energy}
	\begin{split}
	 &\|h_\ve(t)\|^2_{L_{x,v}(\cM)} \leq  \|h_\ve(0)\|^2_{L_{x,v}(\cM)}\,, \quad \forall t \geq 0\,,\\
	 &\text{and} \\
	 &\int_0^t \|h_\ve - \cP h_\ve\|_{-\beta}^2 \, \dd s \leq \frac{\gamma(\ve)}{2\lambda} \|h_\ve(0)\|^2_{L_{x,v}(\cM)}\,.
	\end{split}
\end{equation}

\subsection{Framework of the calculations}\label{ss:fourier}

We perform a Fourier transform in the spatial variable in the rescaled equation \eqref{eq:hkineticeps}
\begin{equation}\label{eq:hkineticepsFT}
		\gamma(\ve)\pa_t \hat{h}_\ve =- i\ve \left(v \cdot \xi\right) \hat{h}_\ve +L \hat{h}_\ve =- i\eta \left(v \cdot \sigma\right) \hat{h}_\ve +L \hat{h}_\ve =: L_\eta \hat{h}_\ve \,,
\end{equation}
where $\hat{h}_\ve(t,\xi,v)$ and $\xi$ describes the spatial Fourier variable and we denoted $\xi=\sigma \left|\xi\right|$, $\eta:=\ve\left|\xi\right|$. The connection with the in \eqref{evp_weight} defined operator $L_\eta$, which's spectrum we investigated extensively in Sections \ref{s:fluidmodes} \& \ref{s:scaling} now comes to light. We further remind the decomposition
\begin{equation}\label{decomphve}
\begin{split}
	&\hat{h}_\ve = \cP \hat{h}_\ve + \cP^{\perp}\hat{h}_\ve\,, \\
	&\cP \hat{h}_\ve := \theta_{\hat{h}_\ve} \frac{|v|^2 -3}{2} + m_{\hat{h}_\ve} \cdot v + \rho_{\hat{h}_\ve}\,, \\
	&\theta_{\hat{h}_\ve}:=\int_{\R^d} \hat{h}_\ve \frac{|v|^2 -3}{d} \wv^{-\beta} \cM \, \dd v \,, \quad m_{\hat{h}_\ve}:=\int_{\R^d} \hat{h}_\ve v \wv^{-\beta} \cM \, \dd v\, \quad \rho_{\hat{h}_\ve}:=\int_{\R^d} \hat{h}_\ve \wv^{-\beta} \cM \, \dd v\,. 
\end{split}
\end{equation}
Moreover, we will use the following orthogonal decomposition of the momentum-part
\begin{align*}
	m_{\hat{h}_\ve} = \sigma\left(\sigma \cdot m_{\hat{h}_\ve}\right) + \sum_{t=1}^{d-1} \overrightarrow{C}^t(0) \left( \overrightarrow{C}^t(0) \cdot m_{\hat{h}_\ve}\right)\,,\quad \overrightarrow{C}^t(0):=\lim_{\eta \to 0} \overrightarrow{C}^t(\eta)\,,
\end{align*}
where, as explained in Remark \ref{r:c}, the vectors $\left\{\left(0,\,\overrightarrow{C}^t(\eta)\,,0\right)\right\} \in \cC^5$, $t \in \{1,2\}$ are the solution vectors of the system \eqref{algsyst} corresponding to the transversal wave-eigenvalues $\left\{\mu(\eta)_t\right\}$, $t \in \{1,2\}$. Following Remark \ref{r:c} further holds
\begin{align*}
	\overrightarrow{C}^{1}(\eta) \perp \overrightarrow{C}^{2}(\eta) \,, \,\, \overrightarrow{C}^{t} \perp \sigma\,, \,\, \forall t \in \{1,2\}\,, \quad \text{and} \quad \left\| \overrightarrow{C}^{t}(\eta)\right\| =1\,.
\end{align*}
Due to the \emph{Plancherel theorem} the above established energy estimate \eqref{est:energy} implies 
\begin{align*}
	&\hat{h}_\ve \in L_t^{\infty}\left([0,\infty),\, L^2_{\xi,v}\left(\cM\right)\right) \quad \text{and} \quad \left\|\hat{h}_\ve - \cP\hat{h}_\ve\right\|_{L_t^2\left([0,\infty),\, L^2_{\xi,v}\left(\wv^{-\beta}\cM\right)\right)} \leq \frac{\gamma(\ve)}{2\lambda}^\frac{1}{2} \left\|\hat{h}_\ve(0) \right\|_{ L^2_{\xi,v}\left(\cM\right)}\,.
\end{align*}

\subsection{Derivation of the macroscopic dynamics}\label{ss:macro}

We start by deriving constraints on the macroscopic moments, which will lead to the \emph{incompressibility condition} and the \emph{Boussinesq relation} in the limit.
Integration of \eqref{eq:hkineticepsFT} against $ \wv^{-\beta} \cM$ yields
\begin{align}\label{c:incomp}
	\pa_t  \rho_{\hat{h}_\ve} = -i \frac{\ve}{\gamma(\ve)} \xi  \cdot m_{\hat{h}_\ve} =-i \frac{\ve|\xi|}{\gamma(\ve)} \sigma \cdot m_{\hat{h}_\ve} \,,
\end{align}
which due to boundedness of the left-hand-side implies $\sigma \cdot m_{\hat{h}_\ve} = \mathcal{O}\left(\frac{\gamma(\ve)}{\ve}\right)$ as $\ve \to 0$. By integration of \eqref{eq:hkineticepsFT} against $v\wv^{-\beta} \cM$ we obtain
\begin{align}\label{c:Bouss}
	\pa_t  m_{\hat{h}_\ve} =& -i \frac{\ve}{\gamma(\ve)} \xi \cdot \int_{\R^3} \left(v \otimes v\right) \hat{h}_\ve \wv^{-\beta} \cM \, \dd v \\
	=& -i \frac{\ve}{\gamma(\ve)} \xi \cdot \int_{\R^3} |v|^2 \hat{h}_\ve \wv^{-\beta} \cM \, \dd v + i \frac{\ve}{\gamma(\ve)} \xi \cdot \int_{\R^3}\left(|v|^2I_3 - \left(v \otimes v\right) \right)\hat{h}_\ve \wv^{-\beta} \cM \, \dd v  \notag\\
	=& -i \frac{\ve}{\gamma(\ve)} \xi  \int_{\R^3} |v|^2 \left(\rho_{\hat{h}_\ve} +\frac{|v|^2-3}{3}\theta_{\hat{h}_\ve}\right) \wv^{-\beta} \cM \, \dd v \notag \\
	&+ i \frac{\ve}{\gamma(\ve)} \xi \cdot \int_{\R^3}\left(|v|^2I_3 - \left(v \otimes v\right) \right)\hat{h}_\ve \wv^{-\beta} \cM \, \dd v \notag \\
	=& -i 3 \frac{\ve}{\gamma(\ve)} \xi \left(\rho_{\hat{h}_\ve}+\theta_{\hat{h}_\ve}\right) + i \frac{\ve}{\gamma(\ve)} \xi \cdot \int_{\R^3}\left(|v|^2I_3 - \left(v \otimes v\right) \right)\hat{h}_\ve \wv^{-\beta} \cM \, \dd v \,.\notag 
\end{align}
Due to boundedness of the left-hand-side we have that also the right-hand side is bounded, thus, as $\ve \to 0$, $|\xi|\left(\rho_{\hat{h}_\ve}+\theta_{\hat{h}_\ve}\right)$ is at least of order $\mathcal{O}\left(\frac{\gamma(\ve)}{\ve}\right)$, which also holds for the second integral term. While the first will result in the \emph{Boussinesq relation}, giving a pressure gradient in the limit, the second gives rise to the dynamics of the first moment in the limit. They will be, as can be seen below, either diffusive or constant depending on the parameter regime.

For the following computation the leading order of the fluid modes will play an important role, hence we remind \eqref{phi00}-\eqref{phi02},
\begin{align}\label{fluid modes0}
				&\phi_{0,0} =  -\sqrt{\frac{2}{5}} + \frac{|v|^2 -3}{\sqrt{10}}\,, \quad \quad &\mu(\eta)_0 \in \R\,, \notag  \\
				&\phi_{0,t}= \overrightarrow{C}_t(0) \cdot v\,, \quad \quad &\mu(\eta)_t \in \R\,,  \\
				&\phi_{0,\pm} = \sqrt{\frac{3}{10} } \pm\frac{1}{\sqrt{2}}\left(v \cdot \sigma\right) + \frac{|v|^2 -3}{\sqrt{30}} \,,  \quad \quad  &\overline{\mu(\eta)_+} = \mu(\eta)_-\,, \notag
\end{align}
as well as the decomposition
\begin{align*}
	\cP \hat{h}_\ve :=& \theta_{\hat{h}_\ve} \frac{|v|^2 -3}{2} + m_{\hat{h}_\ve} \cdot v + \rho_{\hat{h}_\ve} \\
	=&\theta_{\hat{h}_\ve} \frac{|v|^2 -3}{2} +\left(v\cdot\sigma\right)\left(\sigma \cdot m_{\hat{h}_\ve}\right) + \sum_{t=1}^{d-1}\left(v\cdot \overrightarrow{C}^t(0)\right) \left( \overrightarrow{C}^t(0) \cdot m_{\hat{h}_\ve}\right) + \rho_{\hat{h}_\ve}\,.
\end{align*}
We test \eqref{eq:hkineticepsFT} against each $\phi_{\eta,l} \cM$, $l=\{0,\pm,1,2\}$ and obtain the following 5 -3imensional system 
\begin{align}\label{syst:mac}
	M_1(\eta) \begin{pmatrix} \pa_t \rho_{\hat{h}_\ve} \\ \pa_t \theta_{\hat{h}_\ve} \\\pa_t (\sigma \cdot m_{\hat{h}_\ve})   \\ \pa_t \left( \overrightarrow{C}^1(0) \cdot m_{\hat{h}_\ve}\right) \\ \pa_t \left( \overrightarrow{C}^2(0) \cdot m_{\hat{h}_\ve}\right)\end{pmatrix} + \pa_t E_1 = -\gamma(\eps)^{-1} M_2(\eta) \begin{pmatrix} \rho_{\hat{h}_\ve} \\  \theta_{\hat{h}_\ve} \\ (\sigma \cdot m_{\hat{h}_\ve})  \\ \left( \overrightarrow{C}^1(0) \cdot m_{\hat{h}_\ve}\right) \\  \left( \overrightarrow{C}^2(0) \cdot m_{\hat{h}_\ve}\right)\end{pmatrix} + \gamma(\eps)^{-1} E_2\,.
\end{align}
First we mention that the error-terms are given by
\begin{align*}
	E_1(\eta):= \begin{pmatrix} 
	\left\langle \cP^{\perp} \hat{h}_\ve ,\, \phi_{\eta,0}\right\rangle \\ 
	\left\langle \cP^{\perp} \hat{h}_\ve ,\, \phi_{\eta,+}\right\rangle \\ 
	\left\langle \cP^{\perp} \hat{h}_\ve ,\, \phi_{\eta,-} \right\rangle \\ 
	\left\langle \cP^{\perp} \hat{h}_\ve ,\, \phi_{\eta,1}\right\rangle  \\ 
	\left\langle \cP^{\perp} \hat{h}_\ve ,\, \phi_{\eta,2}\right\rangle 
	\end{pmatrix} \,, \quad \text{and} \quad
	E_2(\eta):= \begin{pmatrix}
	 \mu(\eta)_0\left\langle \cP^{\perp} \hat{h}_\ve ,\, \phi_{\eta,0}\right\rangle_{-\beta} \\ 
	 \overline{\mu(\eta)_+}\left\langle \cP^{\perp} \hat{h}_\ve ,\, \phi_{\eta,+}\right\rangle_{-\beta} \\ 
	 \mu(\eta)_+ \left\langle \cP^{\perp} \hat{h}_\ve ,\, \phi_{\eta,-}\right\rangle_{-\beta} \\
	 \mu(\eta)_1 \left\langle \cP^{\perp} \hat{h}_\ve ,\, \phi_{\eta,1}\right\rangle_{-\beta} \\
	 \mu(\eta)_2 \left\langle \cP^{\perp} \hat{h}_\ve ,\, \phi_{\eta,2}\right\rangle_{-\beta} 
	 \end{pmatrix}\,.
\end{align*}
With the significant scaling choice 
$$
	\gamma(\ve) = \Theta(\ve) = \begin{cases} \ve^2 \quad &\text{if } \beta + 6 < \alpha\,, \\ \ve^{\frac{-4+\alpha+\beta}{1+\beta}} \quad &\text{if } \beta + 6 >\alpha \,, \end{cases}
$$
the errors $\pa_t E_1$ and $\gamma(\ve)^{-1}E_2$ converge to zero as $\ve \to 0$, which can be shown using estimates \eqref{decay}  and \eqref{est:energy}, which shows that the part of the fluid-modes and of $\hat{h}_\ve$ which is orthogonal to the projection vanishes with order $\gamma(\ve)^{\frac{1}{2}}$.

The matrix $M_1(\eta)$ on the left-hand-side is given by
\begin{align*}
	&M_1(\eta):= \begin{pmatrix} 
	\langle 1,\, \phi_{\eta,0} \rangle & \left\langle \frac{|v|^2 -3}{2},\, \phi_{\eta,0} \right\rangle & \langle (\sigma \cdot v),\, \phi_{\eta,0} \rangle & \left\langle \left( \overrightarrow{C}^1(0) \cdot v\right),\, \phi_{\eta,0} \right\rangle  & \left\langle \left( \overrightarrow{C}^2(0) \cdot v\right),\, \phi_{\eta,0} \right\rangle \\  
	\langle 1,\, \phi_{\eta,+} \rangle  &  \left\langle \frac{|v|^2 -3}{2},\, \phi_{\eta,+} \right\rangle  & \langle (\sigma \cdot v),\, \phi_{\eta,+} \rangle & \left\langle \left( \overrightarrow{C}^1(0) \cdot v\right),\, \phi_{\eta,+} \right\rangle  & \left\langle \left( \overrightarrow{C}^2(0) \cdot v\right),\, \phi_{\eta,+} \right\rangle \\  
	\langle 1,\, \phi_{\eta,-} \rangle  &\left\langle \frac{|v|^2 -3}{2},\, \phi_{\eta,-} \right\rangle & \langle (\sigma \cdot v),\, \phi_{\eta,-} \rangle & \left\langle \left( \overrightarrow{C}^1(0) \cdot v\right),\, \phi_{\eta,-} \right\rangle  & \left\langle \left( \overrightarrow{C}^2(0) \cdot v\right),\, \phi_{\eta,-} \right\rangle\\ 
	\langle 1,\, \phi_{\eta,1} \rangle  &\left\langle \frac{|v|^2 -3}{2},\, \phi_{\eta,1} \right\rangle  & \langle (\sigma \cdot v),\, \phi_{\eta,1} \rangle& \left\langle \left( \overrightarrow{C}^1(0) \cdot v\right),\, \phi_{\eta,1} \right\rangle  & \left\langle \left( \overrightarrow{C}^2(0) \cdot v\right),\, \phi_{\eta,1} \right\rangle \\ 
	\langle 1,\, \phi_{\eta,2} \rangle  &\left\langle \frac{|v|^2 -3}{2},\, \phi_{\eta,2} \right\rangle & \langle (\sigma \cdot v),\, \phi_{\eta,2} \rangle  & \left\langle \left( \overrightarrow{C}^1(0) \cdot v\right),\, \phi_{\eta,2} \right\rangle  & \left\langle \left( \overrightarrow{C}^2(0) \cdot v\right),\, \phi_{\eta,2} \right\rangle
	\end{pmatrix}\,,
\end{align*}
which at leading order is of the form
\begin{align*}
	&M_1(0)= \begin{pmatrix}
	 \langle 1,\, \phi_{0,0} \rangle &  \left\langle \frac{|v|^2 -3}{2},\, \phi_{0,0} \right\rangle  & 0  & 0 &0 \\  
	 \langle 1,\, \phi_{0,+} \rangle &  \left\langle \frac{|v|^2 -3}{2},\, \phi_{0,+} \right\rangle & \langle v,\, \phi_{0,+} \rangle  & 0 &0 \\  
	 \langle 1,\, \phi_{0,+} \rangle & \left\langle \frac{|v|^2 -3}{2},\, \phi_{0,+} \right\rangle & -\langle v,\, \phi_{0,+} \rangle & 0 & 0  \\ 
	 0 & 0 & 0 & 1 & 0 \\
	  0 & 0 & 0& 0 &1 
	 \end{pmatrix}\,,
\end{align*}
since $ \left\langle \left( \overrightarrow{C}^t(0) \cdot v\right),\, \phi_{0,t} \right\rangle=1$ due to our normalisation conditions. Although we know the exact values of the entries of $M_1(0)$, it suffices to say that the its inverse is of the following form
\begin{align*}
	M_1(0)^{-1} = \begin{pmatrix} a_1 & a_2 & a_2 & 0 & 0\\ b_1 & b_2 & b_2& 0 & 0 \\ 0 & c_1 & -c_1 & 0 & 0 \\  0 & 0 & 0 &1 &0 \\  0 & 0 & 0 & 0 &1 \end{pmatrix}\,, \quad a_1, a_2, b_1, b_2, c_1>0\,,
\end{align*}
where $a_{1,2}$ equals to $-b_{1,2}$ up to some multiplicative constant. Also the matrix $M_2(\eta)$ on the right-hand-side has block-structure, which is due to the orthogonality of the transversal fluid modes $\phi_{\eta,t}$ to all the other fluid modes $\phi_{\eta,0,\pm}$:
\begin{align*}
	&M_2(\eta) := -\begin{pmatrix}
		&B(\eta) & 0 \\
		&0 & D(\eta)
	\end{pmatrix}\,,
\end{align*}
with $B(\eta) \in \C^3 \times \C^3$ being given as
\begin{align*}
	B(\eta):= &\begin{pmatrix} 
	\mu(\eta)_0\langle 1,\, \phi_{\eta,0} \rangle_{-\beta} & \mu(\eta)_0\left\langle \frac{|v|^2 -3}{2},\, \phi_{\eta,0} \right\rangle_{-\beta}  & \mu(\eta)_0\langle \left(\sigma \cdot v \right),\, \phi_{\eta,0} \rangle_{-\beta} \\  
	\overline{\mu(\eta)_+}\langle 1,\, \phi_{\eta,+} \rangle_{-\beta} & \overline{\mu(\eta)_+}\left\langle \frac{|v|^2 -3}{2},\, \phi_{\eta,+} \right\rangle_{-\beta}  & \overline{\mu(\eta)_+}\langle  \left(\sigma \cdot v \right),\, \phi_{\eta,+} \rangle_{-\beta}\\  
	\overline{\mu(\eta)_-}\langle 1,\, \phi_{\eta,-} \rangle_{-\beta} & \overline{\mu(\eta)_-}\left\langle \frac{|v|^2 -3}{2},\, \phi_{\eta,-} \right\rangle_{-\beta} & \overline{\mu(\eta)_-}\langle  \left(\sigma \cdot v \right) ,\, \phi_{\eta,-} \rangle_{-\beta}
	\end{pmatrix}
\end{align*}
and $D(\eta) \in \C^2 \times \C^2$ corresponds to the contribution of the transversal modes and has the following diagonal shape
\begin{align*}
	D(\eta) := &\begin{pmatrix}
	&\mu(\eta)_1\left\langle \left(\overrightarrow{C}^1(\eta) \cdot v\right),\, \phi_{\eta,1} \right\rangle_{-\beta} &0  \\  
	& 0& \mu(\eta)_2\left\langle \left(\overrightarrow{C}^2(\eta) \cdot v\right),\, \phi_{\eta,2} \right\rangle_{-\beta}  
	\end{pmatrix}\,.
\end{align*}
Further exploiting symmetry properties of the wave-fluid modes $\phi_{\eta,\pm}$ as well as of their corresponding eigenvalues \eqref{c:symm} we are able to rewrite
\begin{align*}
	B(\eta)=&\begin{pmatrix} 
	\mu(\eta)_0\langle 1,\, \phi_{\eta,0} \rangle_{-\beta}  & \mu(\eta)_0\left\langle \frac{|v|^2 -3}{2},\, \phi_{\eta,0} \right\rangle_{-\beta} & \mu(\eta)_0\langle \left(\sigma \cdot v \right),\, \phi_{\eta,0} \rangle_{-\beta} \\  
	\overline{\mu(\eta)_+}\langle 1,\, \phi_{\eta,+} \rangle_{-\beta}  &  \overline{\mu(\eta)_+}\left\langle \frac{|v|^2 -3}{2},\, \phi_{\eta,+} \right\rangle_{-\beta} & \overline{\mu(\eta)_+}\langle \left(\sigma \cdot v \right),\, \phi_{\eta,+} \rangle_{-\beta} \\  
	\mu(\eta)_+\langle 1,\, \overline{\phi_{\eta,+}}  \rangle_{-\beta} & \mu(\eta)_+\left\langle \frac{|v|^2 -3}{2},\, \overline{\phi_{\eta,+}} \right\rangle_{-\beta} & -\mu(\eta)_+\langle \left(\sigma \cdot v \right),\, \overline{\phi_{\eta,+}} \rangle_{-\beta} 
	\end{pmatrix} \\
	=&\begin{pmatrix} 
	\mu(\eta)_0\langle 1,\, \phi_{\eta,0} \rangle_{-\beta}  & \mu(\eta)_0\left\langle \frac{|v|^2 -3}{2},\, \phi_{\eta,0} \right\rangle_{-\beta} & \mu(\eta)_0\langle \left(\sigma \cdot v \right),\, \phi_{\eta,0} \rangle_{-\beta} \\  
	\re(\mu(\eta)_+)\langle 1,\, \phi_{\eta,+} \rangle_{-\beta}  & \re(\mu(\eta)_+)\left\langle \frac{|v|^2 -3}{2},\, \phi_{\eta,+} \right\rangle_{-\beta} & \re(\mu(\eta)_+) \langle \left(\sigma \cdot v \right),\, \phi_{\eta,+} \rangle_{-\beta} \\  
	\re(\mu(\eta)_+) \langle 1,\, \overline{\phi_{\eta,+}}  \rangle_{-\beta} & \re(\mu(\eta)_+) \left\langle \frac{|v|^2 -3}{2},\, \overline{\phi_{\eta,+}} \right\rangle_{-\beta} & -\re(\mu(\eta)_+)\langle \left(\sigma \cdot v \right),\, \overline{\phi_{\eta,+}} \rangle_{-\beta} 
	\end{pmatrix} \\
	&-i\im(\mu(\eta)_+) \begin{pmatrix} 
	0&0&0 \\  
	\langle 1,\, \phi_{\eta,+} \rangle_{-\beta}  & \left\langle \frac{|v|^2 -3}{2},\, \phi_{\eta,+} \right\rangle_{-\beta} & \langle \left(\sigma \cdot v \right),\, \phi_{\eta,+} \rangle_{-\beta} \\  
	 \langle 1,\, \overline{\phi_{\eta,+}}  \rangle_{-\beta} &  \left\langle \frac{|v|^2 -3}{2},\, \overline{\phi_{\eta,+}} \right\rangle_{-\beta} & - \langle \left(\sigma \cdot v \right),\, \overline{\phi_{\eta,+}} \rangle_{-\beta} 
	\end{pmatrix}\,.
\end{align*}
With the definitions
\begin{align*}
	&\Theta(\eta) := \begin{cases} \eta^2 \quad &\text{if }\alpha>6+\beta\,,\\  \eta^{\frac{-4+\alpha+\beta}{1+\beta}} \quad &\text{if }\alpha <6+\beta\,, \end{cases} \,\quad &\tilde{\Theta}(\eta) := \begin{cases} \eta^2 \quad &\text{if }\alpha>4+\beta\,,\\  \eta^{\frac{-2+\alpha+\beta}{1+\beta}} \quad &\text{if }\alpha <4+\beta\,, \end{cases}
\end{align*}
used in Proposition \ref{p:eigenvalues}, we can extract the different scalings of $B(\eta)$ and $D(\eta)$ in the following way
\begin{align*}
	B(\eta)=\Theta(\eta)&\begin{pmatrix} 
	\tilde{\mu}_0 \langle 1,\, \phi_{\eta,0} \rangle_{-\beta}  &\tilde{\mu}_0\left\langle \frac{|v|^2 -3}{2},\, \phi_{\eta,0} \right\rangle_{-\beta} & \tilde{\mu}_0\langle \left(\sigma \cdot v \right),\, \phi_{\eta,0} \rangle_{-\beta} \\  
	\re(\tilde{\mu}_+)\langle 1,\, \phi_{\eta,+} \rangle_{-\beta}  & \re(\tilde{\mu}_+)\left\langle \frac{|v|^2 -3}{2},\, \phi_{\eta,+} \right\rangle_{-\beta} & \re(\tilde{\mu}_+) \langle \left(\sigma \cdot v \right),\, \phi_{\eta,+} \rangle_{-\beta} \\  
	\re(\tilde{\mu}_+) \langle 1,\, \overline{\phi_{\eta,+}}  \rangle_{-\beta} & \re(\tilde{\mu}_+) \left\langle \frac{|v|^2 -3}{2},\, \overline{\phi_{\eta,+}} \right\rangle_{-\beta} & -\re(\tilde{\mu}_+)\langle \left(\sigma \cdot v \right),\, \overline{\phi_{\eta,+}} \rangle_{-\beta} 
	\end{pmatrix} \\
	&+i\eta \im(\bar{\mu}_+) \begin{pmatrix} 
	0&0&0 \\  
	\langle 1,\, \phi_{\eta,+} \rangle_{-\beta}  & \left\langle \frac{|v|^2 -3}{2},\, \phi_{\eta,+} \right\rangle_{-\beta} & \langle \left(\sigma \cdot v \right),\, \phi_{\eta,+} \rangle_{-\beta} \\  
	 \langle 1,\, \overline{\phi_{\eta,+}}  \rangle_{-\beta} &  \left\langle \frac{|v|^2 -3}{2},\, \overline{\phi_{\eta,+}} \right\rangle_{-\beta} & - \langle \left(\sigma \cdot v \right),\, \overline{\phi_{\eta,+}} \rangle_{-\beta} 
	\end{pmatrix}\,,
\end{align*}
and 
\begin{align*}
D(\eta)=&\tilde{\Theta}(\eta)
\begin{pmatrix}
	&\tilde{\mu}_1\left\langle \left(\overrightarrow{C}^1(\eta) \cdot v\right),\, \phi_{\eta,1} \right\rangle_{-\beta} &0  \\  
	& 0& \tilde{\mu}_2\left\langle \left(\overrightarrow{C}^2(\eta) \cdot v\right),\, \phi_{\eta,2} \right\rangle_{-\beta}  
	 \end{pmatrix} \,,
\end{align*}
where the positive constants of order one are given by 
\begin{align*}
	&\tilde{\mu}_0:=\lim_{\eta \to 0} \frac{\mu(\eta)_0}{\Theta(\eta)}\,, \quad  \re(\tilde{\mu}_+):=\lim_{\eta \to 0} \frac{\re(\mu(\eta)_+)}{\Theta(\eta)}\,, \\
	&\im(\bar{\mu}_+) := \lim_{\eta \to 0} \frac{\im(\mu(\eta)_+)}{\eta}\,, \quad \tilde{\mu}_t:=\lim_{\eta \to 0} \frac{\mu(\eta)_t}{\tilde{\Theta}(\eta)}\,, \quad t \in \{1,2\}\,,
\end{align*}
which were calculated in each parameterregime in Section \ref{s:scaling}. Exploiting the explicite form of the fluid modes at $\eta=0$ allows us to determine the sub-matrices at leading order 
\begin{align*}
	B(\eta)=\Theta(\eta)&\begin{pmatrix} 
	-\tilde{\mu}_0\sqrt{\frac{5}{2}} & \tilde{\mu}_0 \frac{3}{\sqrt{10}}&0  \\  
	\re(\tilde{\mu}_+)\sqrt{\frac{3}{10}} & \re(\tilde{\mu}_+) \sqrt{\frac{3}{10}}&\frac{\re(\tilde{\mu}_+)}{\sqrt{2}} \\  
	\re(\tilde{\mu}_+)\sqrt{\frac{3}{10}} & \re(\tilde{\mu}_+) \sqrt{\frac{3}{10}} &-\frac{\re(\tilde{\mu}_+)}{\sqrt{2}} \\  
	\end{pmatrix} -i\eta \im(\bar{\mu}_+) \begin{pmatrix} 
	0&0&0 \\  
	\sqrt{\frac{3}{10}} &\sqrt{\frac{3}{10}} & \frac{1}{\sqrt{2}}  \\
	-\sqrt{\frac{3}{10}} & -\sqrt{\frac{3}{10}} & \frac{1}{\sqrt{2}}
	\end{pmatrix}\,,
\end{align*}
and 
\begin{align*}
	D(\eta)=\tilde{\Theta}(\eta)\begin{pmatrix} 
	\tilde{\mu}_1 & 0 \\  
	0 & \tilde{\mu}_2 
	\end{pmatrix}\,.
\end{align*}
Multiplication of the system \eqref{syst:mac} by $\gamma(\ve)^{-1}M_1(0)^{-1}$ yields at lowest order
\begin{align}\label{systat0}
	\begin{pmatrix} \pa_t \rho_{\hat{h}_\ve} \\ \pa_t \theta_{\hat{h}_\ve} \\\pa_t (\sigma \cdot m_{\hat{h}_\ve})   \\ \pa_t \left( \overrightarrow{C}^1(0) \cdot m_{\hat{h}_\ve}\right) \\ \pa_t \left( \overrightarrow{C}^2(0) \cdot m_{\hat{h}_\ve}\right)\end{pmatrix} = -\gamma(\eps)^{-1} M_1(0)^{-1}M_2(\eta) \begin{pmatrix} \rho_{\hat{h}_\ve} \\  \theta_{\hat{h}_\ve} \\ (\sigma \cdot m_{\hat{h}_\ve})  \\ \left( \overrightarrow{C}^1(0) \cdot m_{\hat{h}_\ve}\right) \\  \left( \overrightarrow{C}^2(0) \cdot m_{\hat{h}_\ve}\right)\end{pmatrix} \,,
\end{align}
where (remembering that $\eta=|\xi|\ve$)
\begin{align*}
	\gamma(\eps)^{-1} M_1(0)^{-1}M_2(\eta) = -\begin{pmatrix}
		&\bar{B}(\ve) & 0 \\
		&0 &\bar{D}(\ve)
	\end{pmatrix}\,,
\end{align*}
with
\begin{align*}
	\bar{B}(\ve):=& -|\xi|^{\zeta} \begin{pmatrix} 
	-a_1\tilde{\mu_0} \sqrt{\frac{5}{3}} + a_2 \re(\tilde{\mu_+})2\sqrt{\frac{3}{10}} & a_1 \tilde{\mu_0}\frac{3}{\sqrt{10}}+a_2\re(\tilde{\mu_+})2\sqrt{\frac{3}{10}} & 0 \\ 
	 -b_1\tilde{\mu_0} \sqrt{\frac{5}{3}} + b_2 \re(\tilde{\mu_+})2\sqrt{\frac{3}{10}} & b_1 \tilde{\mu_0}\frac{3}{\sqrt{10}}+b_2\re(\tilde{\mu_+})2\sqrt{\frac{3}{10}} & 0 \\ 
	 0 & 0 & c_1 \re(\tilde{\mu_+})\sqrt{2} 
	\end{pmatrix} \\
	&-i |\xi| \frac{\ve}{\gamma(\ve)}  \im(\bar{\mu_+}) \begin{pmatrix}
		0 & 0 & \sqrt{2}a_2 \\
		0 & 0 & \sqrt{2}b_2 \\
		2\sqrt{\frac{3}{10}}c_1 & 2\sqrt{\frac{3}{10}} c_1 &0
	\end{pmatrix}\,,
\end{align*}
where we want to remind that $a_{1,2}$ is equal to $b_{1,2}$ up to multiplicative constant. Moreover, we have
\begin{align*}
	\bar{D}(\ve):= -|\xi|^{\tilde{\zeta}}\ve^{\zeta-\tilde{\zeta}} \begin{pmatrix} 
	1 & 0 \\ 
	 0 &1
	\end{pmatrix}
\end{align*}
where we defined
\begin{equation}\label{d:zeta}
\begin{split}
	&\zeta:= \begin{cases}2 \quad &\text{if } \alpha > 6+\beta\,, \\ \frac{-4+\alpha+\beta}{1+\beta} \quad &\text{if } \alpha < 6+\beta\,,\end{cases} \\
	&\tilde{\zeta}:= \begin{cases}2 \quad &\text{if } \alpha > 4+\beta\,, \\ \frac{-2+\alpha+\beta}{1+\beta} \quad &\text{if } \alpha < 6+\beta\,. \end{cases}
\end{split}
\end{equation}
In total, we obtain in the limit $\ve \to 0$ from the system \eqref{systat0}, denoting $(\hat{\rho},\hat{m},\hat{\theta}):=\lim_{\ve \to 0}(\rho_{\hat{h}_\ve}, m_{\hat{h}_\ve}, \theta_{\hat{h}_\ve})$:
\begin{itemize}
\item For the \emph{mass} and \emph{energy} of the system
\begin{align*}
	&\pa_t  \hat{\rho} = - k_{1,1} |\xi|^{\zeta}  \hat{\rho} - k_ {1,2} |\xi|^{\zeta}  \hat{\theta} -i |\xi| a_2\sqrt{2} \im(\bar{\mu_+}) \lim_{\ve \to 0}\left(\frac{\ve}{\gamma(\ve)} (\sigma \cdot m_{\hat{h}_\ve} )\right)\,,\\
	&\pa_t  \hat{\theta} = - k_{2,1} |\xi|^{\zeta}  \hat{\rho} - k_ {2,2} |\xi|^{\zeta}  \hat{\theta} -i |\xi| b_2\sqrt{2} \im(\bar{\mu_+}) \lim_{\ve \to 0}\left(\frac{\ve}{\gamma(\ve)} (\sigma \cdot m_{\hat{h}_\ve} )\right)\,,
\end{align*}
where we want to point out that the limit $\lim_{\ve \to 0}\left(\frac{\ve}{\gamma(\ve)}(\sigma \cdot m_{\hat{h}_\ve} )\right)$ is of order 1 due to the incompressibility condition \eqref{c:incomp}. We further observe that since $a_2=-b_2$ (up to a multiplicative constant) we have also have $k_{1,1}=-k_{1,2}$  and $k_{1,2}=-k_{2,2}$. Remembering the Boussinesq - relation \eqref{c:Bouss} we obtain after taking a proper linear combinations of the above two equations
\begin{align*}
	&\pa_t  \hat{\rho} = -i |\xi|\sqrt{2} \im(\bar{\mu_+}) \lim_{\ve \to 0}\left(\frac{\ve}{\gamma(\ve)} (\sigma \cdot m_{\hat{h}_\ve} )\right)\,,\\
	&\pa_t  \hat{\theta} = \kappa |\xi|^{\zeta}  \hat{\theta}\,.
\end{align*}
After inverse Fouriertransformation the first equation yields the incompressibility condition in the limit, while the second equation is of the form
\begin{equation*}
\begin{split}
	&\pa_t  \theta= \Delta \theta\,, \quad \text{if } \alpha > 6 +\beta \,, \\
	&\pa_t \theta = \Delta^{\frac{\zeta}{2}} \theta\,, \quad \text{if } \alpha < 6 +\beta \,,
\end{split}
\end{equation*}
where the fraction $\zeta$ is given by the fraction \eqref{d:zeta}.
\item For the \emph{momentum} we obtain
	\begin{align*}
		\pa_t (\sigma \cdot \hat{m})  = -|\xi|^{\zeta} c_1 \re(\tilde{\mu_+})\sqrt{2} \lim_{\ve \to 0} (\sigma \cdot m_{\hat{h}_\ve} ) +i\im(\bar{\mu_+}) c_1 2\sqrt{\frac{3}{10}} \lim_{\ve \to 0}  \frac{\ve}{\gamma(\ve)} |\xi|\left(\rho_{\hat{h}_\ve}  + \theta_{\hat{h}_\ve} \right)\,,
	\end{align*}
	where the first term is 0 due to the incompressibility condition \eqref{c:incomp}, and the second term is of order 1 due to the Boussinesq- condition \eqref{c:Bouss}, which will result in a pressure gradient in the limit. For the directions orthogonal to $\sigma$ we have each
	\begin{align*}
		\pa_t \left( \overrightarrow{C}^t(0) \cdot \hat{m}\right) =  \begin{cases} -|\xi|^2 \left( \overrightarrow{C}^t(0) \cdot \hat{m}\right)\quad &\text{if } \alpha >6 +\beta\,,  \\ 0  \quad &\text{if } \alpha < 6 +\beta\,,\end{cases} \quad t \in \{1,2\}\,,
	\end{align*}
	which is due to the fact that the transversal waves scale slower in the fractional regimen, hence in the regime $\alpha<6+\beta$ the $\ve^{\zeta-\tilde{\zeta}} \longrightarrow 0$, as $\ve \to 0$. In total we obtain for the momentum after inverse Fouriertransform
	\begin{align*}
		&\pa_t m= \nabla \cdot (A \nabla m) + \nabla p\, \quad &\text{if } \alpha > 6 +\beta \,, \\
		&\pa_tm =  \nabla p \, \quad &\text{if } \alpha < 6 +\beta \,,
	\end{align*}
	where the diffusion matrix $A$ can be computed from the two directions $\overrightarrow{C}^t(0)$, $t \in \{1,2\}$ and the pressure term occurs in the limit $\im(\bar{\mu_+}) c_12\sqrt{\frac{3}{10}} \hat{p}:=\lim_{\ve \to 0}\frac{\ve}{\gamma(\ve)}\left(\rho_{\hat{h}_\ve}  + \theta_{\hat{h}_\ve} \right)$.
\end{itemize}

%\section{Proof for the BGK-type operator}

%\section{Proof for the Fokker-Planck type operator} Idea: linearise the non-linear Fokker-Planck operator with a fat tail equilibrium. This operator does not have the ``Ellis-Pinsky'' structure.

%\section{Proof for the 4 waves kinetic equation} To be calculated for the next version of the paper.

%%%%%%%%%%%%%%%%%%%%%%%%%%%%%%%%%%%%%%%%%%%%%%%%%%%%%%%%%%%%%%%%%%
\bibliographystyle{abbrv}
\bibliography{biblio}

\vspace{0.3cm}
\end{document}